\documentclass{amsart}
\usepackage{amsmath,amssymb,amsthm,amsfonts}
\usepackage{paralist}
\usepackage{booktabs}
\usepackage[usenames,dvipsnames]{color}
\usepackage{comment}
\usepackage{graphicx,subfigure}
\usepackage{latexsym}
\usepackage[latin1]{inputenc}
\usepackage[shortlabels]{enumitem}
\usepackage{thmtools}
\usepackage{url}
\usepackage[all,cmtip]{xy}
\usepackage{hyperref}
\usepackage{tikz}
\tikzset{font={\fontsize{4pt}{12}\selectfont}}
\usepackage{hyperref} 

\usepackage{epstopdf}
\usepackage{epsfig, psfrag, multicol}

\usepackage{ytableau}

\newcommand{\hookweight}{hook weight }

\newcommand{\excise}[1]{}

\newtheorem{thm}{Theorem}[section]
\newtheorem{lemma}[thm]{Lemma}

\newtheorem{cor}[thm]{Corollary}
\newtheorem{prop}[thm]{Proposition}
\newtheorem{conj}[thm]{Conjecture}
\newtheorem{conv}[thm]{Convention}
\newtheorem{question}[thm]{Question}

\numberwithin{equation}{section} 

\theoremstyle{definition}
\newtheorem{ex}[thm]{Example}
\newtheorem{rem}[thm]{Remark}
\newtheorem{Warn}[thm]{Caution}

    {\end{Example}}
    {\end{Remark}}

\newcommand{\ssf}{\mathsf{f}}
\newcommand{\ssp}{\mathsf{p}}
\newcommand{\ssq}{\mathsf{q}}
\newcommand{\Alt}{\mathsf{Alt}}

\def\wh{\widehat}
\def\bu{\bullet}

\def\zz{\mathbb Z}
\def\nn{\mathbb N}

\def\rr{\mathbb R}

\def\SS{\mathbb S}

\def\ov{\overline}

\def\Ga{\Gamma}

\def\la{\lambda}
\def\ga{\gamma}
\def\si{\sigma}
\def\de{\delta}
\def\ep{\epsilon}
\def\al{\alpha}
\def\be{\beta}
\def\om{\omega}

\def\cT{\mathcal T}

\def\ssu{\subset}

\def\<{\langle}
\def\>{\rangle}

\def\y{ {\text {\rm y}  } }

\def\rP{ {\text {\rm P} } }

\def\Ups{{\small {\Upsilon}}}

\def\0{{\mathbf 0}}

\def\.{\hskip.06cm}
\def\ts{\hskip.03cm}

\def\bx{{\textbf{x}}}
\def\by{{\textbf{y}}}

\def\Lam{\gimel}
\def\LA{\Lambda}
\def\LAM{\Lambda^\triangledown}

\def\di{{\small{\ts\diamond\ts}}}

\def\Pr{{\text{\rm Prob}}}
\def\OMabc{\cT_{abc}}

\newcommand{\hf}[1]{\widehat{#1}}

\newcommand{\ED}{\mathcal{E}}
\newcommand{\NIP}{\mathcal{NIP}}

\newcommand{\PP}{\operatorname{PP}}
\newcommand{\Hex}{\operatorname{H}}
\newcommand{\RPP}{\operatorname{RPP}}

\newcommand{\SYT}{\operatorname{{\rm SYT}}}
\def\HT{{{\mathbf{\tts hl}}}}

\DeclareMathOperator{\bwt}{wt}

\DeclareMathOperator{\rw}{rw}

\DeclareMathOperator{\skewsh}{skewsh}

\def\.{\hskip.06cm}
\def\ts{\hskip.03cm}

\def\nin{\noindent}
\def\tts{\hskip.02cm}

\begin{document}

\title[Hook formulas for skew shapes~III]{Hook formulas for skew shapes~III. Multivariate \\ and product formulas}

\author[Alejandro Morales, Igor Pak, Greta Panova]{Alejandro H.~Morales$^\star$,
\ \ Igor Pak$^\dagger$, \ and \ \ Greta Panova$^\ddagger$}

\thanks{\today}
\thanks{\thinspace ${\hspace{-.45ex}}^\star$Department of Mathematics
  and Statistics,
UMass Amhest, Amherst, MA~01003.
Email:
\texttt{ahmorales@math.umass.edu}}

\thanks{\thinspace ${\hspace{-.45ex}}^\dagger$Department of Mathematics,
UCLA, Los Angeles, CA~90095.
\hskip.06cm
Email:
\texttt{pak@math.ucla.edu}}

\thanks{\thinspace ${\hspace{-.45ex}}^\ddagger$
Department of Mathematics,
USC, 
Los Angeles, CA, 90089.
\hskip.06cm
Email:
\texttt{gpanova@usc.edu}}

\begin{abstract}
We give new product formulas for the number of standard Young tableaux
of certain skew shapes and for the principal evaluation of certain Schubert polynomials.  These are proved by utilizing symmetries
for evaluations of factorial Schur functions, extensively studied in
the first two papers in the series~\cite{MPP1,MPP2}.  We also apply
our technology to obtain determinantal and product formulas for the
partition function of certain weighted lozenge tilings, and give
various probabilistic and asymptotic applications.
\end{abstract}

\maketitle

\section{Introduction}

\subsection{Foreword}
It is a truth universally acknowledged, that a combinatorial theory
is often judged not by its intrinsic beauty but by the examples and
applications.  Fair or not, this attitude is historically grounded
and generally accepted.  While eternally challenging, this helps
to keep the area lively, widely accessible, and growing in
unexpected directions.

\smallskip

There are two notable types of examples and applications one can think of: \ts
\emph{artistic} \ts and \ts \emph{scientific} \ts (cf.~\cite{Gow}).  The former are
unexpected results in the area which are both beautiful and
mysterious.  The fact of their discovery is the main application,
even if they can be later shown by a more direct argument.
The latter are results which represent
a definitive progress in the area, unattainable by other means.
To paraphrase Struik, this is ``something to take home'', rather
than to simply admire (see~\cite{Rota}).  While the line is often
blurred, examples of both types are highly desirable, with
the best examples being both artistic and scientific.

\smallskip

This paper is a third in a series and continues our study of the
\emph{Naruse hook-length formula}~\eqref{eq:Naruse}, its generalizations and
applications.  In the first paper~\cite{MPP1}, we introduced two
$q$-analogues of the NHLF and gave their (difficult) bijective proofs.
In the second paper \cite{MPP2}, we investigated the special case
of \emph{ribbon hooks}, which were used to obtain two new elementary proofs
of NHLF in full generality, as well as various new mysterious summation and
determinant formulas.

\smallskip

\nin
In this paper we present three new families of examples and applications of our tools:

\smallskip

$\bu$ \ new product formulas for the number of \emph{standard Young tableaux}
of certain skew shapes,

$\bu$ \ new product formulas for the principal evaluation of certain \emph{Schubert polynomials},

$\bu$ \ new determinantal formulas for weighted enumeration of \emph{lozenge tilings} of a hexagon.

\smallskip
\nin
All three directions are so extensively studied from enumerative point of view,
it is hard to imagine there is room for progress.  In all three cases, we
generalize a number of existing results within the same general framework of
factorial Schur functions.  With one notable exception (see~$\S$\ref{sec:dewitt}),
we cannot imagine a direct combinatorial proof of the new product formulas circumventing
our reasoning (cf.~$\S$\ref{sec:bij}, however).
As an immediate consequence of our results,  we obtain exact asymptotic formulas which
were unreachable until now (see sections~\ref{sec:asy} and~\ref{sec:prob}).
Below we illustrate our results one by one, leaving full statements and generalizations
for later.

\subsection{Number of SYT  of skew shape} \label{intro:skewshapesprod}
\emph{Standard Young tableaux} are fundamental objects in enumerative and
algebraic combinatorics and their enumeration is central to the area
(see e.g.~\cite{SSG,St1}).  The number  \ts $f^{\lambda} = \bigl|\SYT(\lambda)\bigr|$ \ts
of standard Young tableaux of shape~$\lambda$ and size~$n$,
is given by the classical \emph{hook-length formula}:
\begin{equation}
\label{eq:HLF} \tag{HLF}
f^{\lambda} \, = \, n! \, \prod_{u\in  [\lambda]} \frac{1}{h(u)}\..
\end{equation}
Famously, there is no general product formula for the
number  \ts $f^{\lambda/\mu} = \bigl|\SYT(\lambda/\mu)\bigr|$ \ts  of standard Young tableaux of skew
shape~$\lambda/\mu$.\footnote{In fact, even for small zigzag shapes $\pi=\de_{k+2}/\de_k$,
< the number $f^\pi$ can have large prime divisors (cf.~$\S$\ref{ss:KMY-321}).} However, such formulas do exist for a few
sporadic families of skew shapes and truncated shapes (see~\cite{AR}).

In this paper we give a six-parameter family of
skew shapes $\lambda/\mu=\Lambda(a,b,c,d,e,m)$ where $\mu=b^a$ with product
formulas for the number of their SYT. The product formulas for these
shapes of size $n$ include the {\em MacMahon box formula} and hook-lengths of
certain cells of $\lambda$:
\begin{equation}
f^{\lambda/\mu} \,=\,  n!\cdot \prod_{i=1}^a\prod_{j=1}^{b}\prod_{k=1}^c \.
\frac{i+j+k-1}{i+j+k-2} \cdot \,\prod_{(i,j) \in \lambda/0^cb^{a}} \,
                             \frac{1}{h_{\lambda}(i,j)} \, ,
\end{equation}
(see Theorem~\ref{thm:skewprod} and Figure~\ref{fig:abcdem-shape} for an illustration of the skew
shape and the cells of
$\lambda$ whose hook-lengths appear above). The three
corollaries below showcase the most elegant special cases. We single out 
two especially interesting special cases: \ts Corollary~\ref{cor:abcde-shape} due to its
connection to the \emph{Selberg integral},
and Corollary~\ref{cor:abc-shape} due to its relation to shifted shapes
and a potential for a bijective proof (see $\S$\ref{sec:dewitt}).  
These two
special cases were known before, but the corresponding proofs do
not generalize to the setting of this paper.

The formulas below are written in terms of {\em superfactorials} $\Phi(n)$,
{\em double superfactorials} $\Lam(n)$,
{\em super doublefactorials} $\Psi(n)$, and {\em shifted super doublefactorials}
$\Psi(n;k)$ defined as follows:
$$
\aligned\Phi(n)\. &:=\. 1!\cdot 2! \ts \cdots \ts (n-1)!\,,  \qquad \quad \.\. \Lam(n)\. := \. (n-2)!(n-4)!\ts \cdots\., \\
\Psi(n) & := \. 1!! \cdot 3!! \ts \cdots \ts (2n-3)!!\,, \qquad \Psi(n;k)\.  :=\.  (k+1)!!\ts \cdots \ts (k+3)!! \ts \cdots \ts (k+2n-3)!!
\endaligned
$$

\smallskip

\begin{cor}[Kim--Oh~\cite{KO}, see \S\ref{sec:kimoh}] \label{cor:abcde-shape}
For all \ts $a,b,c,d,e \in \nn$, let $\lambda/\mu$ be the skew
shape in \ts {\rm Figure~\ref{fig:intro_shapes}~(i)}. Then
the number \ts $f^{\lambda/\mu} = \bigl|\SYT(\lambda/\mu)\bigr|$ \ts is equal to
\[
n! \, \. \frac{\Phi(a)\.\Phi(b)\.\Phi(c)\.\Phi(d)\.
\Phi(e)\.\Phi(a+b+c)\.\Phi(c+d+e)\.\Phi(a+b+c+d+e)}{\Phi(a+b)\.
\Phi(d+e)\.\Phi(a+c+d)\.\Phi(b+c+e)\.\Phi(a+b+2c+d+e)}\,,
\]
where  $n=|\la/\mu|=(a+c+e)(b+c+d)-ab-ed$.
\end{cor}

Note that in~\cite[Cor.~4.7]{KO}, the product formula is equivalent,
but stated differently.

\smallskip

\begin{cor}[DeWitt \cite{DeW}, see \S\ref{sec:dewitt}] \label{cor:abc-shape}
For all \ts  $a,b,c\in \nn$, let $\lambda/\mu$ be the skew shape
in \ts {\rm Figure~\ref{fig:intro_shapes}~(ii)}.
Then the number  \ts $f^{\lambda/\mu} = \bigl|\SYT(\lambda/\mu)\bigr|$ \ts  is equal to
\[
n! \,  \,  \frac{\Phi(a)\.\Phi(b)\.\Phi(c)\.
\Phi(a+b+c)\ts\cdot\ts \Psi(c)\.\Psi(a+b+c)}{\Phi(a+b)\.\Phi(b+c)
\.\Phi(a+c)\ts\cdot\ts  \Psi(a+c)\.\Psi(b+c)\.\Psi(a+b+2c)}\,,
\]
where   $n=|\la/\mu|=\binom{a+b+2c}{2}-ab$.
\end{cor}

\smallskip

\begin{figure}[hbt]
\begin{center}
\includegraphics{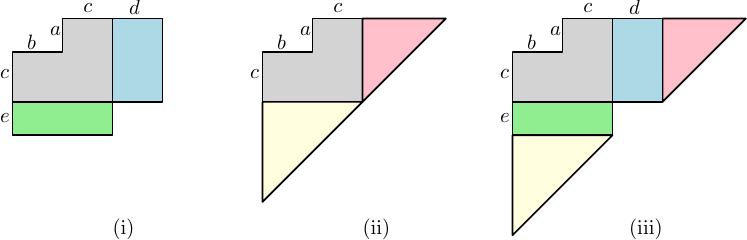}
\caption{Skew shapes $\Lambda(a,b,c,d,e,0)$, $\Lambda(a,b,c,0,0,1)$
  and $\Lambda(a,b,c,d,e,1)$ with product formulas for the number of SYT.}
\label{fig:intro_shapes}
\end{center}
\end{figure}

\begin{cor} \label{cor:abcde1-shape}
For all \ts  $a,b,c,d,e\in \nn$, let $\lambda/\mu$ be the skew
shape in \ts {\rm Figure~\ref{fig:intro_shapes}~(iii)}. Then the number
\ts $f^{\lambda/\mu} = \bigl|\SYT(\lambda/\mu)\bigr|$ \ts  is equal to
$$
 \frac{n! \.\cdot \. \Phi(a) \ts \Phi(b) \ts \Phi(c) \ts \Phi(a+b+c) \ts\cdot\ts
  \Psi(c;d+e)\ts \Psi(a+b+c;d+e) \ts\cdot \ts \Lam(d) \ts \Lam(e) \ts \Lam(2a+2c) \ts\Lam(2b+2c)}
{\Phi(a+b) \ts\Phi(b+c)  \ts\Phi(a+c)\cdot \Psi(a+c) \ts\Psi(b+c)\ts
    \Psi(a+b+2c;d+e) \cdot \Lam(2a+2c+d) \ts\Lam(2b+2c+e)}\,,
$$
where $\ts n=|\la/\mu|=(a+c+e)(b+c+d)+\binom{a+c}{2}+\binom{b+c}{2}-ab-ed$.
\end{cor}

Let us emphasize that the proofs of
corollaries~\ref{cor:abcde-shape}--\ref{cor:abcde1-shape}
are quite technical in nature.  Here is a brief non-technical explanation.  Fundamentally,
the Naruse hook-length formula (NHLF) provides a new way to understand SYT of skew shape,
coming from geometry rather than representation theory.  What we show in this paper
is that the proof of the NHLF has ``hidden symmetries'' which can be turned into
product formulas (cf.~$\S$\ref{sec:bij}).  We refer to Section~\ref{sec:skewprod} for
the complete proofs and common generalizations of these results, including $q$-analogues
of the corollaries.


\subsection{Product formulas for principal evaluations of Schubert polynomials}
The \emph{Schubert polynomials} \ts $\mathfrak{S}_w \in \zz[x_1,\ldots,x_{n-1}]$, $w\in S_n$,
are generalizations of the Schur polynomials and play a key role in the geometry of
flag varieties (see e.g.~\cite{Mac,Man}).  They can be expressed in terms of
\emph{reduced words} (factorizations) of the permutation~$w$ via
the \emph{Macdonald identity}~\eqref{eq:macdonald}, and have been an object
of intense study in the past several decades.  In this paper we obtain several new
product formulas for the \emph{principal evaluation} \ts
$\mathfrak{S}_w(1,\ldots,1)$ \ts which has been extensively studied
in recent years (see e.g.~\cite{BHY,MeS,SeS,St2,We,Woo}).  Below we present
two such formulas:

\begin{cor} [= Corollary~\ref{cor:2413-rest}]
\label{cor:2413}
For the permutation \ts $w(a,c) := 1^c \times (2413 \otimes 1^a)$, where $c\geq a$, we have:
\[
\mathfrak{S}_{w(a,c)}(1,1,\ldots,1) \, = \, \frac{\Phi(4a+c)\.\Phi(c) \.\Phi(a)^4\.\Phi(3a)^2
}{\Phi(3a+c)\.\Phi(a+c) \.\Phi(2a)^2\.\Phi(4a)}\,.
\]
\end{cor}

Here $\si \times \om$ and $\si \otimes \om$ are the \emph{direct sum} and the \emph{Kronecker product}
of permutations $\si$ and~$\om$ (see~$\S$\ref{sec:perms}). We denote by by $1^n$ the identity
permutation in~$S_n$.

\begin{cor} [= Corollary~\ref{cor:351624-main}]
\label{cor:351624}
For the permutation \ts $s(a):=351624 \otimes 1^a$, we have:
\[
\mathfrak{S}_{s(a)}(1,1,\ldots,1) \, =
 \, \frac{\Phi(a)^5\.\Phi(3a)^2\.\Phi(5a)}{\Phi(2a)^4\.\Phi(4a)^2}\,.
\]
\end{cor}

These results follow from two interrelated connections between principal evaluations of
Schubert polynomials and the number of SYT of skew shapes.  Below we give a brief outline,
which is somewhat technical (see Section~\ref{sec:not} for definitions and details).

\smallskip

The first connection is in the case of  {\em vexillary} ($2143$-avoiding) permutations.
The excited diagrams first appeared in a related context in work of
Wachs~\cite{W} and Knutson--Miller--Yong \cite{KMY}, where they gave an
explicit formula for the {\em double Schubert polynomials} of
vexillary permutations in terms of excited diagrams
(see~$\S$\ref{subsec:schubs}) of a skew shape associated to the
permutation. As a corollary, the principal evaluation gives the number
of excited diagrams of the skew shape (Theorem~\ref{cor:KMY-excited}).
Certain families of vexillary permutations have skew shapes with product formulas for the number of excited diagrams (see above).
Corollary~\ref{cor:2413} is one such example.

The second connection is in the case of $321$-avoiding permutations.
Combining the Macdonald identity and the results of
Billey--Jockusch--Stanley \cite{BJS}, we show that the principal evaluation for
$321$-avoiding permutations is a multiple of~$f^{\lambda/\mu}$ for a
skew shape associated to the permutation (Theorem~\ref{prob:skew321}).
In fact, {\em every} skew shape can
be realized via a $321$-avoiding permutation, see~\cite{BJS}.
In particular, permutations corresponding to the
skew shapes in $\S$\ref{intro:skewshapesprod} have product formulas for the principal
evaluations. Corollary~\ref{cor:351624} follows along these lines from
Corollary~\ref{cor:abcde-shape} with $a=b=c=d=e$.

\subsection{Determinantal formulas for lozenge tilings}

Lozenge tilings have been studied extensively in statistical mechanics and
integrable probability, as exactly solvable dimer models on the hexagonal grid.
When the tilings are chosen uniformly at random on a given domain
and the mesh size $\to 0$, they have exhibited remarkable
limit behavior like limit shape, frozen boundary, Gaussian Unitary
Ensemble eigenvalue distributions, Gaussian Free Field fluctuations, etc.
Such tilings are studied via a variety of
methods ranging from variational principles
(see e.g.~\cite{CLP, Kenyon-Dimers, KO-Burgers}),  to asymptotics
of Schur functions and determinantal processes (see e.g.~\cite{BGR, GP, Petrov}).

Lozenge tilings of hexagonal shapes correspond naturally to plane
partitions, when the lozenges are interpreted as sides of cubes and
the tiling is interpreted as a projection of a stack of boxes.
For example, lozenge tilings of the hexagon
$$\Hex(a,b,c) \. := \. \<\ts a\times b \times c \times a \times b \times c\ts\>
$$
are in bijection with solid partitions which fit inside the \ts $[a\times b \times c]$ \ts box.
Thus, they are counted by the MacMahon box formula, see~$\S$\ref{ss:not-PP}~:
\begin{equation} \label{eq:macmahon}
\bigl|\PP(a,b,c)\bigr| \, = \, \frac{\Phi(a+b+c)\.\Phi(a)\.\Phi(b)\.\Phi(c)}{\Phi(a+b)\.\Phi(b+c)\.\Phi(a+c)}\,.
\end{equation}
This connection allows us to translate our earlier results into
the language of weighted lozenge tilings with multivariate weights on horizontal lozenges
(Theorem~\ref{prop:EDtiling}).
As a result, we obtain a number of determinantal formulas for the weighted sums
of such lozenge tilings (see Figure~\ref{fig:excited2lozenge}:Right for weights of lozenges).
Note that a similar but different extension of~\eqref{eq:macmahon} to weighted lozenge tilings
was given by Borodin, Gorin and Rains in~\cite{BGR}; see~$\S$\ref{subsec:Racah}
for a curious common special case of both extensions.

We then obtain new probabilistic results for random locally-weighted
lozenge tilings.  Specifically, observe that every vertical boundary edge is connected
to an edge on the opposite side of the hexagon by a path~$\ssp$, which goes through
lozenges with vertical edges (see Figure~\ref{fig:lozenge_det}).  Our main application
is Theorem~\ref{thm:lozengepathpr}, which gives a determinant formula for the probability
of~$\ssp$ in the weighted lozenge tiling.

\begin{figure}[hbt]
\includegraphics[scale=0.55]{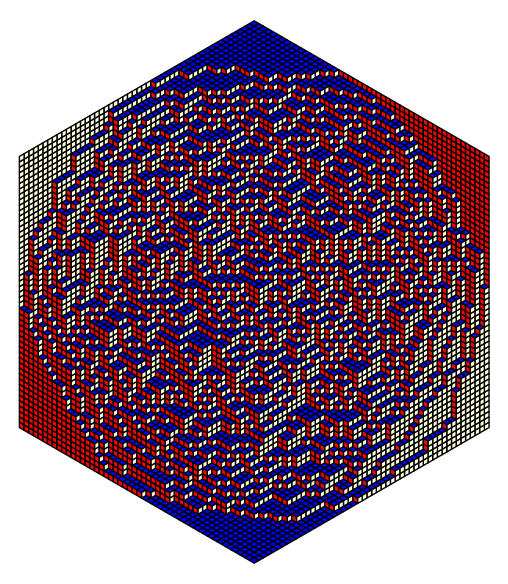} \,\,\includegraphics[scale=0.55]{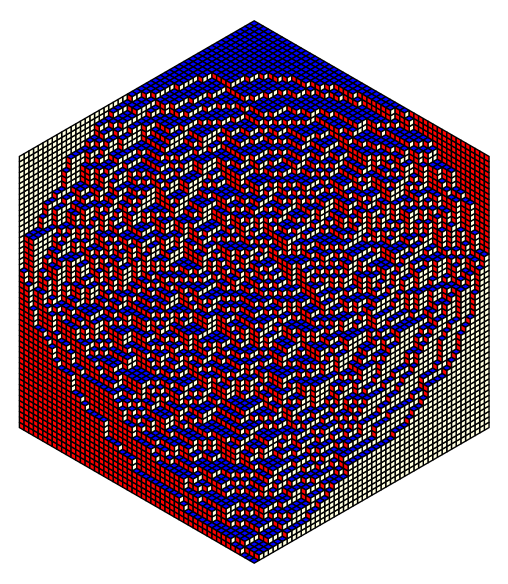}
\caption{Random tilings of hexagon \ts $\Hex(50,50,50)$ \ts with uniform and hook
weighted horizontal lozenges.}
\label{f:lozenge-hexagon}
\end{figure}

We illustrate the difference between the uniform and weighted lozenge tilings in
Figure~\ref{f:lozenge-hexagon}.  Here both tilings of the hexagon $\Hex(50,50,50)$
are obtained by running the Metropolis algorithm for \. $2\cdot 10^9$ \ts steps.\footnote{In the
uniform case, a faster algorithm to generate such random tilings is given in~\cite{BG} (see also~\cite{Bet}).}
In the latter case, the weight is defined to be a product over horizontal
lozenges of a linear function in the coordinates (see Section~\ref{sec:lozenge}).
Note that the \emph{Arctic circle} in the uniform case is replaced by a
more involved limit shape as in the figure (see also Figure~\ref{f:stair}).
In fact, the results in~$\S$\ref{ss:lozenge-explicit2} explain why
the latter limit shape is tilted upward, even if they are not strong
enough to prove its existence (see~$\S$\ref{ss:asy-lozenge}).

\subsection{Structure of the paper}
We begin with a lengthy Section~\ref{sec:not} which summarizes the notation
and gives a brief review of the earlier work.  In the next Section~\ref{sec:multi},
we develop the technology of multivariate formulas including two key identities
(Theorems~\ref{thm:thickstrip} and~\ref{thm:sympaths_rect_mu}). We use these identities
to prove the product formulas for the number $f^{\la/\mu}$ of SYT of skew shape in
Section~\ref{sec:skewprod},  including generalization of
corollaries~\ref{cor:abcde-shape}--\ref{cor:abcde1-shape}.
In Section~\ref{sec:KMY} we use our technology
to obtain product formulas for the principal evaluation of Schubert polynomials.
These results are used in Section~\ref{sec:asy} to obtain asymptotic
formulas in a number of special cases. In Section~\ref{sec:lozenge},
we obtain explicit determinantal formulas for the number of weighted
lozenge tilings, which are then interpreted probabilistically and applied
in two natural special cases in Section~\ref{sec:prob}.  We conclude with
final remarks and open problems in Section~\ref{sec:fin-rem}.

\bigskip

\section{Notation and Background} \label{sec:not}

\subsection{Young diagrams and skew shapes} \label{ss:not-yd}
Let $\lambda=(\lambda_1,\ldots,\lambda_r),
\mu=(\mu_1,\ldots,\mu_s)$ denote integer partitions of
length $\ell(\lambda)=r$ and $\ell(\mu)=s$. The {\em size} of the partition
is denoted by $|\lambda|$ and $\lambda'$
denotes the {\em conjugate partition} of $\lambda$. We use $[\lambda]$ to
denote the \emph{Young diagram} of the partition $\lambda$. The \emph{hook length}
$h_{\lambda}(i,j) = \la_i - i +\la_j' -j +1$ of a square $u=(i,j)\in [\la]$ is
the number of squares directly to the right or directly below~$u$
in~$[\la]$ including $u$.

A {\em skew shape} is denoted by $\lambda/\mu$ for partitions
$\mu\subseteq \lambda$. The \emph{staircase} shape is denoted by
$\delta_n=(n-1,n-2,\ldots,2,1)$.
Finally, a skew shape $\lambda/\mu$ is called \emph{slim} if $\la$ has $d$ parts
and $\la_d \geq \mu_1 + d-1$, see \cite[$\S$11]{MPP3}.

\subsection{Permutations} \label{sec:perms}
We write permutations of $\{1,2,\ldots,n\}$
as $w=w_1w_2\ldots w_n \in S_n$, where $w_i$ is the
image of~$i$. Given a positive integer~$c$, let $1^c \times w$ denote the
\emph{direct sum permutation}
$$1^c \times w \. := \. 1\ts 2\ts \ldots \ts c \, (c+w_1)\ts (c+w_2)\.\ldots \. (c+w_n)\..$$
Similarly, let $w\otimes 1^c$ denote the \emph{Kronecker product permutation}
of size $\ts c\ts n\ts$ whose permutation
matrix equals the Kronecker
product of the permutation matrix $P_w$ and the identity matrix $I_c$. See
Figure~\ref{fig:1stshape} for an example.

To each permutation $w\in S_n$, we associate the subset of $[n]\times [n]$
given by
$$
D(w)\. = \. \bigl\{(i,w_j) \ts \mid \ts i<j, \ts w_i>w_j\bigr\}\ts.
$$
This set is called the
(Rothe) \emph{diagram} of $w$ and can be viewed as the complement in
$[n]\times [n]$ of the
hooks from the cells $(i,w_i)$ for $i=1,2,\ldots,n$. The size of this set is the length of $w$ and
it uniquely determines $w$. Diagrams
of permutations play  in the theory of Schubert polynomials the role
that  partitions play in the theory of symmetric functions. The {\em essential set} of a
permutation $w$ is given by
$$Ess(w) \. = \. \bigl\{(i,j) \in D(w) \, \bigl| \, (i+1,j),
(i,j+1),(i+1,j+1)\not\in D(w)\bigr\}\ts.$$
See Figure~\ref{fig:permdiags}
for an example of a diagram $D(w)$ and $Ess(w)$.

\begin{figure}[hbt]
\subfigure[]{
\includegraphics[scale=0.6]{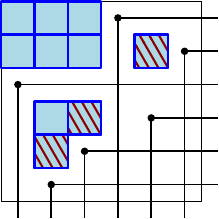}
\quad
\includegraphics[scale=0.6]{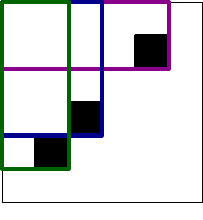}
\quad
\includegraphics[scale=0.6]{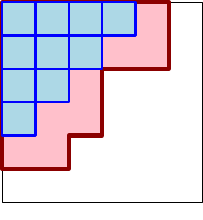}
\label{fig:permdiagsvex}
}
\quad\qquad
\subfigure[]{
\includegraphics[scale=0.6]{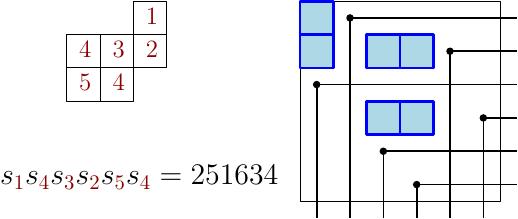}
\label{fig:permdiags321}
}
\caption{(a) The diagram of the vexillary permutation $w=461532$ (with
  cells in the essential set tiled in red). Up to permuting rows and
  columns it is the diagram of $\mu(w)=4321$; the supershape
  $\lambda(w)=55332$ defined by the essential set; the
  skew shape $\lambda(w)/\mu(w)$. (b) Example of
  correspondence between skew shapes and $321$-avoiding
  permutations for $w=251634$.}
\label{fig:permdiags}
\end{figure}

The diagrams of two families of permutations have very appealing
properties. These families are also
described using the notion of {\em pattern avoidance} of
permutations~\cite{Kit} and play an important role in Schubert
calculus. We refer to \cite[\S 2.1-2]{Man} for details and further
examples.

A permutation is {\bf vexillary}
if $D(w)$ is, up to permuting rows and columns, the Young diagram of a
partition denoted by $\mu=\mu(w)$. Equivalently, these are 2143-avoiding
permutations, i.e.\ there is no sequence $i<j<k<\ell$ such that $w_j<w_i<w_\ell
<w_k$.  Given a vexillary permutation let
$\lambda=\lambda(w)$ be the smallest partition containing the diagram
$D(w)$. This partition is also the union over the $i\times j$
rectangles with NW--SE corners $(1,1)$, $(i,j)$ for each $(i,j) \in
Ess(w)$. We call this
partition the {\bf supershape} of $w$ and note that $\mu(w)\subseteq
\lambda(w)$ (see Figure~\ref{fig:permdiagsvex}). Examples of vexillary permutations are {\bf dominant} permutations
(132-avoiding) and {\bf Grassmannian} permutations (permutations with at
most one descent).

A permutation is {\bf 321}-avoiding if there is no sequence $i<j<k$
such that $w_i>w_j>w_k$. The diagram $D(w)$ of such a permutation is, up
to removing rows and columns of the board not present
in the diagram and flipping columns, the Young diagram of a skew shape
that we denote $\skewsh(w)$. Conversely, every skew shape $\lambda/\mu$ can be
obtained from the diagram of a $321$-avoiding permutation~\cite{BJS}.

\begin{thm}[Billey--Jockusch--Stanley \cite{BJS}] \label{thm:shape2perm}
For every skew shape $\lambda/\mu$ with $(n-1)$ diagonals, there is a
$321$-avoiding permutation $w\in S_n$, such that $\skewsh(w)=\lambda/\mu$.
\end{thm}

The construction from \cite{BJS} to prove this theorem is as follows: Label the diagonals of
$\lambda/\mu$ from right to left $\sf{d}_1,\sf{d}_2,\ldots$ and label
the cells of $\lambda/\mu$ by the index of the their diagonal. Let $w$
be the permutation whose reduced word is obtained by reading the
labeled cells of the skew shape from left to right top to bottom
(see Figure~\ref{fig:permdiags321}); we denote this
reduced word by $\rw(\lambda/\mu)$. Note that $\rw(\lambda/\mu)$ is the {\em
  lexicographically minimal} among the reduced words of~$w$.


\subsection{Plane partitions}\label{ss:not-PP}

Let $\PP(a,b,c)$ and $\RPP(a,b,c)$ denote the sets of ordinary and reverse plane
partitions $\pi$, respectively, that fit into an $\ts [a\times b \times c]
\ts$ box with nonnegative entries, and let $|\pi|$ denote the sum of
entries of the plane partition.  Recall the
\emph{MacMahon box formula}~\eqref{eq:macmahon} for the number of such  (reverse)
plane partitions, which can also be written as follows:
\begin{equation} \label{eq:macmahon1}
\bigl|\PP(a,b,c)\bigr| \. = \. \bigl|\RPP(a,b,c)\bigr| \. = \, \prod_{i=1}^a\prod_{j=1}^{b}\prod_{k=1}^c \.
\frac{i+j+k-1}{i+j+k-2} \,,
\end{equation}
%
and its $q$-analogue:
\begin{equation} \label{eq:qmacmahon}
\sum_{ \pi \in \RPP(a,b,c)} q^{|\pi|} \, = \, \prod_{i=1}^a\prod_{j=1}^b\prod_{k=1}^c
\, \frac{1-q^{i+j+k-1}}{1-q^{i+j+k-2}}\,.
\end{equation}

\subsection{Factorial Schur functions} \label{sec:factorial_schurs}

The
{\em factorial Schur function} (e.g.~see \cite{MS}) is defined as
\begin{equation}\label{eq:fschur-det}
s_{\mu}^{(d)}({\bf x} \mid {\bf a}) := \frac{\det \bigl[
    (x_i-a_1)\cdots (x_i -
    a_{\mu_j+d-j})\bigr]_{i,j=1}^d}{\Delta(x_1,\ldots,x_d)}\,,
\end{equation}
where \ts ${\bf x} = x_1,\ldots,x_d$ \ts are variables, \ts ${\bf a} =a_1,a_2,\ldots$
\ts are parameters, and
\begin{equation} \label{eq:vandermonde}
\Delta(x_1,\ldots,x_d) \, = \, \Delta({\bf x}) :=
\prod_{1\leq i<j\leq d}  \. (x_i-x_j)
\end{equation}
is the Vandermonde determinant. By convention $\mu_j=0$ for
$j>\ell(\mu)$. This function has an explicit
expression in terms of semi-standard tableaux of shape~$\mu$~:
\begin{equation} \label{eq:fschur-explicit}
s_{\mu}^{(d)}({\bf x} \,\vert\, {\bf a}) \, = \, \sum_{T} \. \prod_{u \in \mu}
\. \bigl(x_{T(u)} - a_{T(u)+c(u)}\bigr)\ts,
\end{equation}
where the sum is over semistandard Young tableaux $T$ of shape $\mu$ with
entries in $\{1,\ldots,d\}$ and $c(u)=j-i$ denotes the content of the
cell $u=(i,j)$. Moreover,  \ts
$s_{\mu}^{(d)}({\bf x} \,\vert\, {\bf a})$ \ts is symmetric in $x_1,\ldots,x_d$.

\subsection{Schubert polynomials} \label{subsec:schubs}

Schubert polynomials were introduced by Lascoux and Sch\"utzen-berger
\cite{LS1} to study Schubert varieties. We denote by
$\mathfrak{S}_w({\bf x}; {\bf y})$ the double Schubert polynomial of
$w$ and by $\mathfrak{S}_w({\bf x}) = \mathfrak{S}_w({\bf x}; {\bf 0})$
the single Schubert polynomial. See \cite[\S 2.3]{Man} and
\cite[$\S$IV,VI]{Mac} for definitions and properties.

The \emph{principal evaluation} of a single Schubert polynomials at $x_i=1$, counting
the number of monomials, is given by the
following  \emph{Macdonald identity} \cite[Eq.~6.11]{Mac} (see
also~\cite[Thm.~2.5.1]{Man} and \cite{BHY} for a bijective proof):
\begin{equation} \label{eq:macdonald}
\Ups_w \, := \, \mathfrak{S}_w(1,1,\ldots,1) \, = \,
\frac{1}{\ell!} \. \sum_{(r_1,\ldots,r_{\ell}) \in R(w)} \. r_1r_2\cdots r_{\ell}\..
\end{equation}
Here $R(w)$ denotes the set of \emph{reduced words} of $w\in S_n\ts$: \ts tuples
\ts $(r_1,r_2,\ldots,r_{\ell})$ \ts  such that \ts $s_{r_1}s_{r_2}\cdots
s_{r_{\ell}}$ \ts is a reduced decomposition of $w$ into simple
transpositions $s_i = (i,i+1)$.

\subsection{Excited diagrams}

Let $\lambda/\mu$ be a skew partition and $D$ be a subset of the Young
diagram of $\lambda$. A cell $u=(i,j) \in D$ is called {\em active} if
  $(i+1,j)$, $(i,j+1)$ and $(i+1,j+1)$ are all in
$[\lambda]\setminus D$.  Let $u$ be an
active cell of $D$, define $\alpha_u(D)$ to be the set obtained by
replacing $(i,j)\in D$ by $(i+1,j+1)$. We call this procedure an {\em excited move}.
An {\em excited diagram} of $\lambda/\mu$ is a subdiagram of $\lambda$ obtained
from the Young diagram of~$\mu$ after a sequence of excited moves on active cells.
Let $\ED(\lambda/\mu)$ be the set of excited diagrams of~$\lambda/\mu$.

\begin{ex} \label{ex:excited}
The skew shape $\lambda/\mu = 332/21$ has five excited diagrams:
\begin{center}
\includegraphics{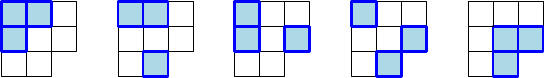}.
\end{center}
\end{ex}

\subsection{Flagged tableaux} \label{ss:excited-flagged}

Excited diagrams of $\lambda/\mu$ are
equivalent to certain {\em flagged tableaux} of shape~$\mu$
(see \cite[\S 3]{MPP1} and \cite[\S 6]{VK}): SSYT of shape $\mu$ with
bounds on the entries of each row. The number of
excited diagrams is given by a determinant, a
polynomial in the parts of $\lambda$ and~$\mu$ as follows.
Consider the diagonal that passes through cell $(i,\mu_i)$, i.e.\ the
last  cell of row $i$ in $\mu$. Let this diagonal intersect the
boundary of $\lambda$ at a row denoted by
$\ssf^{(\lambda/\mu)}_i$.  Given an excited diagram $D$ in
$\ED(\lambda/\mu)$, each cell $(x,y)$ in $[\mu]$ corresponds to a cell $(i,j)$ in $D$, let
$\varphi(D):=T$ be the tableau of shape $\mu$ with $T_{x,y} =
i$.

\begin{prop}[\cite{MPP1}] \label{prop:flagged}
The map $\varphi$ is a bijection between excited diagrams of $\lambda/\mu$
and SSYT of shape $\mu$ with entries in row $i$ at
most $\ssf^{(\lambda/\mu)}_i$. Moreover,
\[
|\ED(\lambda/\mu)| = \det\left[
  \binom{\ssf^{(\lambda/\mu)}_i+\mu_i-i+j-1}{\ssf^{(\lambda/\mu)}_i-1}\right]_{i,j=1}^{\ell(\mu)}.
\]
\end{prop}

Note that in the setting of this proposition, bounding all the entries
of the SSYT is equivalent to bounding only the entries in the corners of this SSYT.

When the last part of $\lambda$ is long enough relative to the parts
of $\mu$, the number of excited diagrams is given by a product.  Recall
the notion of slim shapes~$\la/\mu$ defined in Section~\ref{ss:not-yd}.

\begin{cor} \label{cor:EDhkcontent}
Let $\lambda/\mu$ be a slim skew shape, \ts $d = \ell(\lambda)$.  Then
\[
|\ED(\lambda/\mu)|  \, = \, s_{\mu}(1^d) \, = \, \prod_{(i,j) \in [\mu]} \,
\frac{d + j-i}{\mu_i+\mu'_j -i-j+1}\,.
\]
\end{cor}

\begin{proof}
In this case, by Proposition~\ref{prop:flagged}, the excited
diagrams of $\lambda/\mu$ are in bijection with SSYT of shape $\mu$
with entries at most $d$. The number of such SSYT is given by the
hook-content formula for $s_{\mu}(1^{d})$, see e.g.~\cite[Cor.~7.21.4]{EC}.
\end{proof}

Next we give a family of skew shapes that come up in the paper with product formulas for
the number of excited diagrams.

\begin{ex}[thick reverse hook] \label{ex:excited_macmahon}
 For the shape $\lambda/\mu=(b+c)^{a+c}/b^a$, the excited
  diagrams correspond to SSYT of shape $b^a$ with entries at most
  $a+c$. By subtracting $i$ from the elements in row $i$ these SSYT are
  equivalent to RPP that fit into an \ts $[a\times b \times c]\ts $ box.
  Thus, \ts $|\ED(\lambda/\mu)| = |\RPP(a,b,c)|$ \ts is given by the MacMahon box
  formula \eqref{eq:macmahon}.
\end{ex}

\subsection{Non-intersecting paths}

Excited diagrams of $\lambda/\mu$ are also in bijection with families of non-intersecting grid paths $\ga_1,\ldots,\ga_k$ with a fixed set of
start and end points, which depend only on~$\la/\mu$.  A variant of
this was proved
by Kreiman \cite[\S 5-6]{VK} (see also \cite[\S 3]{MPP2}).

\smallskip

Formally, given a connected skew shape $\lambda/\mu$ let $\ga^*_1$ be
the path starting at the northwestern-most box, following the noertwest boundary and ending at the
southeastern-most box of the connected component of this skew shape. Clearly, $(\lambda/\mu) \setminus \ga^*_1$ will be a skew
shape. We iterate the construction of the paths on each connected component of
this new skew shape ordered bottom to top and obtain a family of non-intersecting paths
 $\ga^*_1,\ldots,\ga^*_k$ in $\lambda$ with support
$\lambda/\mu$, where each path $\gamma^*_i$ starting at a box $(a_i,b_i)$ and
ending at a box $(c_i,d_i)$.  Let $\NIP(\lambda/\mu)$ be the set of $k$-tuples
$\Gamma:=(\ga_1,\ldots,\ga_k)$ of non-intersecting paths contained in
$[\lambda]$ with $\ga_i:(a_i,b_i)\to (c_i,d_i)$.

\begin{prop}[{Kreiman~\cite{VK}, see also \cite[\S 3.3]{MPP2}}] \label{prop:excited2nips}
Non-intersecting paths in $\NIP(\lambda/\mu)$ are uniquely determined by
their \emph{support}, i.e.\ set of squares.  Moreover, the set of such
supports is exactly the set of complements $[\la]\setminus D$ of excited diagrams
$D\in \ED(\lambda/\mu)$. 
\end{prop}

\begin{ex} \label{ex:444-21}
The complements of excited diagrams in $\ED(444/21)$ correspond to
tuples $(\ga_1,\ga_2)$ of nonintersecting paths in $[444]$ with
$\ga_1=(3,1)\to (1,4)$ and $\ga_2 = (3,3)\to (2,4)$:
\begin{center}
\includegraphics{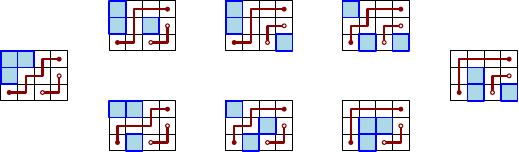}
\end{center}
\end{ex}

\begin{rem}
The convention in \cite[Lemma 5.3]{VK} and \cite[\S 3.3]{MPP2} for the
paths and  starting/ending points is slightly different: the paths begin in the southern
box of a column and end at the eastern box of a row instead. However, since
the supports of excited diagram of a shape $\lambda/\mu$ only vary along the
diagonals of $\mu$, then the portions of the paths outside this region
will stay the same. Thus, it does not matter how the paths are drawn
outside of this region.
\begin{center}
\includegraphics[scale=0.7]{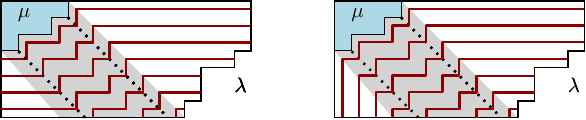}
\end{center}
\end{rem}

\begin{rem} \label{rem:pathparticleduality}
The excited diagrams of a skew shape have a ``path-particle duality''
of sorts since they can be viewed as the cells or ``particles'' of the
Young diagram of $\mu$ sliding down the cells of the Young diagram of
$\lambda$ and also their complements are in correspondence with
certain non-intersecting lattice paths. In the second
part of the paper we give two other interpretations of excited
diagrams as lozenge tilings and as terms in a known rule
for Schubert polynomials of vexillary permutations
(see $\S$\ref{sec:lozenge},~\ref{sec:KMY}).
\end{rem}

\subsection{The Naruse hook-length formula}
Recall the formula of Naruse for
$f^{\lambda/\mu}$ as a sum of products of hook-lengths (see~\cite{MPP1,MPP2}).

\begin{thm}[NHLF; Naruse \cite{Strobl}] \label{thm:IN}
Let $\lambda,\mu$ be partitions, such that $\mu \ssu \la$.  We have:
\begin{equation} \label{eq:Naruse} \tag{NHLF}
f^{\lambda/\mu} \,  = \, n! \, \sum_{D \in \ED(\lambda/\mu)}\,\.\.
 \prod_{u \in [\lambda]\setminus D} \frac{1}{ h(u)}\.\ts\.,
\end{equation}
where the sum is over all excited diagrams $D$ of $\lambda/\mu$.
\end{thm}

For the $q$-analogues we use a $q$-analogue from \cite{MPP1} for skew
semistandard Young tableaux.

\begin{thm}[\cite{MPP1}] \label{thm:skewSSYT}
We have:\footnote{In \cite{MPP1}, this is the first $q$-analogue of~\eqref{eq:Naruse}.
The second $q$-analogue is in terms of reverse plane partitions.}
\begin{equation} \label{eq:skewschur}  \tag{$q$-NHLF}
s_{\lambda/\mu}(1,q,q^2,\ldots) \, = \, \sum_{D\in \ED(\lambda/\mu)}
\. \.\.\prod_{(i,j) \in [\lambda]\setminus D}\frac{q^{\lambda'_j-i}}{1-q^{h(i,j)}}\ts.
\end{equation}
\end{thm}

These two results were the main object of our study in the two previous
papers in the series~\cite{MPP1,MPP2}.  It is also the key to most results
in this paper.  However, rather than apply it as ``black box'' we need to
use the technology of multivariate sums in the proof of~(NHLF).

\subsection{Asymptotics}\label{ss:not-asy}
We use the standard asymptotics notations $f\sim g$, $f=o(g)$, $f=O(g)$ and $f=\Omega(g)$,
see e.g.~\cite[$\S$A.2]{FS}.
Recall Stirling's formula $\ts \log n! = n \log n - n  +O(\log n)$.
Here and everywhere below $\ts\log\ts$ denotes natural logarithm.

Below is a quick list of asymptotic formulas for other functions in the introduction:
$$
\aligned
& \log \ts (2n-1)!! \,  = \,  n \ts \log n \. + \. (\log 2 \ts - \ts 1)\ts n \. + \. O(1)\ts, \\
& \log \ts \Phi(n) \,  = \,  \frac{1}{2} \. n^2 \ts \log n \. - \. \frac34 \. n^2 \. + \. O(n \ts \log n)\ts,\\
& \log \ts \Psi(n) \,  = \,  \frac12\ts n^2 \ts \log n \. + \. \left(\frac{\log 2}{2} \ts - \ts \frac34\right) n^2
\. + \. O(n \ts \log n)\ts,\\
& \log \ts \ts \Lam(n) \,  = \, \frac{1}{4} \. n^2 \ts \log n \. - \. \frac38 \. n^2 \. + \. O(n \ts \log n) \ts,
\endaligned
$$

\nin
see \cite[\href{http://oeis.org/A001147}{A001147}]{OEIS}, \cite[\href{http://oeis.org/A008793}{A008793}]{OEIS},  \cite[\href{http://oeis.org/A057863}{A057863}]{OEIS}, and \cite[\href{http://oeis.org/A113296}{A113296}]{OEIS}.
We should also mention that the numbers $\Phi(n)$ are the integer values of the
\emph{Barnes $G$-function}, whose asymptotics has been extensively studied, see e.g.~\cite{AsR}.

\bigskip

\section{Multivariate path identity}\label{sec:multi}

\subsection{Multivariate sums of excited diagrams}
For the skew shape $\lambda/\mu \subseteq d \times (n-d)$ we define $F_{\lambda/\mu}({\bf x}
\,\vert\, {\bf y})$ and $G_{\lambda/\mu}({\bf x} \,\vert\, {\bf y})$ to be the multivariate sums of excited diagrams

\begin{align*}
G_{\lambda/\mu}({\bf x} \,\vert\, {\bf y}) &\, := \,\sum_{D \in \ED(\lambda/\mu)}
  \prod_{(i,j) \in D} (x_i-y_j)\,,\\
F_{\lambda/\mu}({\bf x} \,\vert\, {\bf y}) & \, := \, \sum_{D \in \ED(\lambda/\mu)}
\prod_{(i,j)\in [\lambda]\setminus D} \frac{1}{x_i-y_j}\,.
\end{align*}

By Proposition~\ref{prop:excited2nips}, the sum \ts $F_{\lambda/\mu}({\bf x}\,\vert\,
{\bf y})$ \ts can be written as a multivariate sum of non-intersecting
  paths.

\begin{cor} \label{cor:FisNIP}  In the notation above, we have:
\[
F_{\lambda/\mu}({\bf x}\,\vert\, {\bf y}) \, = \. \sum_{\Gamma \in
  \NIP(\lambda/\mu)} \prod_{(i,j) \in \Gamma} \frac{1}{x_i-y_j}\,.
\]
\end{cor}

Note that by evaluating $F_{\lambda/\mu}({\bf x} \,\vert\, {\bf y})$ at
$x_i = \lambda_i-i+1$ and $y_j=-\lambda'_j+j$ and multiplying by
$|\lambda/\mu|!$ we obtain the RHS of \eqref{eq:Naruse}.
\begin{equation} \label{eq:relFandf_lam/mu}
\left. F_{\lambda/\mu}({\bf x} \,\vert\, {\bf y}) \right|_{\substack{x_i=
    \lambda_i-i+1\\y_j=-\lambda'_j+j}} \, \.  = \, \frac{f^{\lambda/\mu}}{|\lambda/\mu|!}\,.
\end{equation}

Note that by evaluating $(-1)^{|\lambda/\mu|}F_{\lambda/\mu}({\bf x} \,\vert\, {\bf y})$ at
$x_i = q^{\lambda_i-i+1}$ and $y_j=q^{-\lambda'_j+j}$ by \eqref{eq:skewschur} we obtain
\begin{equation} \label{eq:qrelFandSchur}
\left. (-1)^{|\lambda/\mu|}F_{\lambda/\mu}({\bf x} \,\vert\, {\bf y}) \right|_{\substack{x_i=
    q^{\lambda_i-i+1}\\y_j=q^{-\lambda'_j+j}}} \, \. = \,
q^{C(\lambda/\mu)} \ts s_{\lambda/\mu}(1,q,q^2,\ldots)\.,
\end{equation}
where 
\begin{equation}\label{eq:defClamu}
C(\lambda/\mu) = \sum_{(i,j) \in \lambda/\mu} (j-i).
\end{equation}

The multivariate sum of excited diagrams can be written as an
evaluation of a factorial Schur function.

\begin{thm}[see \cite{IN}]
For a skew shape $\lambda/\mu$ inside the rectangle $d
  \times (n-d)$ we have:
\[
G_{\lambda/\mu}({\bf x} \,\vert\, {\bf y})  \,=\,
s_{\mu}^{(d)}(y_{\lambda_1+d},y_{\lambda_2+d-1},\ldots,y_{\lambda_d+1}
\,\vert\, y_1,\ldots,y_n).
\]
\end{thm}

\smallskip

\begin{ex} \label{ex:qexcited_macmahon}
 Continuing with Example~\ref{ex:excited_macmahon}, take the thick
 reverse hook $\lambda/\mu=(b+c)^{a+c}/b^a$. When we evaluate \ts
 $G_{\lambda/\mu}({\bf x} \,\vert\, {\bf y})$ \ts at $x_i =q^i$, $y_j=0$,
 we obtain the $q$-analogue of the MacMahon box
formula\eqref{eq:qmacmahon}~:
\begin{equation} \label{eq:qsumexcited_macmahon}
G_{(b+c)^{a+c}/b^a}(q^1,q^2,\ldots \,\vert\, 0,0,\ldots) \, = \,
q^{b\binom{a+1}{2}} \. \prod_{i=1}^a\prod_{j=1}^b\prod_{k=1}^c
\frac{1-q^{i+j+k-1}}{1-q^{i+j+k-2}}\..
\end{equation}
\end{ex}

\smallskip

Let ${\bf   z}^{\<\lambda\>}$ be the tuple of length $n$ of $x$'s and $y$'s by
reading the horizontal and vertical steps of~$\lambda$ from $(d,1)$ to
$(1,n-d)$:  i.e. $z_{\lambda_i+d-i+1} = x_i$ and
$z_{d+j-\lambda'_j}=y_j$. For example, for $d=4$, $n=9$
and $\lambda = (5533)$, we have \ts
${\bf z}^{\<\lambda\>} = (y_1,y_2,y_3,x_4,x_3,y_4,y_5,x_2,x_1)$:
\begin{center}
\includegraphics{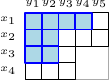}
\end{center}

Combining results of Ikeda--Naruse \cite{IN}, Knutson--Tao
\cite{KT},  Lakshmibai--Raghavan--Sankaran \cite{LRS}, one obtains
the following formula for an evaluation of factorial Schur functions.

\begin{lemma}[Theorem~2 in~\cite{IN}] \label{lem:key_border_strips}
For every skew shape \ts $\lambda/\mu \subseteq d\times (n-d)$, we have:
\begin{equation} \label{eq:G2factschur}
G_{\lambda/\mu}({\bf x} \,\vert\, {\bf y}) \, = \, s_{\mu}^{(d)}({\bf x} \,\vert\,
{\bf z^{\<\lambda\>}})\..
\end{equation}
\end{lemma}

\begin{cor}  We have:
\begin{equation} \label{eq:F2factschur}
F_{\lambda/\mu}({\bf x} \,\vert\, {\bf y}) \, = \, \frac{s_{\mu}^{(d)}({\bf x} \,\vert\, {\bf
    z}^{\<\lambda\>})}{s_{\lambda}^{(d)}({\bf x} \,\vert\, {\bf
    z}^{\<\lambda\>})}\..
\end{equation}
\end{cor}

\begin{proof}
By definition the multivariate polynomial $G_{\lambda/\lambda}({\bf
  x} \,\vert\, {\bf y})$ is the product $\prod_{(i,j) \in [\lambda]}
  (x_i-y_j)$ and thus we can write $F_{\lambda/\mu}({\bf x} \,\vert\,
  {\bf y})$ as the following quotient
\[
F_{\lambda/\mu}({\bf x} \,\vert\, {\bf y}) \,=\,  \frac{G_{\lambda/\mu}({\bf
  x} \,\vert\, {\bf y})}{G_{\lambda/\lambda}({\bf
  x} \,\vert\, {\bf y})}.
\]
The result now follows by applying Lemma~\ref{lem:key_border_strips}
to both the numerator and denominator on the RHS above.
\end{proof}

\subsection{Symmetries}\label{ss:multi-sym}
The factorial Schur function $s^{(d)}_{\mu}({\bf x} \,\vert\, {\bf y})$ is
symmetric in~$\bx$. By Lemma~\ref{lem:key_border_strips}, the
multivariate sum $G_{\lambda/\mu}({\bf x}\,\vert\, {\bf y})$ is an
evaluation of a certain factorial Schur function, which in general
is not symmetric in $\bx$.

\begin{ex}
The shape $\lambda/\mu = 332/21$ from Example~\ref{ex:excited} has five excited diagrams. One can check that the multivariate polynomial
\begin{multline*}
G_{333/21}(x_1,x_2,x_3\,\vert\, y_1,y_2,y_3) \, =  \, (x_1-y_1)(x_1-y_2)(x_2-y_1) \. + \.
(x_1-y_1)(x_1-y_2)(x_3-y_2) \\ \qquad \quad + \.(x_1-y_1)(x_2-y_3)(x_2-y_1) \. + \.
(x_1-y_1)(x_2-y_3)(x_3-y_2) \. + \.(x_2-y_2)(x_2-y_3)(x_3-y_2)\ts,
\end{multline*}
is not symmetric in  $\bx = (x_1,x_2,x_3)$.
\end{ex}

Now, below we present two cases when the sum
$G_{\lambda/\mu}({\bf x} \,\vert\, {\bf y})$ is in fact symmetric in~$\bx$.
The first case is when~$\mu$ is a rectangle contained in~$\lambda$.

\begin{prop} \label{prop:sym2}
Let $\mu = p^k$ be a rectangle, $p\geq k$, and let $\lambda$ be arbitrary
partition containing~$\mu$.  Denote \ts $\ell:=\max\{i: \lambda_i-i \geq p-k\}$.
Then:
$$
G_{\lambda/p^k}({\bf x} \,\vert\, {\bf y}) \, = \,
s_{p^k}^{(\ell)}(x_1,\ldots,x_\ell \,\vert\, y_1,\ldots,y_{p+\ell-k})\ts.
$$
In particular, the polynomial \ts $G_{\lambda/p^k}({\bf x} \,\vert\, {\bf y})$\ts
is symmetric in $(x_1,\ldots,x_{\ell})$.
\end{prop}
\begin{proof}
First, observe that \ts $\ED(\la/p^k) = \ED( (p+\ell-k)^\ell/p^k)$ \ts since the
movement of the excited boxes is limited by the position of the corner
box of~$p^k$, which moves along the diagonal $j-i = p-k$ up to the
boundary of $\la$, at position $(\ell, p+\ell-k)$. Thus, the
excited diagrams of $\la/\mu$ coincide, as sets of boxes with the
excited diagrams of $(p+\ell-k)^\ell/\mu$. Then:
\begin{align*}
G_{\lambda/p^k}({\bf x} \,\vert\, {\bf y})  & \, = \, \sum_{D \in
  \ED((p+\ell-k)^\ell/p^k)} \prod_{(i,j)\in D} (x_i-y_j) \, = \, G_{(p+\ell-k)^\ell/p^k}({\bf x} \,\vert\, {\bf y}) \\
  & \, = \,
  s_{p^k}^{(\ell)}\bigl(x_1,\ldots,x_{\ell} \,\vert\, {\bf z}^{\<(p+\ell-k)^{\ell}\>}\bigr)\..
\end{align*}
Note that ${\bf z}^{\<(p+\ell-k)^{\ell}\>} = (y_1,\ldots,y_{p+\ell-k},
x_{\ell},\ldots,x_1)$. Let us now invoke the original combinatorial
formula for the factorial Schur functions,
equation~\eqref{eq:fschur-explicit}, with $a_j = y_j$ for $j\leq
p+\ell-k$ and $a_{p+\ell-k+j} = x_{\ell+1-j}$ otherwise. Note also that
when $T$ is an SSYT of shape $p^k$ and entries at most~$\ell$, by the
strictness of columns we have $T(i,j) \leq \ell-(k-i)$ for all
entries in row $i$. We conclude:
\[
T(u) +c(u) \leq \ell-(k-i) +j-i =\ell-k+j
\leq \ell-k +p \..
\]
Therefore, \ts $a_{T(u)+c(u)}=y_{T(u)+c(u)}$, where only the first
$p+\ell-k$ parameters $a_i$ are involved in the formula. Then:
\begin{align*}
s_\mu^{(\ell)}(x_\ell,\ldots,x_1\,\vert\, a_1,\ldots,a_{p+\ell-k},a_{p+\ell-k+1},\ldots)
& \, = \, s_\mu^{(\ell)}(x_\ell,\ldots,x_1\,\vert\, a_1,\ldots,a_{p+\ell-k})\\
&\, = \, s_\mu^{(\ell)}(x_1,\ldots,x_{\ell}\,\vert\, y_1,\ldots,y_{p+\ell-k}),
\end{align*}
 since now the parameters of the factorial Schur are independent
 of the variables ${\bf x}$ and the function is also symmetric in ${\bf x}$.
\end{proof}

The second symmetry involves slim skew shapes (see Section~\ref{ss:not-yd}).
An example includes a skew shape $\la/\mu$, where $\lambda$ is the
rectangle $(n-d)^d$ and $\mu_1 \leq n-2d+1$.

\begin{prop}\label{prop:z_vs_xy}
Let $\la/\mu$ be a slim skew shape inside the rectangle $d\times (n-d)$. Then:
$$
G_{\lambda/\mu}({\bf x}\,\vert\, {\bf y}) \, =\, s_{\mu}^{(d)}(x_1,\ldots,x_d\,\vert\, y_1,\ldots,y_{\lambda_d})\ts.
$$
In particular, the polynomial \ts $G_{\lambda/\mu}({\bf x} \,\vert\, {\bf y})$ \ts is symmetric in $(x_1,\ldots,x_{d})$.
\end{prop}

\begin{proof}
Note that $\ts {\bf z}^{\<\lambda\>} =
(y_1,\ldots,y_{\lambda_d},x_d,\ldots)$. Note also that for all
$j=1,\ldots,d$,
\[
\mu_j+d-j \leq
\lambda_d -d+1+d-j \leq \lambda_d \leq \lambda_j,
\]
and so $z_1,\ldots,z_{\mu_j+d-j} =y_1,\ldots,y_{\mu_j+d-j}$.
Next, we evaluate the
factorial Schur function on the RHS of \eqref{eq:G2factschur} via its determinantal formula
\eqref{eq:fschur-det}.  We obtain:
\begin{align*}
G_{\lambda/\mu}({\bf x} \,\vert\, {\bf y}) & \. = \. s_{\mu}^{(d)}(x_1,\ldots,x_d
\,\vert\, z_1,\ldots,z_n) \. = \.
\frac{\det[ (x_{d+1-i} -z_1) \cdots (x_{d+1-i} - z_{\mu_j+d-j})]_{i,j=1}^d}{\Delta(x_1,\ldots,x_d)}\\
&= \. \frac{\det[ (x_{i} -y_1) \cdots (x_{i} - y_{\mu_j+d-j})]_{i,j=1}^d}{\Delta(x_1,\ldots,x_d)}
\. = \. s_{\mu}^{(d)}(x_1,\ldots,x_d \,\vert\, y_1,\ldots,y_{\lambda_d})\ts,
\end{align*}
where the last equality is by the same determinantal formula.
\end{proof}

\begin{ex}
For  $\lambda/\mu = 444/21$, the multivariate sum $G_{444/21}(x_1,x_2,x_3 \,\vert\,
y_1,y_2,y_3,y_4)$ of the eight excited diagrams in $\ED(444/21)$ is
symmetric in $x_1,x_2,x_3$.
\end{ex}

\subsection{Multivariate path identities} \label{ss:multi-id}
We give two identities for the multivariate sums over non-intersecting
paths as applications of each of Propositions~\ref{prop:sym2} and
\ref{prop:z_vs_xy}.

\begin{figure}[hbt]
\begin{center}
\includegraphics{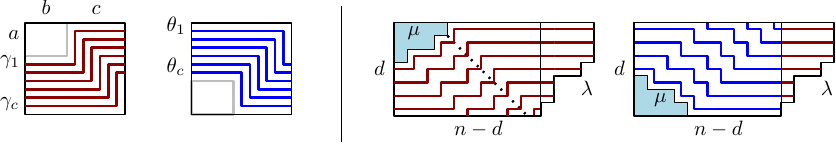}
\caption{Left: paths and flipped paths in Theorem~\ref{thm:thickstrip}. Right: paths
  and flipped paths in Theorem~\ref{thm:sympaths_rect_mu}.}
\label{fig:path_symmetries}
\end{center}
\end{figure}

\begin{thm}\label{thm:thickstrip}
We have the following identity for multivariate rational functions:
\begin{equation}
\label{eq:Naruse-1-path}
\sum_{\substack{\Gamma=(\ga_1,\ldots,\ga_c) \\ \ga_p:(a+p,1)\to
  (p,b+c)}} \prod_{(i,j) \in \Gamma} \frac{1}{x_i-y_j} \ \, =
   \sum_{\substack{\Theta=(\theta_1,\ldots,\theta_c)\\ \theta_p:(p,1)
  \to (a+p,b+c)}} \prod_{(i,j) \in \Theta} \frac{1}{x_i-y_j}\,,
 \end{equation}
where the sums are over non-intersecting lattice paths as above.
Note that the LHS is equal to \ts $F_{(b+c)^{a+c}/b^a}({\bf x} \,\vert\, {\bf y})$ \ts
defined above.
\end{thm}

In the next section we use this identity
to obtain product formulas for $f^{\lambda/\mu}$ for certain
families of shapes $\lambda/\mu$.  In the case $c=1$,
we evaluate \eqref{eq:Naruse-1-path} at $x_i =i$ and $y_j = -j+1$
obtain the following corollary.

\begin{cor}[\cite{MPP1}] \label{cor:Naruse-1-path}  We have:
\begin{equation} \label{eq:Naruse-1-path}
\sum_{\ga: (a,1) \to (1,b)} \prod_{(i,j) \in \ga} \frac{1}{i+j-1}  \  =
\sum_{\ga:(1,1) \to (a,b)}\prod_{(i,j) \in \ga} \frac{1}{i+j-1}\,.
\end{equation}
\end{cor}

Equation \eqref{eq:Naruse-1-path} is a
special case of \eqref{eq:Naruse} for the skew shape $(b+1)^{a+1}/b^a$
\cite[\S 3.1]{MPP1}.
This equation is also a special case of {\em
Racah formulas} in \cite[\S 10]{BGR} (see in \S~\ref{subsec:Racah}).

\begin{proof}[Proof of Theorem~\ref{thm:thickstrip}]
By Proposition~\ref{prop:z_vs_xy} for the shape
$(b+c)^{a+c}/b^a$, we have:
\[
G_{(b+c)^{a+c}/b^a}({\bf x} \,\vert\, {\bf y}) \, = \, s_{b^a}(x_1,\ldots,x_{a+c}
\, \,\vert\, \, y_1,\ldots,y_{b+c})\..
\]
Divide the LHS by $\prod_{(i,j) \in (b+c)^{a+c}} (x_i-y_j)$ to obtain
$F_{(b+c)^{a+c}/b^a}({\bf x} \,\vert\, {\bf y})$, the
multivariate sum over excited diagrams. By
Corollary~\ref{cor:FisNIP}, this is also a multivariate sum over
tuples of non-intersecting paths in $\NIP( (b+c)^{a+c}/b^a)$~:
\begin{equation} \label{eq:sumpaths2factschur}
\sum_{\substack{\Gamma=(\ga_1,\ldots,\ga_c) \\ \ga_p:(a+p,1)\to
  (p,b+c)}} \prod_{(i,j) \in \Gamma} \frac{1}{x_i-y_j} \, =\,  s_{b^a}(x_1,\ldots,x_{a+c}
\, \,\vert\, \, y_1,\ldots,y_{b+c}) \prod_{(i,j) \in (b+c)^{a+c}} \frac{1}{x_i-y_j}\,.
\end{equation}
Finally, the symmetry in $x_1,\ldots,x_{a+c}$ of the RHS above
implies that we can flip these variables and
consequently the paths $\ga'_p$ to paths $\theta_p:(p,1)
  \to (a+p,b+c)$ (see Figure~\ref{fig:path_symmetries}),
  and obtain the needed
  expression.
\end{proof}


For a partition $\mu$ inside the rectangle $d\times (n-d)$ of length $\ell$, let $\hf{\mu}$
denote the tuple $(0^{d-\ell},\mu_{\ell},\mu_{\ell-1},\ldots,\mu_1)$.

\begin{thm} \label{thm:sympaths_rect_mu}
Let $\lambda/\mu \ssu d\times (n-d)$ be a slim skew shape. Then:
\begin{equation} \label{eq:sympaths_rect_mu}
\sum_{\Ga \in \NIP( \lambda /\mu)} \prod_{(i,j) \in \Gamma}
  \frac{1}{x_i-y_j} = \sum_{\Ga \in \NIP( \lambda /\hf{\mu})} \prod_{(i,j) \in \Gamma}
  \frac{1}{x_i-y_j}\,.
\end{equation}
\end{thm}

\begin{proof}
By Proposition~\ref{prop:sym2} for the shape $\lambda/\mu$ we have
that
\[
G_{(n-d)^d/\mu}({\bf x} \,\vert\, {\bf y}) = s_{\mu}^{(d)}(x_1,\ldots,x_d
\,\vert\, y_1,\ldots,y_{\lambda_d}).
\]
The rest of the proof follows mutatis mutandis that of Theorem~\ref{thm:thickstrip}
for the shape $\lambda/\mu$ instead of the shape
$(b+c)^{a+c}/b^a$. See Figure~\ref{fig:path_symmetries}.
\end{proof}

\begin{rem}
In \cite{MPP4} we use this second symmetry identity to give new lower bounds on $f^{\lambda/\mu}$
for several other families of slim shapes $\lambda/\mu$.
\end{rem}

\subsection{Variant of excited diagrams for rectangles and slim shapes} \label{subsec:neED}

Recall that for $\mu \subseteq d \times (n-d)$ of length
$\ell$, we denote by $\hf{\mu}$ the tuple \ts
$({\bf 0}^{d-\ell},\mu_{\ell},\mu_{\ell-1},\ldots,\mu_1)$.
We interpret the complements of the supports of the paths in
$\NIP\bigl((b+c)^{a+c}/0^cb^a\bigr)$ and in $\NIP(\lambda/\hf{\mu})$,
as variants of excited diagrams.

A {\em NE-excited} diagram of shape $\lambda/\mu$ is a
subdiagram of $\lambda$ obtained from the Young diagram of
$\hf{b^a} =(0^cb^a)$ (and $\hf{\mu}$)
after a sequence of moves from $(i,j)$ to $(i-1,j+1)$ provided $(i,j)$ is
in the subdiagram $D$ and all of $(i-1,j),(i-1,j+1),(i,j+1)$ are in
$[\lambda] \setminus D$. We denote the set of such diagrams by
$\ED^{\nearrow}(\lambda/\mu)$. Analogous to Proposition~\ref{prop:excited2nips}, the
complements of these diagrams correspond to tuples of paths in
$\NIP\bigl((b+c)^{a+c}/0^cb^a\bigr)$ (in $\NIP(\lambda/\hf{\mu})$).
  Flipping horizontally the \ts $[d\times \lambda_d]$ \ts rectangle  gives a
  bijection between excited diagrams and NE-excited diagrams of
  $\lambda/\mu$. Thus
\[
\bigl|\ED^{\nearrow}(\lambda/\mu)\bigr| \. = \. \bigl|\ED(\lambda/\mu)\bigr|\ts.
\]

Moreover, equation~\eqref{eq:sympaths_rect_mu} states that
  such a flip also preserves the multivariate series
  $F_{\lambda/\mu}({\bf x}\,\vert\, {\bf y})$ and polynomial $G_{\lambda/\mu}({\bf x}
  \,\vert\, {\bf y})$.

\begin{cor} \label{cor:sympaths_rect_rect_sum_excited}  We have:
\begin{equation} \label{eq:symrectF}
F_{(b+c)^{a+c}/b^a}({\bf x} \,\vert\, {\bf y}) \, = \, \sum_{D \in
  \ED^{\nearrow}\bigl((b+c)^{a+c}/0^cb^a\bigr)} \prod_{(i,j) \in [\lambda]\setminus D} \frac{1}{x_i-y_j}\,,
\end{equation}
and
\begin{equation} \label{eq:symrectG}
G_{(b+c)^{a+c}/b^a}({\bf x} \,\vert\, {\bf y}) \, = \, \sum_{D \in
  \ED^{\nearrow}\bigl((b+c)^{a+c}/0^cb^a\bigr)} \prod_{(i,j) \in D}
(x_i-y_j)\,.
\end{equation}
\end{cor}

\begin{proof}
This follows from the discussion above, Corollary~\ref{cor:FisNIP} and
Theorem~\ref{thm:thickstrip}.
\end{proof}

\begin{cor} \label{cor:sympaths_rect_mu_sum_excited}
For a slim skew shape $\lambda/\mu$, we have:
\begin{equation}
F_{\lambda/\mu}({\bf x} \,\vert\, {\bf y}) \, = \, \sum_{D \in
  \ED^{\nearrow}(\lambda/\mu)} \prod_{(i,j) \in [\lambda]\setminus D} \frac{1}{x_i-y_j}\,.
\end{equation}
\end{cor}

\begin{proof}
This follows from the discussion above, Corollary~\ref{cor:FisNIP} and
Theorem~\ref{thm:sympaths_rect_mu}.
\end{proof}

\bigskip

\section{Skew shapes with product formulas} \label{sec:skewprod}

In this section we use Theorem~\ref{thm:thickstrip} to obtain product formulas for a
family of skew shapes.

\subsection{Six-parameter family of skew shapes}\label{subsec:shapes_families}

For all \ts $a,b,c,d,e,m \in \nn$, let \ts $\LA(a,b,c,d,e,m)$ \ts denote
the skew shape $\lambda/b^a$, where $\lambda$ is given by
\begin{equation} \label{eq:deflam}
\lambda \. := \. (b+c)^{a+c} \. + \. \bigl(\nu \cup \theta'\bigr)\.,
\end{equation}
and where $\ts \nu  =(d+(a+c-1)m, d+(a+c-2)m,\ldots,d)$, \ts
$\theta=(e+(b+c-1)m,e+(b+c-2)m,\ldots,e)$; see Figure~\ref{fig:abcdem-shape}. \ts
This shape satisfies two key properties:
\begin{align}
\lambda_{a+c+1} &\leq b+c \tag{P1} \label{prop1:lam}\ts, \\
 \lambda_i+\lambda'_j&=\lambda_{r}+\lambda'_s, \quad \text{ if }
                       i+j=r+s \text{ and }   (i,j), (r,s) \in
(b+c)^{a+c}\ts.   \tag{P2} \label{prop2:lam}
\end{align}
The second property implies that $\lambda_i -\lambda_{i+1} = \lambda'_j-\lambda'_{j +1}$ for all $i \leq a+c-1$ and
$j\leq b+c-1$, and therefore $\lambda_i-\lambda_{i+1}$ is independent
of $i$,  i.e.\ the parts of $\lambda$ are given by
an arithmetic progression. Also, the antidiagonals in
$(b+c)^{a+c}$ inside $\lambda$ have the same
hook-lengths.

Here are two extreme special cases:
$$\LA(a,b,c,0,0,1) \. = \. \delta_{a+b+2c}/b^a\,, \qquad \LA(a,b,c,d,e,0) \. = \. (b+c+d)^{a+c}(b+c)^{e}/b^a\..
$$
Note that these shapes are depicted in
Figure~\ref{fig:intro_shapes}(ii) and Figure~\ref{fig:intro_shapes}(i), respectively.

\begin{figure}[hbt]
\begin{center}
\includegraphics{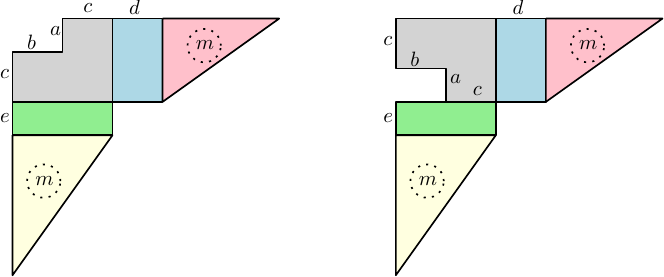}
\caption{Left: Skew shape $\LA(a,b,c,d,e,m)$. Right: the cells
whose hook-lengths appear in the product formula of Theorem~\ref{thm:skewprod}.}
\label{fig:abcdem-shape}
\end{center}
\end{figure}

Next, we give a product formula for $f^{\pi}$ where
  $\pi=\LA(a,b,c,d,e,m)$ in terms of falling superfactorials
$$\Psi^{(m)}(n) \, := \, \prod_{i=1}^{n-1}\. \prod_{j=1}^i \, (jm+j-1)_m\,,
\quad \text{where} \quad
(k)_m\. = \. k(k-1)\cdots (k-m+1).
$$
Note that $\Psi^{(0)}(n)=1$ and $\Psi^{(1)}(n) = \Psi(n)$.

\begin{thm} \label{thm:skewprod}
Let \ts $\pi=\LA(a,b,c,d,e,m)$ \ts be as above. Then \ts
$f^{\pi}$ \ts is given by the following product:
\begin{align} \label{eq:prodformula0}
f^{\pi} &\,=\, \.  n!\cdot \prod_{i=1}^a\prod_{j=1}^{b}\prod_{k=1}^c \.
\frac{i+j+k-1}{i+j+k-2} \cdot \,\prod_{(i,j) \in \lambda/(0^cb^{a})} \,
                             \frac{1}{h_{\lambda}(i,j)} \, , \\
 &\, = \, \. n! \. \cdot \.
\frac{\Phi(a+b+c)\.\Phi(a)\.\Phi(b)\.\Phi(c)}{\Phi(a+b)\.\Phi(b+c)\.\Phi(a+c)\.\Psi^{(m)}(a+c)\.\Psi^{(m)}(b+c)}
   \,\. \times
\label{eq:prodformula}
\\
\noalign{\centering $\displaystyle \times \prod_{i=0}^{a+c-1} \frac{(i(m+1))!}{(d+i(m+1))!} \, \. \prod_{i=0}^{b+c-1} \frac{(i(m+1))!}{(e+i(m+1))!}
\, \,
\frac{\prod_{i=0}^{b-1} \prod_{j=0}^{a-1} (1+d+e+(c+i+j)(m+1))}{\prod_{i=0}^{b+c-1}
  \prod_{j=0}^{a+c-1} (1+d+e+(i+j)(m+1))}\,.$} \notag
\end{align}
\end{thm}

\smallskip

\begin{proof}[Proof of Corollaries~\ref{cor:abcde-shape},
  \ref{cor:abc-shape} and \ref{cor:abcde1-shape}]
Use~\eqref{eq:prodformula} for the shapes $\LA(a,b,c,e,d,0)$, $\LA(a,b,c,0,0,1)$, and $\LA(a,b,c,d,e,1)$,
respectively.
\end{proof}

We also give a product formula for the generating function of SSYT of
these shapes.

\begin{thm} \label{thm:q_skewprod}
Let $\pi = \LA(a,b,c,d,e,m)$ be as above. Then:
\begin{equation} \label{eq:q_prodformula}
s_{\pi}(1,q,q^2,\ldots) \,= \, q^{N}\. \prod_{i=1}^a\prod_{j=1}^{b}\prod_{k=1}^c \.
\frac{1-q^{(m+1)(i+j+k-1)}}{1-q^{(m+1)(i+j+k-2)}} \, \prod_{(i,j) \in
  \lambda/(0^cb^{a})} \.\frac{1}{1-q^{h_{\lambda}(i,j)}}\,,
\end{equation}
where $\ts N \ts = \ts \sum_{(i,j) \in \la/b^a} (\lambda'_j-i)$.
\end{thm}

As in the proof above, we obtain explicit formulas for the skew shapes
$\LA(a,b,c,d,e,0)$, $\LA(a,b,c,0,0,1)$, and $\LA(a,b,c,d,e,1)$. By comparing
\eqref{eq:prodformula} and \eqref{eq:q_prodformula}, up to the power of $q$, these
cases are obtained by ``$q$-ifying'' their counterpart formulas for
$f^{\lambda/\mu}$. The formulas are
written in terms of:

\begin{center}
\begin{tabular}{llcl}
 {\em $q$-factorials}& $[m]!$ &$:=$&$(1-q)(1-q^2)\cdots
(1-q^m)$\\
 {\em $q$-double factorials} &
$[2n-1]!!$ &$:=$& $(1-q)(1-q^3)\cdots (1-q^{2n-1})$\\
 {\em $q$-superfactorials} &
$\Phi_q(n)$&$:=$&$[1]!\cdot [2]! \. \cdots \. [n-1]!$\\
{\em $q$-super
  doublefactorials} & $\Psi_q(n)$ &$:=$& $[1]!! \. \cdot \. [3]!! \cdots [2n-3]!!$\\
{\em $q$-double superfactorial} &$\Lam_q(n)$ &$:=$&
$[n-2]!\. [n-4]!\. \cdots \. $\\
{\em $q$-shifted super doublefactorial} &$\Psi_q(n;k)$ &$:=$&
$[k+1]!!\. [k+3]!! \. \cdots \. [k+2n-3]!!$
\end{tabular}
\end{center}
Note that in the classical notation $[m]!= \prod \frac{1-q^i}{1-q}$ (e.g. from \cite{EC}), however here the factors of $(1-q)$ are omitted, otherwise the formulas below would have an additional factor of $(1-q)^{|\pi|}$.

\smallskip

\begin{cor} \label{cor:qabcde-shape}
For the skew shape \ts $\pi=\LA(a,b,c,d,e,0)$, we have:
\[
s_{\pi}(1,q,\ldots) \,=\,
\frac{q^{N}\ts\Phi_q(a)\ts\Phi_q(b)\ts\Phi_q(c)\ts\Phi_q(d)\ts\Phi_q(e)\ts\Phi_q(a+b+c)\ts\Phi_q(c+d+e)
\ts\Phi_q(a+b+c+d+e)}{\Phi_q(a+b)\ts\Phi_q(d+e)\ts\Phi_q(a+c+d)\ts\Phi_q(b+c+e)\ts\Phi_q(a+b+2c+d+e)}\,,
\]
where $\ts N = b\binom{c+e}{2} +c\binom{a+c+e}{2} + d\binom{a+c}{2}$.
\end{cor}

Note that this is a $q$-analogue of Corollary~\ref{cor:qabcde-shape} by Kim and~Oh 
(see also~$\S$\ref{sec:kimoh}). 

\smallskip

\begin{cor}[Krattenthaler--Schlosser~\cite{KS}, see \S\ref{sec:dewitt}] \label{cor:qabc-shape}
For the skew shape \ts $\pi=\LA(a,b,c,0,0,1)$, we have:
\[
s_{\pi}(1,q,\ldots) \,=\, q^{N} \, \frac{\Phi_q(a)\. \Phi_q(b)\. \Phi_q(c)\. \Phi_q(a+b+c) \ts \cdot \ts \Psi_q(c)\Psi_q(a+b+c)}{\Phi_q(a+b)\. \Phi_q(b+c)\.\Phi_q(a+c) \ts \cdot \ts \Psi_q(a+c)\Psi_q(b+c)\Psi_q(a+b+2c)}\,,
\]
where $N = \binom{a+b+2c}{3} + b\binom{a+1}{2} + a\binom{b+1}{2}-ab(a+b+2c)$.
\end{cor}

\smallskip

\begin{cor} \label{cor:qabcde1-shape}
For the skew shape \ts $\pi=\LA(a,b,c,d,e,1)$, we have:
\begin{multline*}
s_{\pi}(1,q,\ldots) \,=\, q^{N} \,
\frac{\Phi_q(a)\.\Phi_q(b)\. \Phi_q(c)\. \Phi_q(a+b+c)}{\Phi_q(a+b)\.
\Phi_q(b+c)\. \Phi_q(a+c)}\,
\times \\
\qquad \times \, \frac{\Psi_q(c;d+e) \. \Psi_q(a+b+c;d+e)\ts \cdot \ts
   \, \Lam_q(2a+2c)\.\Lam_q(2b+2c)}
{\Psi_q(a+b+2c;d+e)\. \Psi_q(a+c)\. \Psi_q(b+c) \ts \cdot \ts
  \Lam_q(d) \ts \Lam_q(e) \ts\Lam_q(2a+2c+d)\. \Lam_q(2b+2c+e)}\,.
\end{multline*}
where $N =\binom{a+b+2c+e}{3} + d\binom{a+c}{2} + \binom{a+c}{3} -
\binom{a+c+e}{3}+b\binom{a+1}{2} + a\binom{b+1}{2} - ab(a+b+2c+e)$.
\end{cor}

\smallskip

\begin{proof}[Proof of Corollaries~\ref{cor:qabcde-shape},
  \ref{cor:qabc-shape}, and \ref{cor:qabcde1-shape}]
We ``$q$-ify'' the formula in corollaries ~\ref{cor:abcde-shape},
  \ref{cor:abc-shape}, and \ref{cor:abcde1-shape} respectively and
  calculate the corresponding power of $q$ in \eqref{eq:q_prodformula} to obtain the stated formula.
\end{proof}

The rest of the section is devoted to the proof of
Theorem~\ref{thm:skewprod} and Theorem~\ref{thm:q_skewprod}.

\subsection{Proof of the product formulas for skew SYT}

\begin{proof}[Proof of Theorem~\ref{thm:skewprod}]

The starting point is showing that the  skew shape
$\lambda/b^a = \LA(a,b,c,d,e,m)$ and the thick reverse hook
\ts $(b+c)^{a+c}/b^{a}=\LA(a,b,c,0,0,0)$ \ts have the
same excited diagrams. To simplify the notation, let $R=(b+c)^{a+c}$
be the rectangle \ts $\bigl[(a+c)\times (b+c)\bigr]$.

\begin{lemma} \label{lem:same_ED}
The skew shapes $\LA(a,b,c,d,e,m)$ and $\LA(a,b,c,0,0,0) =
R/b^a$ have the same excited diagrams.
\end{lemma}

\begin{proof}
This can be seen directly from the description of excited diagrams:
 by property~\eqref{prop1:lam} from Section~\ref{subsec:shapes_families}, the cell $(b,a)$ of $[\mu]$ cannot go past the cell $(b+c,a+c)$
so the rest of $[\mu]$ is confined in the rectangle $(b+c)^{a+c}$. Alternatively by Proposition~\ref{prop:flagged}, the
excited diagrams of both shapes correspond to SSYT of shape $b^a$
with entries at most $b+c$. Then the map $\varphi^{-1}$ applied to
such tableaux yields the
same excited diagrams.
\end{proof}

By \eqref{eq:Naruse} and Lemma~\ref{lem:same_ED} we have:
\begin{equation} \label{eq2:pfskewprod}
\frac{f^{\lambda/b^a}}{n!}\, = \, \left[\prod_{u\in [\lambda] \setminus
    R} \, \frac{1}{h_{\lambda}(i,j)}\right] \, \sum_{D \in
  \ED(R/b^a)}  \prod_{(i,j) \in R\setminus D}\frac{1}{h_{\lambda}(i,j)}\,.
\end{equation}

The sum over excited diagrams of $R/b^a$
  with hook-lengths in $\lambda$ on the RHS above evaluates to a
  product.

\begin{lemma}  \label{lem:sumED2prod}
For $\lambda$ and $R$ as above we have:
\begin{equation}
\sum_{D \in
  \ED(R/b^a)} \prod_{(i,j) \in R\setminus D} \frac{1}{h_{\lambda}(i,j)} \,=\, \frac{\Phi(a+b+c)\.\Phi(a)\.
  \Phi(b)\.\Phi(c)}{\Phi(a+b)\.\Phi(b+c)\.\Phi(a+c)} \, \prod_{(i,j) \in R/0^cb^{a}}\,\frac{1}{h_{\lambda}(i,j)}\,.
\end{equation}
\end{lemma}

\begin{proof}
We write the sum of excited diagrams as an evaluation of \ts $F_{(b+c)^{a+c}/b^a}({\bf x} \,\vert\, {\bf y})$.
\begin{equation}\label{eq3:pfskewprod}
\sum_{D \in
  \ED(R/b^a)} \prod_{(i,j) \in R\setminus D} \frac{1}{h_{\lambda}(i,j)} \,=\,
\left.  F_{(b+c)^{a+c}/b^{a}}({\bf x} \,\vert\, {\bf y}) \right|_{\substack{x_i = \lambda_i  -i+1 \\y_j =j-\lambda'_j}}
\end{equation}
where $m=(b+c)(a+c)-ba$. Using Theorem~\ref{thm:thickstrip} to obtain
the symmetry of the series $F_{(b+c)^{a+c}/b^a}({\bf x} \,\vert\, {\bf y})$ in~$\bx$~:
\begin{equation} \label{eq4:pfskewprod}
\left. F_{(b+c)^{a+c}/b^{a}}({\bf x} \,\vert\, {\bf y}) \right|_{\substack{x_i= \lambda_i -i+1 \\y_j=j-\lambda'_j}} \, =\,
\sum_{\Theta} \prod_{(i,j) \in \Theta} \frac{1}{h_{\lambda}(i,j)}\,,
\end{equation}
where the sum is over tuples $\ts \Theta:=(\theta_1,\ldots,\theta_c)$ \ts of
nonintersecting paths inside $(b+c)^{a+c}$ with endpoints \ts
$\theta_p:(p,1) \to (a+p,b+c)$. Note that each tuple $\Theta$ has the same
number of cells in each diagonal $i+j=k$. Also, by property \eqref{prop2:lam}
of~$\lambda$, the sum \ts $\lambda_i+\lambda'_j$ \ts is constant
when \ts $i+j$ \ts is constant. Thus each tuple $\Theta$
will have the same contribution to the sum on the RHS of
\eqref{eq4:pfskewprod}, namely
\[
\prod_{(i,j) \in \Theta} \. \frac{1}{h_{\lambda}(i,j)} \,=\,
\prod_{(i,j) \in R/0^cb^{a}} \frac{1}{h_{\lambda}(i,j)} \,.
\]
Lastly, the number of tuples $\Theta$ in \eqref{eq4:pfskewprod} equals the number of
excited diagrams $(b+c)^{a+c}/b^{a}$, given by \eqref{eq:macmahon}, see
Example~\ref{ex:excited_macmahon}.
\end{proof}

By Lemma~\ref{lem:sumED2prod}, \eqref{eq2:pfskewprod} becomes the following product formula for
$f^{\lambda/b^a}/n!$ equivalent to \eqref{eq:prodformula0}~:
\begin{equation} \label{eq6:pfskewprod}
\frac{f^{\lambda/b^a}}{n!} \,=\, \frac{\Phi(a+b+c)\.\Phi(a)\.\Phi(b)\.
\Phi(c)}{\Phi(a+b)\.\Phi(b+c)\.\Phi(a+c)} \,
\prod_{(i,j) \in \lambda/0^cb^{a}} \, \frac{1}{h_{\lambda}(i,j)} \,.
\end{equation}
See Figure~\ref{fig:abcdem-shape} for an illustration of the cells of
$[\lambda]$ whose hook-lengths appear above.
Finally, we carefully rewrite this product in terms of $\Psi(\cdot)$ and
$\Psi^{(m)}(\cdot)$ to obtain \eqref{eq:prodformula}.
\end{proof}

\subsection{Proof of the product formula for skew SSYT} \label{sec:qSSYT}

\begin{proof}[Proof of Theorem~\ref{thm:q_skewprod}]
By \eqref{eq:skewschur} and Lemma~\ref{lem:same_ED}, we have:
\begin{equation} \label{eq2:pf_q_skewprod}
s_{\lambda/b^a}(1,q,q^2,\ldots) \, = \,
\left[ \prod_{(i,j) \in \la\setminus R} \frac{q^{\lambda'_j-i}}{1-q^{h_{\lambda}(i,j)}}\right]\,
\sum_{D \in \ED(R/b^a)} \prod_{(i,j) \in R\setminus D} \frac{q^{\lambda'_j-i}}{1-q^{h_{\lambda}(i,j)}}\,.
\end{equation}

The sum over excited diagrams on the RHS evaluates to a product. We
break the proof into two stages.

\begin{lemma} \label{lem:qsumED2prod}
For $\lambda$ and $R$ as in the previous section, we have:
\begin{multline} \label{eq:lem2qsum}
\sum_{D \in \ED(R/b^a)} \, \prod_{(i,j) \in R\setminus D} \,
\frac{q^{\lambda'_j-i}}{1-q^{h_{\lambda}(i,j)}} \\
= \, q^{C(R/b^a)} \, \left[\prod_{(i,j)\in R/0^cb^a}
  \frac{q^{a+c+e+m(b+c)}}{1-q^{h_{\lambda}(i,j)}}\right]\, \sum_{D \in
  \ED^{\nearrow}(R/0^cb^a)} \. \prod_{(i,j) \in R\setminus D}
                                                    q^{-j(m+1)}\, ,
\end{multline}
where $C(\la/\mu)$ is defined in \eqref{eq:defClamu}.
\end{lemma}

\begin{proof}
We write the sum of excited diagrams as an evaluation of $F_{(b+c)^{a+c}/b^a}({\bf x} \,\vert\, {\bf y})$.
\begin{equation}\label{eq3:pf_q_skewprod}
\sum_{D \in \ED(R/b^a)} \prod_{(i,j) \in R\setminus D}
\frac{q^{\lambda'_j-i}}{1-q^{h_{\lambda}(i,j)}}
\, = \, (-1)^{|R/b^a|} \. q^{C(R/b^{a})}
\left.F_{(b+c)^{a+c}/b^{a}}({\bf x} \,\vert\, {\bf y}) \right|_{\substack{x_i= q^{\lambda_i -i+1}\\y_j=q^{j-\lambda'_j}}}
\end{equation}
By Theorem~\ref{thm:thickstrip}, we have:
\begin{equation} \label{eq4:pf_q_skewprod}
\left. F_{(b+c)^{a+c}/b^{a}}({\bf x} \,\vert\, {\bf y}) \right|_{\substack{x_i= q^{\lambda_i -i+1}\\y_j=q^{j-\lambda'_j}}} \, =
\, \sum_{\Theta} \, (-1)^{|R/b^a|}  \prod_{(i,j) \in \Theta} \. \frac{q^{\lambda'_j -j}}{1-q^{h_{\lambda}(i,j)}}\,.
\end{equation}
Each tuple $\Theta$ has the same number of cells in each diagonal
$i+j=k$.  Also by property \eqref{prop2:lam} of~$\lambda$, the sum \ts $\lambda_i+\lambda'_j$ \ts
is constant when \ts $i+j$ \ts is constant.
Thus each term in the sum corresponding to a  tuple $\Theta$ has the same denominator.
Factoring this contribution out of the sum and using
\ts $\lambda'_j = a+c+e+m(b+c-j)$, gives:
\begin{align*}
\sum_{D \in \ED(R/b^a)} \.\prod_{(i,j) \in R\setminus D} \,
\frac{q^{\lambda'_j-i}}{1-q^{h_{\lambda}(i,j)}} & \, = \,
  q^{C(R/b^a)} \, \left[\prod_{(i,j)\in R/0^cb^a}
  \frac{q^{a+c+e+m(b+c)}}{1-q^{h_{\lambda}(i,j)}}\right]
  \, \sum_{\Theta} \. \prod_{(i,j)\in \Theta} \. q^{-j(m+1)}\..
\end{align*}
Finally, we rewrite the sum over tuples $\Theta$ as a sum over
NE-excited diagrams $\ED^{\nearrow}(R/0^cb^a)$, see Section~\ref{subsec:neED}.
\end{proof}

Next, we prove that the sum over NE-excited diagrams on the RHS of \eqref{eq:lem2qsum} also factors.

\begin{lemma}\label{lem2:qsumED2prod}
In the notation above, we have:
\[
\sum_{D \in
  \ED^{\nearrow}(R/0^cb^a)} \. \prod_{(i,j) \in R\setminus D}
                                                    q^{-j(m+1)} \,=\, q^{-N_2}\prod_{i=1}^a \prod_{j=1}^{b}\prod_{k=1}^c
\frac{1-q^{(m+1)(i+j+k-1)}}{1-q^{(m+1)(i+j+k-2)}}\,,
\]
where $\ts N_2 = (m+1)\left((a+c)\binom{b+c+1}{2} - a\binom{b+1}{2} \right)$.
\end{lemma}

\begin{proof}
We factor out a power of $q^{-N_1}$ where $N_1 = \sum_{(i,j) \in R}
j(m+1)$, so that the weight of each excited diagram $D$ is $\prod_{(i,j) \in
  D} q^{j(m+1)}$.  We have:
\[
\sum_{D \in
  \ED^{\nearrow}(R/0^cb^a)} \. \prod_{(i,j) \in R\setminus D}
                               \. q^{-j(m+1)} \, = \, q^{-N_1} \.
\sum_{D \in \ED^{\nearrow}(R/0^cb^a)} \. \prod_{(i,j)\in D} \. q^{j(m+1)}\,.
\]
Reflecting by the
diagonal, this sum equals the sum over excited diagrams $\ED(R'/a^b)$
where $R'=(a+c)^{b+c}$. We then use \eqref{eq:qsumexcited_macmahon} in
Example~\ref{ex:qexcited_macmahon}, with \ts $q\gets q^{m+1}$, to obtain:
\begin{align*}
\sum_{D \in  \ED^{\nearrow}(R/0^cb^a)} \. \prod_{(i,j) \in R\setminus D}
                                                    q^{-j(m+1)} & \, = \, q^{-N_1}  \. \sum_{D
                                                     \in
                                                     \ED(R'/a^b)}
                                                     \prod_{(i,j)\in
                                                     D} q^{i(m+1)},\\
&= \, q^{-N_1+(m+1)a\binom{b+1}{2}}\.\prod_{i=1}^a \prod_{j=1}^{b}\prod_{k=1}^c
\frac{1-q^{(m+1)(i+j+k-1)}}{1-q^{(m+1)(i+j+k-2)}}\,,
\end{align*}
as desired.
\end{proof}

Combining Lemmas~\ref{lem:qsumED2prod} and~\ref{lem2:qsumED2prod}, we obtain:
\begin{multline} \label{eq4:pf_q_skewprod}
\sum_{D \in \ED(R/b^a)} \. \prod_{(i,j) \in R\setminus D} \,
\frac{q^{\lambda'_j-i}}{1-q^{h_{\lambda}(i,j)}} \\
= \, q^{C(R/b^a)-N_2} \. \left(\prod_{(i,j)\in R/0^cb^a}
  \frac{q^{a+c+e+m(b+c)}}{1-q^{h_{\lambda}(i,j)}}\right)\prod_{i=1}^a \prod_{j=1}^{b}\prod_{k=1}^c \,
\frac{1-q^{(m+1)(i+j+k-1)}}{1-q^{(m+1)(i+j+k-2)}}\,.
\end{multline}
Next we find a simpler expression for the power of $q$ above.

\begin{prop} \label{prop:power}
The power of $q$ on the RHS of \eqref{eq2:pf_q_skewprod} is equal to
\[
\sum_{(i,j) \in R/b^a} (\lambda'_j-i)\..
\]
\end{prop}

\begin{proof}
The term $N_2$ in the power of $q$ can be written as  $\ts N_2= \sum_{(i,j) \in R/b^a}
(m+1)j$. Using this, we have:
\[
C(R/b^a) - N_2 \. = \, \sum_{(i,j) \in R/b^a} \. (-i - mj)\..
\]
Since $R/0^cb^a$ has the same number of cells as $R/b^a$ and
$\lambda'_j = a+c+e+m(b+c-j)$, then
\[
C(R/b^a) - N_2  + (a+c+e+m(b+c))\,\,\bigl|R/0^cb^a\bigr| \, = \,\sum_{(i,j) \in
  R/b^a} \. (\lambda'_j-i)\.,
\]
is the desired degree in the RHS of~\eqref{eq2:pf_q_skewprod}.
\end{proof}

Finally, Theorem~\ref{thm:q_skewprod} follows by substituting
\eqref{eq4:pf_q_skewprod}
in the RHS of \eqref{eq2:pf_q_skewprod}, simplifying the power of
$q$ with Proposition~\ref{prop:power}, and collecting the other powers
of $q$ from the cells $(i,j)$ in $\lambda\setminus R$.
\end{proof}

\bigskip

\section{Excited diagrams and Schubert polynomials} \label{sec:KMY}

In this section we obtain a number of product formulas for principal evaluations
of Schubert polynomials for two permutation families: vexillary and
$321$-avoiding permutations.

\subsection{Vexillary permutations}

Recall from $\S$\ref{sec:perms} that to a vexillary permutation
$w$ we associate a shape $\mu(w)$ contained in a supershape
$\lambda(w)$. A formula for the double Schubert polynomial
$\mathfrak{S}_w({\bf x}; {\bf y})$ of
a vexillary permutation in terms of excited diagrams of the skew shape
$\lambda(w)/\mu(w)$ is given in~\cite{KMY}. This formula was already known in terms of
flagged tableaux~\cite{W} (see Section~\ref{ss:excited-flagged}),
and in terms of {\em flagged Schur functions} \cite{LS1,LS2}.

\begin{thm}[Wachs~\cite{W}, Knutson--Miller--Yong~\cite{KMY}] \label{thm:SchubsKMY}
Let $w$  be a vexillary permutation of shape $\mu$ and supershape
$\lambda$.  Then the double Schubert polynomial of $w$ is equal to
\begin{equation} \label{eq:SchubsKMY}
\mathfrak{S}_{w}({\bf x}; {\bf y}) \, = \, \sum_{D \in \ED(\lambda/\mu)}
\prod_{(i,j) \in D} (x_i - y_j)\ts.
\end{equation}
\end{thm}

\begin{ex}\label{ex:exc-KMY}
For the permutation $w=1432$, we have the shape $\mu = 21$ and the supershape
$\lambda=332$:
\begin{center}
\includegraphics[scale=0.8]{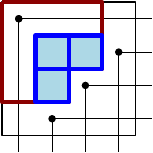}
\end{center}
There are five
excited diagrams in $\ED(332/21)$ (see Example~\ref{ex:excited}), and so
\begin{multline*}
\mathfrak{S}_{1432}({\bf x}; {\bf y}) = (x_1-y_1)(x_1-y_2)(x_2-y_1)  +
(x_1-y_1)(x_1-y_2)(x_3-y_2) +\\ (x_1-y_1)(x_2-y_3)(x_2-y_1) +
(x_1-y_1)(x_2-y_3)(x_3-y_2) + (x_2-y_2)(x_2-y_3)(x_3-y_2).
\end{multline*}
\end{ex}

We have seen that he multivariate sum over excited diagrams on the RHS of \eqref{eq:SchubsKMY} is also an evaluation of a
factorial Schur function.

\begin{cor}
Let $w$ be a vexillary permutation of shape $\mu$ and supershape~$\lambda$,
such that \ $\lambda/\mu \subset d \times (m-d)$ for some $d,m$.  Then:
\[
\mathfrak{S}_w({\bf x}; {\bf y}) \, = \, s^{(d)}_{\mu}({\bf x} \,\vert\, {\bf z^{\<\lambda\>}})\ts.
\]
\end{cor}

\begin{proof}
By \eqref{eq:SchubsKMY}, we have: \ts $\mathfrak{S}_w({\bf x}; {\bf
  y}) = G_{\lambda/\mu}({\bf x} \,\vert\, {\bf y})$. By
Lemma~\ref{lem:key_border_strips}, this is given by an evaluation of a
factorial Schur function.
\end{proof}

Combining this result with the Macdonald identity~\eqref{eq:macdonald}
for single Schubert polynomials gives the following identity for
the principal evaluation $\ts \Ups_w$ \ts of the Schubert polynomial.

\begin{thm} \label{cor:KMY-excited}
Let $w$ be a vexillary permutation of shape $\mu$ and supershape~$\lambda$.  Then:
\[
\Ups_w \, = \, \bigl|\ED(\lambda/\mu)\bigr|\ts.
\]
\end{thm}

\begin{proof}
This follows directly from~\eqref{eq:SchubsKMY} by setting $x_i=1$ and
$y_i=0$ for all~$i$, and from the Macdonald identity~\eqref{eq:macdonald}.
\end{proof}

\begin{ex}  Continuing the previous Example~\ref{ex:exc-KMY},
the reduced words for $w=1432$ are $(2,3,2)$ and $(3,2,3)$.  We indeed have:
\[
\Ups_{1432} \, =  \, \mathfrak{S}_{1432}(1,1,1) \, = \,
\frac{1}{3!}\left ( 2\cdot 3 \cdot 2 \ts + \ts 3\cdot 2 \cdot 3\right)
\. = \. 5\. = \. \ts |\ED(332/21)|\ts.
\]
\end{ex}

Theorem~\ref{cor:KMY-excited} generalizes an identity in~\cite[Thm.~2.1]{FK}
from dominant permutations
(avoiding $132$) to vexillary permutations. To state their result we
need the following notation. Given a partition
$\mu$ and $c\in \nn$, let $RPP_{\mu}(c)$ be the
number of reverse plane partitions of shape $\mu$ with entries
$\leq c$. Let \ts $1^c \di \mu = (\mu_1+c)^c(c+\mu_1)(c+\mu_2)\ldots$

\begin{prop} \label{prop:EDdominant}
For the shape $\mu$ of a dominant permutation,  we have:
\[
|\RPP_{\mu}(c)| \, = \, \bigl|\ED\bigl((1^c \di \mu)/\mu\bigr)\bigr|\..
\]
\end{prop}

\begin{proof}
By Proposition~\ref{prop:flagged}, the RHS is equal to the number of
SSYT of shape~$\mu$ with entries in row~$i$ at most~$c+i$. By
subtracting $i$ from the entries in row $i$, such SSYT are in
correspondence with RPP of shape $\mu$ with entries $\leq c$.
\end{proof}

For the rest of the section we will use the following notation
for the principal evaluation:
$$\Ups_w(c) \, := \, \Ups_{1^c \times w} \, \, = \, \frac{1}{\ell(w)!}\,
\sum_{(r_1,\ldots,r_{\ell}) \in R(w)} \. (c+r_1)\cdots
(c+r_{\ell})\..$$

\begin{cor}[Fomin--Kirillov \cite{FK}] \label{cor:FK-dominant}
For a dominant permutation $w$ of shape $\mu$ we have:
\[
\Ups_w(c) \. = \.  |\RPP_{\mu}(c)|\ts.
\]
\end{cor}

\begin{proof}
The permutation $\ts 1^c \times w=(1,2,\ldots,c, c+w_1,c+w_2,\ldots)$ \ts is a vexillary permutation of
shape $\mu$ and supershape \ts $\lambda = 1^c \di \mu$. Also, the reduced words of $1^c \times w$ are of
the form \ts $(c+r_1,\ldots,c+r_{\ell})$, where \ts $(r_1,\ldots,r_{\ell})$ \ts is
a reduced word of~$w$. We then apply
Theorem~\ref{cor:KMY-excited} and Proposition~\ref{prop:EDdominant} to obtain the result.
\end{proof}

\subsection{Product formulas for Macdonald type sums} \label{ss:kmy-mac}
As special cases of  Theorem~\ref{cor:KMY-excited} we obtain two
identities from \cite{FK} for two families of dominant permutations,
followed by new identities for families of vexillary permutation.
See Figure~\ref{fig:vexillaryexs} for illustrations of some of these families.

\begin{cor}[staircase~\cite{FK}] \label{cor:kmy-stair}
For the permutation \ts $w_0 = n\ldots 21$, we have:
\[
\Ups_{w_0}(c) \, = \, \frac{\Phi(2c+2n-1)\. \Phi(n)\. \cdot \. \Lam(2c+1)
\ts \Lam(2n-1)}{\Phi(n+2c)\. \Phi(2n-1)\. \cdot \. \Lam(2c+2n-1)}\,.
\]
\end{cor}

\begin{proof}
The longest element $w_0$ is the dominant permutation with
shape $\mu = \delta_n :=(n-1,\ldots,2,1)$. The result follows by
Corollary~\ref{cor:FK-dominant} and Proctor's formula~\cite{Pr}:
\begin{equation}\label{eq:proctor}
\Ups_{w_0}(c) \, \. = \, \. |\RPP_{\delta_n}(c)| \, \.  = \,
\prod_{1\leq i<j \leq n} \frac{2c+i+j-1}{i+j-1}\,,
\end{equation}
written in terms of superfactorials.
\end{proof}

Note that the case $c=1$ above gives $\Ups_{w_0}(1) =
\frac{1}{n+1}\binom{2n}{n}$; see \cite{Woo} for several proofs of this
case.

\begin{cor}[box formula~\cite{FK}] \label{cor:kmy-box}
Consider the permutation \ts $u(a,b)$ \ts defined as
$$u(a,b)\. := \. b(b+1)\cdots (a+b) \ts 12\ldots (b-1)\ts.
$$
Then we have:
\[
\Ups_{u(a,b)}(c) \, = \,
\frac{\Phi(a+b+c)\ts\Phi(a)\ts\Phi(b)\ts\Phi(c)}{\Phi(a+b)\ts\Phi(b+c)\ts\Phi(a+c)}\,.
\]
\end{cor}

\begin{proof}
The permutation $u(a,b)$ is a dominant permutation with shape~$b^a$.
The result follows by Corollary~\ref{cor:FK-dominant} and the
MacMahon box formula \eqref{eq:macmahon}.
\end{proof}

For the rest of this subsection, we consider examples that are
vexillary but not dominant. These results partially answer
a question in \cite[Open Problem 2]{BHY}.  First, we give a family of permutations $z(a)$ with principal evaluation
given by a power of $2$.

\begin{cor} \label{cor:pow2case}
Consider the permutation $z(a) := 135\ldots (2a-1)\,246\cdots (2a)$.
Then we have 
\[
\Ups_{z(a)} \, = \, 2^{\binom{a}{2}}. 
\]
\end{cor}

\begin{proof}
The vexillary (actually, Grassmannian) permutation $z(a)$ has shape $\mu=\delta_{a}$ and
supershape $\lambda=(2a-2)^a$. By Proposition~\ref{prop:flagged}, the
number of excited diagrams equals the number of SSYT of shape $\mu$
with entries at most $a$. This number is given by the hook-content
formula
\[
s_{\delta_{a}}(1^a) \,=\, \prod_{(i,j) \in [\delta_a]}
\. \frac{a+j-i}{h_{\delta_a}(i,j)} \, = \, \prod_{i=1}^{a-1}\prod_{j=1}^{a-i} \frac{a+j-i}{2(a-i-j)+1}
\]
A direct calculation gives the desired formula (see e.g.~\cite[Prop. 10.3]{MPP3}).
\end{proof}

Second, we restate
Corollary~\ref{cor:2413} as follows:

\begin{cor}[$2413 \otimes 1^a$ case]
\label{cor:2413-rest}
Consider the permutation $v(a):=2413 \otimes 1^a$. Then, for all $c\geq a$, we have:
\[
\Ups_{v(a)}(c) \, = \, \frac{\Phi(4a+c)\.\Phi(c)\.
\Phi(a)^4\.\Phi(3a)^2}{\Phi(3a+c)\.\Phi(a+c) \.\Phi(2a)^2\.\Phi(4a)}\,.
\]
\end{cor}

\begin{proof}
The vexillary permutation $1^c \times v(a)$ has length $3a^2$, shape $\mu =
(2a)^a a^a$ and supershape $\la=(c+3a)^{c+2a}$. The reduced words of
$1^c \times v(a)$ are obtained from those of $v(a)$ after shifting by~$c$.
By Theorem~\ref{cor:KMY-excited} for $1^c \times v(a)$, we have:
\[
\Ups_{v(a)}(c) \, = \, \bigl|\ED(\lambda/\mu)\bigr|\..
\]
By Proposition~\ref{prop:flagged}, the number of excited diagrams
equals the number of SSYT of shape $\mu$ with entries at most $2a+c$.
This number is given by the hook-content formula
\[
s_{\mu}(1^{2a+c}) \, = \,\prod_{(i,j)\in [\mu]} \. \frac{2a+c + j-i}{h_{\mu}(i,j)}\..
\]
This product can be written in terms of superfactorials as stated.
\end{proof}

Next, we consider whether the skew shapes in the first part of the
paper come from vexillary permutations. We failed to obtain the
skew shape $\LA(a,b,c,d,e,0)$ this way, but the next vexillary
permutation yields a shape similar to $\LA(a,a,c,a,a,0)$.

\begin{cor}
For the vexillary permutation
\begin{equation} \label{eq:schubperm}
w(a):= (a+1,a+2,\ldots,2a-1,\, 2a+1, \,
1,2,\ldots,a-1,\, 2a,\, {a}).
\end{equation}
we have:
\[
\Ups_{w(a)}(c) \, = \, \frac{\Phi(2a+c)\ts\Phi(a)^2\ts\Phi(c)}{\Phi(a+c)^2\ts\Phi(2a-1)}
\,
\left[\frac{a\ts (2a+c)\ts (2ac+4a^2-1)}{2\ts (4a^2-1)}
\right]\ts.
\]
\end{cor}

\begin{proof}
The vexillary permutation $1^c \times w(a)$ has length
$2+a^2$, shape $\mu=(a+1)a^{a-1}1$ and supershape $\lambda
= (2a+c)^{c+a}(a+c)^a$, see
Figure~\ref{fig:vexillaryexs}. The reduced words of
$1^c \times w(a)$ are obtained from those of $w(a)$ by shifting by~$c$.
By Theorem~\ref{cor:KMY-excited} for \ts $1^c \times w(a)$, we have:
\[
\Ups_{w(a)}(c) \, = \,  |\ED(\lambda/\mu)|\..
\]
By Proposition~\ref{prop:flagged}, the number of excited diagrams of
shape $\lambda/\mu$ is equal to the number of SSYT of shape $\mu$ with entries in the top $a$
rows at most $a+c$ and the single box in the $a+1$ row at most
$2a+c$. Depending on the value of this single box, whether it is at
most $a+c$ or between $a+c+1$ and $2a+c$, this number equals the sum of two specializations of Schur functions:
\[
|\ED(\lambda/\mu)| \,= \, s_{(\nu,1)}(1^{a+c}) \. + \. a\cdot s_{\nu}(1^{a+c})\.,
\]
where  $\nu = (a+1)a^{a-1}$. Using the hook-content formula, this
number can be written in terms of superfactorials as in the corollary.
\end{proof}

\begin{figure}
\includegraphics[scale=0.7]{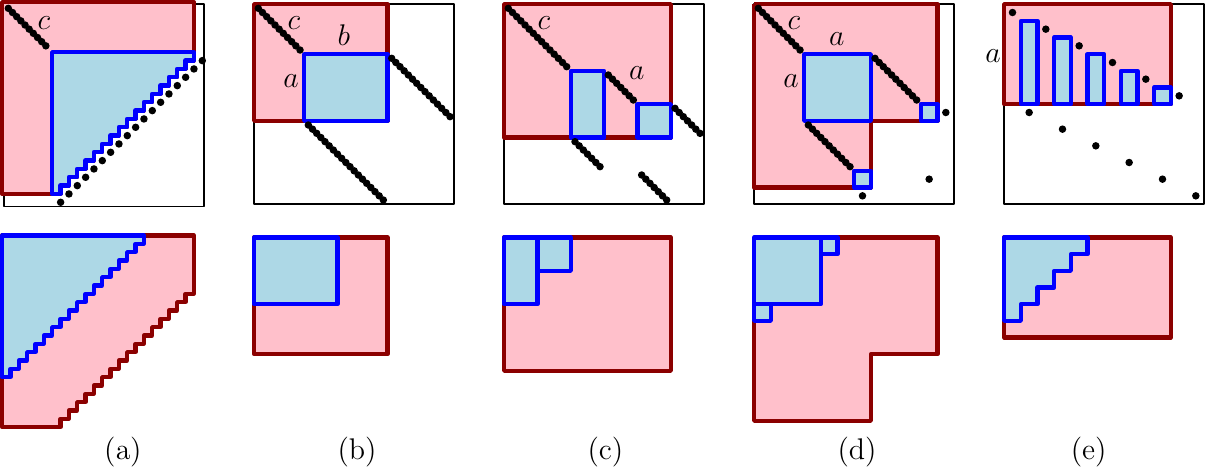}
\label{fig:reduced_words}
\caption{The diagram (top) and skew shape (bottom) of the vexillary permutations
  $1^c\times w$ where $w= w_0, u(a,b), v(a)$, $w(a)$, and $z(a)$ respectively.}
\label{fig:vexillaryexs}
\end{figure}

\subsection{$321$-avoiding permutations} \label{ss:KMY-321}
Recall from Section~\ref{sec:perms} that the diagram of a
$321$-avoiding permutation is, up to removing empty rows and columns
and flipping columns, the diagram of a skew shape $\lambda/\mu$. By
Theorem~\ref{thm:shape2perm}(stating \cite[Prop 2.2]{BJS}) we can realize every skew
shape $\lambda/\mu$ as the diagram of the $321$-avoiding permutation
given by the reduced word $\rw(\lambda/\mu)$. The map from shapes to
permutations is outlined in Section~\ref{sec:perms}.

\begin{thm} \label{prob:skew321}
Let $w$ be  a $321$-avoiding permutation.  Then its diagram gives a
skew shape~$\lambda/\mu$. Conversely, every skew shape $\lambda/\mu$ can
be realized from the diagram of a $321$-avoiding permutation $w$. In both
cases, we have:
\[
\Ups_w \, = \, \frac{1}{\ell!} \, \. r_1\. \cdots \. r_{\ell} \,
f^{\lambda/\mu}\.,
\]
where $\ts\ell = |\lambda/\mu|\ts $ and $\ts (r_1,\ldots,r_{\ell})\ts $
is a reduced word of~$w$.
\end{thm}

\begin{proof}
The fact that diagrams of $321$-avoiding permutations yield skew
shapes and its converse are explained in Section~\ref{sec:perms}.

Assume that the $321$-avoiding permutation has skew shape \ts $\skewsh(w)
= \lambda/\mu$. The reduced words of a $321$-avoiding permutation are obtained from
one another by only using commutation relations $s_is_j = s_js_i$ for
$|i-j|>1$ \cite[Thm.~2.1]{BJS}. Thus, all reduced words $(r_1,\ldots,r_\ell)$
of~$w$ have the same product $\ts r_1 \ts\cdots\ts r_{\ell}$. Also, the number
of reduced words of~$w$ equals \ts $f^{\lambda/\mu}$, see \cite[Cor.~2.1]{BJS}.
The result then follows by using
these two facts and Macdonald's identity~\eqref{eq:macdonald}.
\end{proof}

As an illustration we obtain permutations such that
$\Ups_w$ give double factorials and  {\em Euler numbers}.

\begin{cor}\label{cor:ups-zigzag}
For the permutations $w=2143\cdots (2n)(2n-1)$ and $w\otimes 1^a$,  we
have:
\begin{align*}
\Ups_w &\, = \, (2n-1)!!  \qquad \text{and} \qquad
\Ups_{w\ts \otimes \ts 1^a} & \, = \, \ts \frac{\Phi(2n a)\. \Phi(a)^{2n-2}}{\Phi(2a)^n} \, \left[\.\prod_{k=1}^{n-1}
                      \frac{\Phi(2k a)}{\Phi\bigl((2k+1) a\bigr)}\ts\right]^2\ts.
\end{align*}
\end{cor}

\begin{proof}
The number of SYT of the diagonal shape $\delta_{n+1}/\delta_n$ is $n!$. By the construction from
Theorem~\ref{thm:shape2perm}, from this shape we read off the reduced
word
\[
\rw(\delta_{n+1}/\delta_n) \. = \. (1,3,5,\ldots,2n-1)\ts,
\]
defining the permutation $w$. See Figure~\ref{fig:fact} for
an example. The product of the
entries of this reduced word is $(2n-1)!!$ \ts  The result then follows by
Theorem~\ref{prob:skew321}. The second formula comes from the $321$-avoiding
permutation $w\otimes 1^a$ whose skew shape consists of $n$ disjoint
$a\times a$ blocks.
\end{proof}

Let $\Alt(n)= \{\si(1)<\si(2)>\si(3)<\si(4)>\ldots\} \ssu \SS_n$ be the set of
{\em alternating permutations}.
The number $E_n=|\Alt(n)|$ is the $n$-th {\em Euler number} (see \cite[\href{http://oeis.org/A000111}{A000111}]{OEIS}), with the
generating function
\begin{equation}\label{eq:tan-sec}
\sum_{n=0}^\infty \. E_n \. \frac{x^n}{n!} \,\. = \,\, \tan(x) \ts + \ts \sec(x)\ts.
\end{equation}

Let $x(n)$ be a permutation with reduced word corresponding to the \emph{zigzag shape}
$$
\rw(\delta_{n+2}/\delta_n)\. = \. (2,1,4,3,\ldots,2n,2n-1,2n+1)\ts.
$$
Similarly, define $y(n)$ and $z(n)$ to be the permutations with reduced words
corresponding to shapes \ts $(n+1)^2n(n-1)\ldots 2/\delta_n$ \ts and \ts
$(n+2)^3(n+1)n\ldots 3/\delta_n$, respectively.

\begin{cor}\label{cor:ups-xyz}
For the permutations $x(n)$, $y(n)$, and $z(n)$ defined above, we have:
\[
\Ups_{x(n)} \. = \. E_{2n+1}, \quad  \Ups_{y(n)} \. =
\. \frac{n! \. E_{2n+1}}{2^{n}}, \quad \Ups_{z(n)}
\.=\. \frac{(n+1)\ts (2n+3)! \. E_{2n+1}^2}{n!\. 2^{5n+1} \ts \bigl(2^{2n+2}-1\bigr)}.
\]
\end{cor}

\begin{proof}
The number of SYT of the zigzag shape $\delta_{n+2}/\delta_n$ is
given by the Euler number $E_{2n+1}$. By the construction from
Theorem~\ref{thm:shape2perm}, from this shape we read off the reduced
word
\[
\rw(\delta_{n+2}/\delta_n) =  (2,1,4,3,\ldots,2n,2n-1,2n+1),
\]
defining the permutation $x(n)$. The product of the
entries of this reduced word is $(2n+1)!$ \ts The first equality then follows by
Theorem~\ref{prob:skew321}.

The second and third equalities follow by a similar argument for
the $3$-zigzag and $5$-zigzag shape, respectively, whose number of
SYT is given by \cite[Thm.~1]{BR}.  We omit the easy details. See
Figure~\ref{fig:euler}, \ref{fig:euler3} and~\ref{fig:euler5} for examples.
\end{proof}

We also obtain a
family of $321$-avoiding permutations $w$ that yield the skew shapes
from Section~\ref{sec:skewprod} with product
$\LA(a,b,c,d,e,f)$. Then by theorems~\ref{thm:skewprod} and~\ref{prob:skew321},
for such permutations, $\Ups_w$ is given by
a product formula. We illustrate this for the cases
$\LA(a,a,a,a,a,0)$ and $\LA(a,a,a,1,1,1)$. See
Figure~\ref{fig:1stshape},\ref{fig:2ndshape} for examples.

\begin{figure}
\includegraphics[scale=0.8]{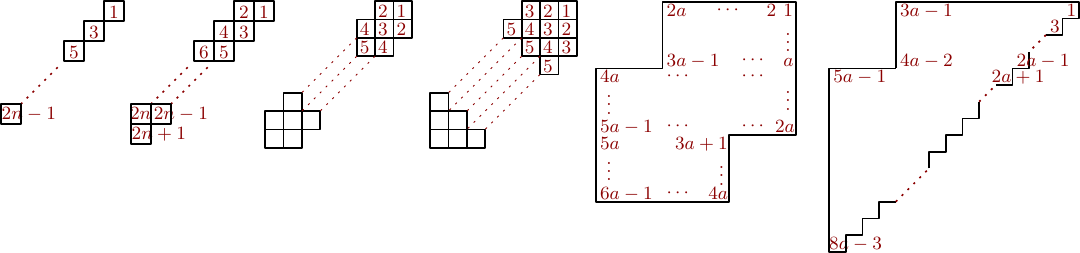}
\caption{The reduced words $\rw(\lambda/\mu)$ of the skew shapes
  $\delta_{n+1}/\delta_n$, the zigzag $\delta_{n+2}/\delta_n$, $3$-zigzag, $5$-zigzag, $(3a)^{2a}(2a)^a/a^a$, and
  $\delta_{4a}/a^a$.}
\label{fig:rwexskew}
\end{figure}

We now restate Corollary~\ref{cor:351624} in the notation above.

\begin{cor}[shape $\ts (3a)^{2a}(2a)^a/a^a$]  \label{cor:351624-main}
For the permutation $\ts s(a):=351624 \otimes 1^a$, we have:
\[
\Ups_{s(a)} \, =
 \, \frac{\Phi(a)^5\. \Phi(3a)^2\. \Phi(5a)}{\Phi(2a)^4\. \Phi(4a)^2} \,.
\]
\end{cor}

\begin{proof}
The reading word associated to the shape $3^22/1$ is
$(2,1,4,3,2,5,4)$ which defines the permutation $351624$. Similarly,
the shape $\ts (3a)^{2a}(2a)^a/a^a$ \ts yields a reduced word \ts
$(r_1,\ldots,r_{7a^2})$, defining the $321$-avoiding permutation \ts
$s(a) = 351624 \otimes 1^a$. By Theorem~\ref{prob:skew321}, we have:
\[
\Ups_{s(a)} \, = \, \frac{r_1\cdots r_{7a^2}}{(7a^2)!} \, f^{(3a)^{2a}(2a)^a/a^a}\,.
\]
The result now follows by writing the product of the entries of
the reduced word as
\[
r_1\cdots r_{7a^2} \, = \, \frac{\Phi(3a)^2\. \Phi(6a)}{\Phi(2a)^2\. \Phi(4a)^2}
\]
(see Figure~\ref{fig:rwexskew}). Now use Corollary~\ref{cor:abcde-shape}
to write the number of SYT as
\begin{equation}\label{eq:f-new}
f^{(3a)^{2a}(2a)^a/a^a} \, = \, \frac{(7a^2)! \, \Phi(a)^5 \.\Phi(5a)}{\Phi(2a)^2\. \Phi(6a)}\,,
\end{equation}
and the result follows.
\end{proof}

\begin{cor}[shape $\ts \delta_{4a}/a^a$]  \label{cor:kmy-dewitt}
Let $t(a)$ be the permutation of size $(8a-2)$ obtained from the reading
word of the skew shape $\ts\delta_{4a}/a^a$. Then:
\[
\Ups_{t(a)} \,  =\, \frac{\Phi(a)^3 \.
  \Phi(3a)\.\Phi(4a-1)\.\Phi(8a-2)\.\cdot \. \Psi(a)\.\Psi(3a)}{\Phi(2a)^2\.\Phi(3a-1)\.\Phi(5a-1)
  \.\cdot \. \Psi(2a)^2\. \Psi(4a)\. \cdot \. \Lam(8a-2)} \,.
\]
\end{cor}

\begin{proof}
The reduced word $\rw(\delta_{4a}/a^a)$ defines the permutation $t(a)$. By Theorem~\ref{prob:skew321} we have:
\[
\Ups_{t(a)} \, = \, \frac{r_1\cdots r_{\ell}}{\ell!} \, f^{\delta_{4a+1}/a^a}.
\]
We can write the product of the entries of
the reduced word as
\[
r_1r_2\cdots r_{\ell} \, = \, \frac{\Phi(4a-1)\.\Phi(8a-2)}{\Phi(3a-1)\. \Phi(5a-1) \. \Lam(8a-2)}\,,
\]
(see Figure~\ref{fig:rwexskew}). On the other hand, Corollary~\ref{cor:abc-shape} gives:
\begin{equation}\label{eq:dewitt-prod}
f^{\delta_{4a+1}/a^a} \, = \, \frac{\ell! \.\cdot \. \Phi(a)^3\.\Phi(3a)\.\Psi(a)\.\Psi(3a)}{\Phi(2a)^2
 \. \Psi(2a)^2\. \Psi(4a)}\,,
\end{equation}
where $\ts \ell = \binom{4a}{2}-a^2$.  Combining these formulas, we obtain the result.
\end{proof}

\begin{figure}
\subfigure[]{
\includegraphics[scale=0.7]{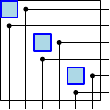}
\label{fig:fact}
}
\subfigure[]{
\includegraphics[scale=0.7]{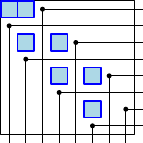}
\label{fig:euler}
}
\subfigure[]{
\includegraphics[scale=0.7]{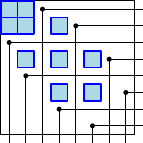}
\label{fig:euler3}
}
\subfigure[]{
\includegraphics[scale=0.7]{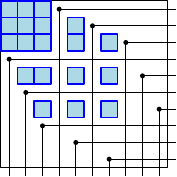}
\label{fig:euler5}
}
\subfigure[]{
\includegraphics[scale=0.7]{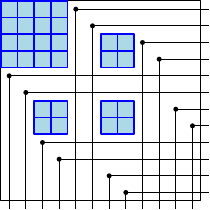}
\label{fig:1stshape}
}
\subfigure[]{
\includegraphics[scale=0.7]{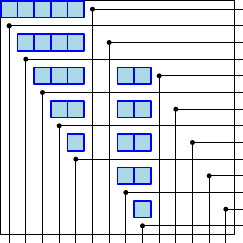}
\label{fig:2ndshape}
}
\caption{The diagram of the $321$-avoiding permutations \ts $w=214365$, $x(3)=31527486$,
  $y(3)$, $z(3)$, $s(2)=351624\otimes 1^2$ and $t(2)$, with skew shapes \ts
  $\delta_4/\delta_3$, $\delta_5/\delta_3$, $4^232/\delta_3$, $5^343/\delta_3$,
  $6^44^2/2^2$ and $\delta_8/2^2$ respectively.}
\label{ex:skew321}
\end{figure}

\subsection{Conjectural formula}
The number of SYT of the skew shape of the vexillary permutation
\ts $1^c \times w(a)$ \ts defined in~\eqref{eq:schubperm} appears to have the
following formula similar to Corollary~\ref{cor:abcde-shape}
when $a=b=d=e$.

\begin{conj}[joint with C.~Krattenthaler] \label{conj:schubskewshape}
Let $\ts \la = (2a+c)^{c+a}(a+c)^a$, $\ts \mu =(a+1)a^{a-1}1$.  Then:
\begin{equation}
\label{eq:conjskewshape}
f^{\lambda/\mu} \, = \, n! \, \. \frac{\Phi(a)^4\.\Phi(c) \. \Phi(4a+c)}{\Phi(2a)^2\. \Phi(4a+2c)}
\, \left[\frac{a^2 \ts \bigl((2a^2+4ac+c^2)^2-a^2\bigr)}{4a^2-1}\right]\ts.
\end{equation}
where $\ts n=|\la/\mu|=(2a+c)^2-2a^2-2$.
\end{conj}

\begin{rem}
For $a=c$, formula~\eqref{eq:conjskewshape} for the
number of SYT of shape $(3a)^{2a}(2a)^a/(a+1)a^{a-1}1$ is
\begin{equation}\label{eq:kratt}
f^{\la/\mu} \, = \, n! \,\. \frac{\Phi(a)^5 \.\Phi(5a)}{\Phi(2a)^2\. \Phi(6a)}\,
\left[\frac{(49a^2-1)\ts a^4}{4a^2-1}\right]\ts.
\end{equation}
This formula was suggested by Christian Krattenthaler,\footnote{Personal
communication.} based on computational data and was a precursor 
of the conjecture above.\footnote{Both Krattenthaler's formula and 
Conjecture~\ref{conj:schubskewshape} were recently established in~\cite{KY}.}
Note the close resemblance of~\eqref{eq:kratt} and~\eqref{eq:f-new},
which are the same up to a polynomial factor.  This suggests that
perhaps there is a common generalization.
\end{rem}

\bigskip

\section{Asymptotic applications}\label{sec:asy}

\subsection{Number of SYT}\label{ss:asy-syt}
In~\cite{MPP3}, we prove that for a sequence \ts $\ov \pi = \bigl\{\pi^{(n)}\bigr\}$ \ts of
strongly stable skew shapes $\pi^{(n)} = \la^{(n)}/\mu^{(n)}$, $\bigl|\pi^{(n)}\bigr| =n$,
we have:
$$
(\ast) \qquad
\log \, \bigl|\SYT\bigl(\pi^{(n)}\bigr)\bigr| \, =  \, \.
\frac{1}{2}\. n \ts \log n \. + \. O(n)\ts.
$$
Furthermore, we conjecture that
$$
(\ast\ast) \qquad
\log \, \bigl|\SYT\bigl(\pi^{(n)}\bigr)\bigr| \, =  \, \.
\frac{1}{2}\. n \ts \log n \. + \. c \ts n \. + \. o(n)\ts,
$$
for some constant $c=c(\ov\pi)$.\footnote{This conjecture was recently 
established in~\cite{MPT}}  Here by the \emph{strongly stable skew shape}
we mean a sharp convergence to the limit shape of the Young diagrams
of $\pi^{(n)}$ under scaling $1/\sqrt{n}$, as $n\to \infty$;
we refer to~\cite{MPP3} for details.\footnote{See \cite{DF} for related
results for other growth regimes.}

\smallskip

Until this paper, the exact value of $c(\ov \pi)$ was possible to compute only for the
usual and shifted shapes.  Here we have a new family of shapes where this is possible.

\begin{thm} \label{thm:asy-limit}
Fix $\al,\be,\ga, \de, \ep \ge 0$, $m \in \nn$, and let
$$
\pi^{(n)} \. = \. \LA\bigl(\lfloor\al n\rfloor, \lfloor\be n\rfloor, \lfloor\ga n\rfloor,
\lfloor\de n\rfloor, \lfloor\ep n\rfloor, m\bigr)\ts.
$$
Then the asymptotic formula $(\ast\ast)$ holds for some $c=c(\al,\be,\ga, \de, \ep, m)$.
\end{thm}

The proof of the theorem is straightforward from the product formula in
Theorem~\ref{thm:skewprod} and asymptotic formulas in~$\S$\ref{ss:not-asy}.
We omit the details.

\begin{ex}
Let $\pi = \LA(a,a,a,a,a,0)$.  Then $|\pi|=7a^2$ and by~\eqref{eq:f-new}, we have:
$$\aligned
\log \ts f^\pi \ & = \  \log \, \frac{(7a^2)! \, \. \Phi(a)^5 \.\Phi(5a)}{\Phi(2a)^2\. \Phi(6a)}  \ =
\  7\ts a^2\log a \ + \\
& \qquad  + \left(\frac{7}2 \. - \. 22 \ts \log 2 \. - \. 18 \ts \log 3 \. + \.
\frac{25}2 \.\log 5 \. + \. 7 \ts \log 7\right) a^2 + \. O(a\log a)\ts,
\endaligned
$$
The sum in parentheses shifted by $(7/2)\log7$ is the exact value of the constant $c(1,1,1,1,1,0)$
 as in the theorem.
\end{ex}

\medskip

\subsection{Principal Schubert evaluations}\label{ss:asy-schubert}
In recent years, there has been some interest in the asymptotics
of the principle evaluation $\ts \Ups_w =  \mathfrak{S}_w(1,\ldots,1)$.
Notably, Stanley~\cite{St2} defined
$$
u(n) \, := \, \max_{w\in S_n} \. \Ups_w
$$
and observed that
$$(\divideontimes) \qquad \quad
\frac{1}{4} \, \le \,  \liminf_{n\to\infty} \. \frac{\log_2 u(n)}{n^2} \, \le \, \limsup_{n\to\infty} \.
\frac{\log_2 u(n)}{n^2} \, \le \, \frac12\,.
$$
Stanley also suggested existence of the limit of \ts $\frac{1}{n^2}\log_2 u(n)$, and
that it is achieved on a certain ``limit shape''. Below we apply our product formulas
to obtain asymptotics of $\Ups_w$ for some families of~$w$.

\begin{prop}[zigzag permutations] \label{prop:asy-zigzag}
For permutations $w\in S_{2n}$ as in Corollary~\ref{cor:ups-zigzag},
$x(n), y(n)\in S_{2n+2}$, and $z(n)\in S_{2n+4}$ as in Corollary~\ref{cor:ups-xyz},
we have: \, $\log \Ups_{w} \ts  = \ts \Theta(n\ts \log n)$, \.
$\log \Ups_{x(n)}\ts = \ts \Theta(n\ts \log n)$, \. $\log \Ups_{y(n)}\ts = \ts \Theta(n\ts \log n)$,
\ts and \. $\log \Ups_{z(n)} \ts = \ts \Theta(n\ts \log n)$.
\end{prop}

The proof follows immediately from the product formulas in corollaries
as above, the asymptotics of $(2n-1)!!$ and of the Euler numbers:
$$
E_n \, \sim \, n!\. \left(\frac{2}{\pi}\right)^n\frac{4}{\pi}\.
\bigl(1+ o(1)\bigr) \quad \text{as} \ \ n\to \infty\ts,
$$
(see e.g.~\cite{FS,Stanley_SurveyAP}).

\smallskip

\begin{prop}[Macdonald permutations] \label{prop:asy-mac-perm}
Consider permutations $w_0\in S_{2k}$, $\wh w_0=1^k \times w_0\in S_{3k}$
as in Corollary~\ref{cor:kmy-stair}, $u(k,k) \in S_{2k}$, $\wh
u(k,k)=1^k \times u(k,k) \in S_{3k}$
as in Corollary~\ref{cor:kmy-box}.
Then we have:
$$
\log \Ups_{\wh w_0} \, \sim \, 2\ts (\log C) \ts k^2\ts, \quad \text{and} \quad \log \Ups_{\wh u(k,k)} \, \sim \, (\log C)\ts k^2\ts,
\quad \text{where} \quad C\. =\, \frac{3^{9/2}}{2^6}\,.
$$
\end{prop}

The proof is straightforward again and combines the corollaries in the proposition
with the asymptotic formulas for~$\Phi(n)$ and$~\Lam(n)$. In fact, the constant~$C$ is the base
of exponent in the symmetric case of the box formula~\eqref{eq:macmahon} for $|\PP(n,n,n)|$,
see \cite[\href{http://oeis.org/A008793}{A008793}]{OEIS}.

From here, for $n=3k$ and $\wh w_0 \in S_n$ as above, we have:
$$
\frac{\log_2 \Ups_{\wh w_0}}{n^2} \, \. \to \,\. \frac{2}{9} \. \log_2 C \, \approx \, 0.251629\,
\quad \text{as} \ \ n\ts =\ts 3k \. \to \ts\infty\ts.
$$
This is a mild improvement over Stanley's lower 
bound~$(\divideontimes)$.\footnote{In~\cite{MPP-schubert} we improve this lower bound 
to about~$0.293$ and prove that this is maximal for principal Schubert evaluations 
of layered permutations.} 

\smallskip

Finally, for comparison, we obtain similar asymptotics for three more
families of \emph{stable permutations}, i.e.\ permutations whose
diagrams have stable shape (cf.~\cite{MPP3}).

\begin{prop}[stable permutations] \label{prop:asy-mac-perm}
Let $v(a)=2413 \otimes 1^a \in S_{4a}$ as in Corollary~\ref{cor:2413-rest},
$s(a)=351624 \otimes 1^a \in S_{6a}$ as in Corollary~\ref{cor:351624},
and $t(a)\in S_{8a-2}$ as in Corollary~\ref{cor:kmy-dewitt}.  Then we have:
\, $\log \Ups_{v(a)} \ts = \ts \Theta(a^2)$, \.
$\log \Ups_{s(a)} \ts = \ts \Theta(a^2)$, \ts and \.
$\log \Ups_{t(a)} \ts = \ts \Theta(a^2)$.
\end{prop}

We omit the proof which is again a straightforward calculation.  To compare this
with Stanley's bound, take the following example:
$$
\frac{\log_2 \Ups_{s(a)}}{n^2} \, \. \to \,
\left(\frac14 \ts \log_2 3 \. + \. \frac{25}{72} \ts\log_2 5 \. - \. \frac{10}9\right)
\, \approx \, 0.091354\, \quad \text{as} \ \ n\ts =\ts  6\ts a \. \to \ts \infty\ts.
$$
This suggests that perhaps every family \ts $\{w\in S_n\}$ \ts of stable
permutations satisfies \ts $\Ups_w = \exp \Theta(n^2)$.  On the other hand, as suggested by
the exact computations in~\cite{MeS,St2}, it is likely that that the maximum
of $\Ups_w$ is achieved on a smaller class of \emph{stable Richardson permutations}.

\bigskip

\section{Lozenge tilings with multivariate weights} \label{sec:lozenge}

\noindent
In this section we study lozenge tilings of regions in the triangular grid.
On a technical level, we show how the multivariate sums
\ts $G_{\lambda/\mu}({\bf x} \,\vert\, {\bf y})$ \ts
appear in the context of lozenge tilings.


\subsection{Combinatorics of lozenge tilings}
Let us show how excited diagrams can be interpreted as lozenge
tilings of certain shapes (plane partitions) with multivariate local
weights. As a consequence, the multivariate sum  $G_{\lambda/\mu}({\bf x} \,\vert\,
{\bf y})$ of excited diagrams is a partition function of such lozenge
tilings, and by Lemma~\ref{lem:key_border_strips} and the definition of factorial Schur
functions it can be  computed as a determinant.

\begin{figure}[hbt]
\includegraphics[scale=0.57]{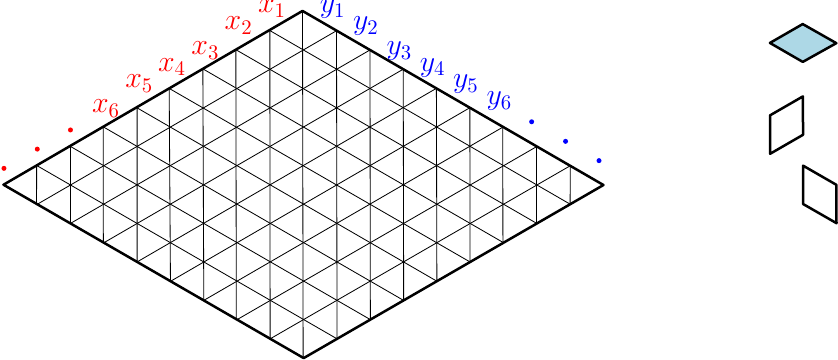}
\caption{Triangular grid in the plane with axes $\bx$ and~$\by$.
The grid is tiled with three types of lozenges. Horizontal
lozenges at position $(i,j)$ have a weight $x_i-y_j$
(see Figure~\ref{fig:excited2lozenge}). }
\label{fig:triangular_grid}
\end{figure}

Consider the triangular grid in the plane where we identify two of the
axes as $x$ and $y$, see Figure~\ref{fig:triangular_grid}. Adjacent triangles can be
paired into  \emph{lozenges}, which can tile certain
prescribed regions in the plane. The lozenges whose long axis is
horizontal, are called \emph{horizontal lozenges}, and are colored in blue in
the picture. Each of these lozenges is assigned a local weight,
depending on its position with respect to the $x$ and $y$ axes.
More precisely, the weight of the lozenge at position $(i,j)$,
is defined to be \ts $(x_i - y_j)$. Let $\Ga$ be a region in the plane, and let $T$ be a tiling
of~$\Ga$ (no holes, no overlaps). Let \ts $\HT(T)$ \ts denote the set of
horizontal lozenges (\includegraphics[scale=0.4]{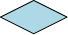}) of~$T$, and let
\begin{equation}\label{eq:wT}
\bwt(T)\, :=  \, \prod_{(i,j) \ts \in \ts \HT(T)} \. (x_i-y_j)
\end{equation}
be the weight of the tiling $T$.

For a partition $\mu$ and an integer $d$, consider plane partitions of
 base $\mu$ and height at most~$d$, these correspond to the set of tilings
 $\Omega_{\mu,d}$ of the plane in the region whose lower side is given by $\mu$,
 and the rest is bounded by the top 4 sides of a hexagon of vertical
 side length~$d$. Given a skew partition $\lambda/\mu$, let
\ts $\Omega_{\mu}(\la)$ \ts be the set of lozenge tilings
 $T\in \Omega_{\mu,d}$, $d=\max\{ f_1^{(\la/\mu)}-1,f_{\ell(\mu)}^{(\la/\mu)}-\ell(\mu)\}$, such that on each vertical diagonal $i-j=k$ there are no
 horizontal lozenges in~$T$ with coordinates $(i,j)$ for $j>\lambda_i$. The
\hookweight of $T$ at position $(i,j)$ is obtained from $\bwt(T)$ by
evaluating \ts $x_i =(\lambda_i-i+1)$ \ts and \ts $y_j = (-\lambda'_j+j)$~:
\begin{equation}\label{eq:hwT}
\bwt_{\lambda}(T) \,:=\, \prod_{(i,j) \ts \in \ts  \HT(T)} (\lambda_i -i + \lambda'_j -j+1)\..
\end{equation}

We define the following map between excited diagrams and lozenge
tilings of base~$\mu$. Let $D \in \ED(\la/\mu)$, then define
$\tau(D):=T$ to be the tiling $T$ with base $\mu$, such that  if box
$(i,j) \in D$, then $T$ has a horizontal lozenge in position $(i,j)$
in the coordinates defined above. See Figure~\ref{fig:excited2lozenge} for an example of $\tau$.

\begin{ex}
There are five lozenge tilings in $\Omega_{21}(332)$ corresponding to
excited diagrams from Example~\ref{ex:excited}~:
\begin{center}
\includegraphics[scale=0.5]{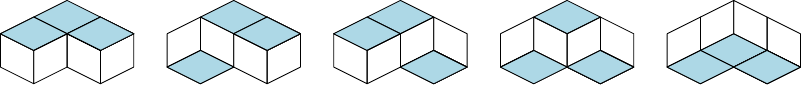}
\end{center}
\end{ex}

 \begin{thm}\label{prop:EDtiling}
 The map $\tau$ is a bijection between excited diagrams
 $\ED(\lambda/\mu)$ and lozenge tilings $\Omega_{\mu}(\la)$.
 \end{thm}

\begin{proof}
We first interpret the excited diagram $D$ as a plane partition $P$ of
shape $\mu$ (and nonpositive entries) under $P_{i,j}=-r_{i,j} +i$, 
alternatively a skew RPP $P'_{i,j} = r_{i,j}-i$,
where $r_{i,j}$ is the row number of the final position of box $(i,j)$
of $\mu$ after it has been moved under the excited moves from $\mu$ to
$D$. Next, $P$ corresponds to a lozenge tiling in the obvious way,
where we set level 0 to be the top $z$-plane and the horizontal
lozenges are moved down to the heights given by $P_{i,j}$. The condition $D
\subset [\lambda]$ is equivalent to the condition that the boxes on diagonal $i-j=k$ cannot
move beyond the intersection of this diagonal and $\lambda$, so that the coordinates of a horizontal lozenge $(i,j)$ must satisfy $j\leq
\lambda_i$.  Finally, note that excited moves correspond to flips
on lozenge tilings, see e.g.~\cite{Thu}. Since the starting excited 
diagram $D=[\mu]$ corresponds to top-adjusted horizontal lozenges whose 
complement can be tiled by non-horizontal lozenges, the same holds for
all $D\in \ED(\lambda/\mu)$.  This implies that~$\tau$ is the desired bijection. 
\end{proof}

From the map $\tau$, adding the corresponding weight of the horizontal lozenges,
we obtain the following result.

\begin{cor}\label{cor:excitedDsLozenge} For a skew shape $\lambda/\mu$, we have
\[
G_{\lambda/\mu}({\bf x} \,\vert\, {\bf y}) \,= \, \sum_{T \in \Omega_\mu(\la)} \, \prod_{(i,j) \in \HT(T)} \. (x_i -
y_{j})\..
\]
\end{cor}

As in the introduction, denote by \ts $\Hex(a,b,c)$ \ts
the \ts $\<a\times b \times c \times a \times b \times c\>$ \ts hexagon with base
$a\times b$ and height~$c$.  Denote by \ts $\OMabc$ \ts
the set of lozenge tilings of \ts $\Hex(a,b,c)$ \ts  weighted as in~\eqref{eq:wT}.

\begin{cor} \label{cor:tilingsprod}
For all \ts $a,b,c,d,e\in \nn$, we have
\begin{equation} \label{eq:cortilings}
\sum_{T \in \OMabc} \. \prod_{(i,j) \in \HT(T)} \. (k-i-j) \, = \,
\frac{\Phi(a+b+c+d+e)\. \Phi(c+d+e)\Phi(a+b+c)\. \Phi(a)\.\Phi(b)\.
\Phi(c)}{\Phi(a+c+d+e)\.\Phi(b+c+d+e)\.\Phi(a+b)\.\Phi(b+c)\.\Phi(a+c)}\,,
\end{equation}
where $k=a+b+2c+d+e+1$.
\end{cor}

\begin{proof}[First proof]
By Corollary~\ref{cor:excitedDsLozenge} and \ref{cor:sympaths_rect_rect_sum_excited} we have that
\[
\sum_{T \in \OMabc} \, \prod_{(i,j) \in \HT(T)} \. (x_i -
y_{j})  \,=\, \, \sum_{S \ts \in\ts
  \ED^{\nearrow}\bigl((b+c)^{a+c}/0^cb^a\bigr)} \prod_{(i,j) \in S}
(x_i-y_j)\ts.
\]
Next we evaluate $x_i = k-i$ and $y_j=j$ to obtain the \hookweight
$(k-i-j)$ for each horizontal lozenge at position $(i,j)$. Note that this \hookweight is
constant when $(i+j)$ is constant.  Thus, each NE-excited diagram on the
RHS above has the same contribution to the sum. Therefore, the product in the RHS
is given by:
\[
\prod_{(i,j) \in S} (k-i-j) \, = \prod_{(i,j) \in (b+c)^{a+c}/0^cb^a} (k-i-j) \,
= \, \frac{\Phi(a+b+c+d+e)\. \Phi(c+d+e) }{\Phi(a+c+d+e)\.\Phi(b+c+d+e)}\,.
\]
Lastly, the number of excited diagrams is given by
\eqref{eq:macmahon}.
\end{proof}

\begin{proof}[Second proof]
Alternatively, by Corollary~\ref{cor:excitedDsLozenge} and
Lemma~\ref{lem:same_ED},
\[
\sum_{T \in \OMabc} \, \prod_{(i,j) \in \HT(T)} \. (k-i-j) \, =
\,\sum_{D \in \ED(\lambda/\mu)} \. \prod_{(i,j) \in D} \. h_{\lambda}(i,j),
\]
where $\lambda/\mu = \Lambda(a,b,c,d,e,0)$. By \eqref{eq:Naruse}, the
RHS is equal to
$$\frac{1}{n!} \. f^{\Lambda(a,b,c,d,e,0)} \. \prod_{(i,j) \in \lambda}
h_{\lambda}(i,j)\..
$$
The result then follows form the product formula
for $f^{\Lambda(a,b,c,d,e,0)}$ from Corollary~\ref{cor:abcde-shape}
and by taking the product of the hooks in~$\lambda$.
\end{proof}

\begin{figure}
\includegraphics[width=12.3cm]{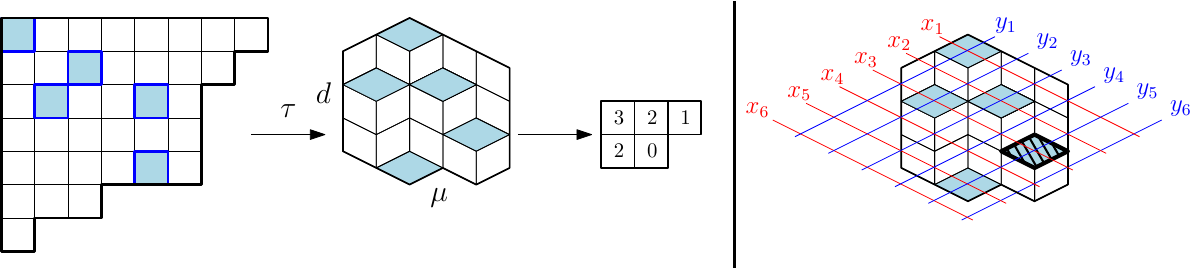}
\caption{Left: the correspondence between excited diagrams with inner
  partition $\mu=32$, lozenge tilings with base $\mu$, and solid
  partitions. Right:  the $x$ and $y$ coordinates giving local weights
  to horizontal lozenges. The highlighted horizontal lozenge has
  weight \ts $(x_3-y_5)$.}
\label{fig:excited2lozenge}
\end{figure}

\subsection{Determinantal formulas for weighted lozenge tilings}
Next, we give determinantal formulas for certain multivariate sums of lozenge tilings in
$\Omega_{\mu}(\lambda)$ and in $\Omega_{\mu,d}$. Recall that
$\Delta({\bf x})$ denotes the Vandermonde determinant.

\begin{thm}\label{thm:tiling_base_mu}
Let $\la/\mu$ be a skew shape, and $d$ a sufficiently large positive integer, so that $\la/\mu \subset d \times (d+\mu_1)$.
Let $d'=d+\ell(\mu)$ and $n:=d'+ d+\mu_1$. Then:
\begin{align*}
\sum_{T \in \Omega_\mu(\la)} \bwt(T)
\, = \,
 \det\bigl[ (x_i -y_1)\cdots (x_i-y_{\mu_j+L_j-j-1})(x_i-x_{d'})\cdots
  (x_i-x_{L_j}) \bigr]_{i,j=1}^n \Delta({\bf x})^{-1}\,,
\end{align*}
where the sum is over lozenge tilings $T$ with base $\mu$ and height~$d$, and
\[
L_j \. := \. \min\{ k : \la_k -k +1 \leq \mu_j -j \}.
\]
\end{thm}
\begin{proof}
We use
Corollary~\ref{cor:excitedDsLozenge} to rewrite the LHS above as the
sum $G_{\lambda/\mu}({\bf x} \,\vert\, {\bf y})$
over excited diagrams. We then use Lemma~\ref{lem:key_border_strips} to
write this sum as an evaluation of the
factorial Schur function $s_\mu^{(d')}({\bf x} \,\vert\, {\bf
  z^{(\lambda)}})$ with $z_{ \lambda_i+ (d'+1-i)}=x_i$ and
$z_{d'+j-\lambda'_j}=y_j$. Next, we evaluate this factorial Schur
function via~\eqref{eq:fschur-det} as a determinant of terms $(x_i-z_1)\cdots (x_i-z_{\mu_j+d'-j})$. We note that
$\{ z_1,\ldots,z_{\mu_j +d'-j} \}= \{ y_1,\ldots,y_{\mu_j+L-j-1},
x_{d'},\ldots,x_{L}\}$, where  $L$ gives the smallest
index of $z$ which is at most $\mu_j+d'-j$ and evaluates to $x$.  In other words, we must have
\ts $\la_L + (d'+1-L) \leq \mu_j+d'-j$.
\end{proof}

\begin{thm}\label{thm:tiling_mu_d}
Consider lozenge tilings with base $\mu$ and height $d$.  Then we have:
$$\sum_{T \in \Omega_{\mu,d}} \bwt(T) \, = \,
\det\bigl[ A_{i,j}(\mu,d) \bigr]_{i,j=1}^{d+\ell(\mu)}\.,
$$
where
$$A_{i,j}(\mu,d) = {\small \begin{cases} (x_i-y_1)\cdots (x_i - y_{d+\ell(\mu)-j})(x_i-x_{i+1})^{-1}\cdots(x_i -x_{d+\ell(\mu)} )^{-1}
&\text{ if } \ j > \ell(\mu), \\
 (x_i - y_1)\cdots(x_i - y_{\mu_j+d}) (x_i-x_{i+1})^{-1}\cdots (x_i -x_{d+j})^{-1} & \text{ if } \ i-d < j \le  \ell(\mu), \\
0 & \text{ if } \ j \leq i-d. \\
\end{cases}}
$$
\end{thm}
\begin{proof}
In Theorem~\ref{thm:tiling_base_mu} we set $\lambda = (\mu_1+d)^d(\mu+d)$,
 where $\mu+d$ means adding $d$ to each part of the partition $\mu$. In other words, $\lambda$ has the same border as $\mu$, but endpoints shifted by $d$ on both axes. By the bijection $\tau$ from Theorem~\ref{prop:EDtiling}, it follows that $\Omega_\mu(\la)$ correspond to $\ED(\la/\mu)$, where the height of the lozenges is determined by how far along the diagonals the excited boxes move. By construction of $\la$, each diagonal has length $d$ between $\mu$ and the border, so $\Omega_\mu(\la)=\Omega_{\mu,d}$.

We apply Theorem~\ref{thm:tiling_base_mu} with the given $\la/\mu$ and $d$. We now plug in the value for $\la$ in terms of $\mu$: $\la_{k+d} = \mu_k +d$ for $k\leq \ell(\mu)$ and $\la_k = d+\mu_1$ for $k\leq d$.
If $k\leq d$, then for all $1\leq j \leq \ell(\mu)$ we have $\la_k+1-k = d+\mu_1+1-k > \mu_1-1 \geq \mu_j-j$.
If $k > d$, then $\la_k + 1-k=\mu_{k-d} + 1 +d-k =\mu_{i'}+1-i'$,
where   $i'=k-d$. Then we see that for $j\leq \ell(\mu)$ we have $L_j=
\min\{ k: \la_k + 1 -k \leq \mu_j-j\} = d+\min\{i': \mu_{i'}+1-i' \leq
\mu_j-j\}=d+j+1$. For $j >\ell(\mu)$, we must have $L_j =d+\ell(\mu)+1$
and there are no $(x_i -x_{d'})\cdots$ terms. Finally, we observe that
$\Delta({\bf x}) = \prod_i (x_i-x_{i+1})\cdots (x_i-x_{d'})$ and divide each entry on line $i$ by the corresponding product $(x_i-x_{i+1})\cdots$.
\end{proof}

\begin{cor}\label{cor:lozengesY0}
Consider lozenge tilings with base $\mu$, such that  $\ell(\mu)=\ell$
and height~$d$, such that horizontal lozenges at position $(i,j)$ have weight~$x_i$.
Then we have the following formula for the \emph{partition function}:
$$\sum_{T \in \Omega_{\mu,d}} \prod_{(i,j) \in \HT(T)}x_i \, \. = \, \det\bigl[ B_{i,j}\bigr]_{i,j=1}^{\ell+d}\,,$$
where
\begin{align*}
B_{i,j} = \begin{cases} x_i^{d+\ell-j}(x_i-x_{i+1})^{-1}\cdots(x_i
    -x_{d+\ell} )^{-1}   &\text{ if } \ j > \ell(\mu), \\
 x_i^{\mu_j+d}(x_i-x_{i+1})^{-1}\cdots (x_i -x_{d+j})^{-1} & \text{ if } \  i-d < j \le \ell, \\
0 & \text{ if } \ j \leq i-d. \\
\end{cases}
\end{align*}
\end{cor}
\begin{proof}
We the apply Theorem~\ref{thm:tiling_mu_d} with $y_j=0$ and $d'=d+\ell$ an obtain the result.\end{proof}

As a byproduct of our calculations we obtain the following determinant formula
given in~\cite[Thm.~6.1]{K} with $\alpha=\beta=0$. To state this
formula we use the standard
notation of the {\em $q$-Pochhammer symbol} $(a;q)_m := (1-a)(1-aq)\cdots (1-aq^{m-1})$.

\begin{cor}[\cite{K}]
Consider the set $PP_\mu(d)$ of plane partitions of base $\mu$ and entries less than or equal to~$d$.
Then their volume generating function is given by the following determinantal formula
$$\sum_{P \in PP_{\mu}(d) } q^{|P|} = q^{N(\mu)} \. \det\bigl[ C_{i,j}\bigr]_{i,j=1}^{\ell+d}\,,$$
where
$$
N(\mu)\. = \. \sum_{r=1}^{\ell(\mu)} \. r \ts \mu_r\,,
$$
\begin{align*}
C_{i,j} = \begin{cases}
(-1)^{d+\ell-i} q^{\alpha}(q;q)^{-1}_{d+\ell-i}  &\text{ if } \ j > \ell(\mu)\ts, \\
(-1)^{d+j-i} q^{\beta}(q;q)^{-1}_{d+j-i} & \text{ if } \ i-d < j \le \ell\ts, \\
0 & \text{ if } \ j \le i-d\ts, \\
\end{cases}
\end{align*}
and
\begin{align*}
\alpha &= (d-i)(d+\ell-j) - (d-i+\ell)(d-i-\ell-1)/2\.,\\
\beta &=(d-i)(\mu_j+d) -  (d+j-i)(d-i-j-1)/2\..
\end{align*}
\end{cor}

\begin{proof}
Let $P\in PP_{\mu}(d)$. As explained in the proof of Theorem~\ref{prop:EDtiling}, it corresponds to a lozenge tiling $T \in\Omega_{\mu,d}$, where the heights of the horizontal lozenges are equal to the corresponding entries in $P$. Suppose that $P_{r,c} = m$, then the corresponding horizontal lozenge has  coordinates given by $(i,j) = (r+d-m, c+ d-m)$, i.e.\ shifted by $(d-m)$ along the diagonal. Let $x_i = q^{d-i}$.  Then:
$$\prod_{(i,j) \in \HT(T)} \. x_i \, = \. \prod_{(i,j) \in \HT(T)} \. q^{d-i} \,
= \, \prod_{(r,c) \in \mu} q^{d -r-d+P_{r,c}} \,  = \, \prod_{(r,c)\in \mu} \. q^{-r} \ts q^{P_{r,c}}
\, = \, q^{|P|-N(\mu)}\..
$$
Therefore, substituting \ts $x_i \gets q^{d-i}$ \ts
in Corollary~\ref{cor:lozengesY0} gives the desired generating function
$$\sum_{P \in PP_{\mu}(d)} q^{|P|} \, = \, q^{N(\mu)} \. \sum_{T \in \Omega_{\mu,d} } \. \prod_{(i,j)\in \HT(T)} q^{d-i}\..
$$
Thus, the entries in the corresponding determinant are given by
\begin{align*}
A_{i,j}(\mu,d)|_{x_i=q^{d-i},y_j=0} \, = \, \begin{cases} q^{(d-i)(d+\ell-j) }(q^{d-i}-q^{d-i-1})^{-1}\cdots(q^{d-i} -q^{-\ell} )^{-1}
&\text{ if } \ j > \ell\., \\
q^{(d-i)(\mu_j+d)} (q^{d-i}-q^{d-i-1})^{-1}\cdots (q^{d-i} -q^{-j})^{-1} & \text{ if } \ i-d < j \le \ell, \\
0 & \text{ if } j \le i-d\.. \\
\end{cases}
\end{align*}
We can simplify the entries as
$$\frac{q^{(d-i)(d+\ell-j) } }{(q^{d-i}-q^{d-i-1})\cdots(q^{d-i} -q^{-\ell} )} \,
= \, \frac{(-1)^{d+\ell-i} q^{(d-i)(d+\ell-j) - (d-i+\ell)(d-i-\ell-1)/2 }}{ (q;q)_{d+\ell-i} }
$$
and
$$\frac{ q^{(d-i)(\mu_j+d)} }{(q^{d-i}-q^{d-i-1})\cdots (q^{d-i} -q^{-j}) } \, = \,
\frac{(-1)^{d+j-i} q^{(d-i)(\mu_j+d) - (d+j-i)(d-i-j-1)/2 }}{(q;q)_{d+j-i}}\,,
$$
which imply the result.
\end{proof}

\bigskip

\section{Probabilistic applications} \label{sec:prob}

Here we present the main application of results in the previous section:
product formulas for the probabilities of two special paths in lozenge tilings
of a hexagon with hook weights of combinatorial significance.

\subsection{Path probabilities}
From here on, we assume that $x_i, y_j \in \rr$, that $(x_i-y_j) \ge 0$
for all $1\le i \le a$ and $1 \le j \le b$ and that $x_i\neq x_j$ for
$i\neq j$.  Recall that \ts $\bigl|\OMabc\bigr|$ \ts
is given by the MacMahon box formula~\eqref{eq:macmahon}. The uniform distribution
on $\OMabc$ is the special case of~\eqref{eq:wT} with $x_i= 1$ and $y_j= 0$, for all $i,j$ as above.

In the hexagon \ts $\Hex(a,b,c)$, consider a path
$\ssp=\ssp(d_0,d_1,\ldots,d_{a+b})$ passing through non-horizontal
lozenges and consisting of unit length segments with  endpoints
$(i,d_i+1/2)$.  Here $i$ is indexing the vertical line, starting with
$i=0$ at the leftmost end of the hexagon, and $d_i+1/2$ is the Euclidean distance measured along that vertical line to its intersection with the top axes $x>0$ or $y>0$, depending on which part the vertical line intersects them.  Note that we
necessarily have \ts $d_0 = d_{a+b}$, $|d_i - d_{i+1}| \leq 1$, $d_i \leq d_{i+1}$ if $ i\leq
 a$, and $d_i \geq d_{i+1}$ if $i>a$.
 Denote by  $\Pr(\ssp)$ the probability that a random weighted lozenge tiling in
 $\OMabc$ contains the path~$\ssp$. In addition, given a partition $\mu
 \subset b^a$, denote by $\mu^\ast$ the complement of $\mu$ in $b^a$.

\begin{ex} \label{ex:lozenge_det}
Figure~\ref{fig:lozenge_det} shows an example of a tiling of a hexagon with $a=2$, $b=3$ and height $c=4$, with a path $\ssp = \ssp(2,3,3,3,3,2)$
dividing the boxed plane partition into tilings with base $\mu = 3\,1$ given by its diagonals $(0,1,1,1,1,0)$,
and $\mu^\ast = 2\,0$.
\end{ex}

\begin{thm}  \label{thm:lozengepathpr} Let $x_i, y_j \in \rr$, $1\le i \le a$, $1\le j \le b$,
s.t. $\min\{x_i\} \ge \max\{y_j\}$.
Consider the distribution on lozenge tilings $T$ of the hexagon \ts $\Hex(a,b,c)$,
weighted by the product $w(T)$ of \ts $(x_i - y_j)$ \ts over all horizontal lozenges. The partition function is then given by
\begin{align*}Z(a,b,c)\, := \, \sum_{T \in \OMabc} \bwt(T)
\, = \, \det\bigl[M_{i,j}\bigr]_{i,j=1}^{a+c}
\end{align*}
where
\[
M_{i,j} = \begin{cases}
(x_i-y_1)\cdots (x_i - y_{c+a-j})(x_i-x_{i+1})^{-1}\cdots (x_i -
  x_{c+a})^{-1}  & \text{ if } \ j > a, \\
 (x_i-y_1)\cdots (x_i - y_{b+c})(x_i -x_{i+1})^{-1} \cdots (x_i -
  x_{c+j})^{-1}  & \text{ if } \ i-c < j \le  a,\\
0, & \text{ if } \ j\le i-c.
 \end{cases}
\]
Moreover, the probability of a path $\ssp=\ssp(d_0,d_1,\ldots,d_{a+b})$ in a random lozenge tiling
$T \in \OMabc$ weighted $w(T)$, is given by
$$\Pr(\ssp) \, = \, \frac{\det\bigl[A_{i,j}(\mu,d_1)\bigr] \, \det\bigl[A^\ast_{i,j}(\mu^\ast,c-d_1-1)\bigr] }{Z(a,b,c)}\,,
$$
where the partition $\mu$ with $\ell(\mu) = a$ is given by its diagonals
$(0,d_2-d_1,d_3-d_1,\ldots)$, and $\mu^\ast$ is the complement of $\mu$ in $b^a$.
Here the matrix $A$~is defined as in Theorem~\ref{thm:tiling_mu_d}, while the matrix~$A^\ast$
is defined similarly, after the substitution \ts $x_i \gets x_{a+c+1-i}$, \ts $y_j \gets y_{b+c+1-j}$.
\end{thm}

\begin{figure}[hbt]
\subfigure[]{
\includegraphics[scale=0.8]{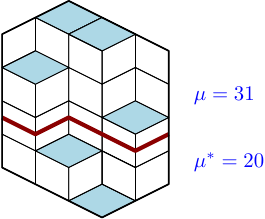}
\label{fig:lozenge_det}
}
\qquad \quad
\subfigure[]{
\includegraphics[scale=0.5]{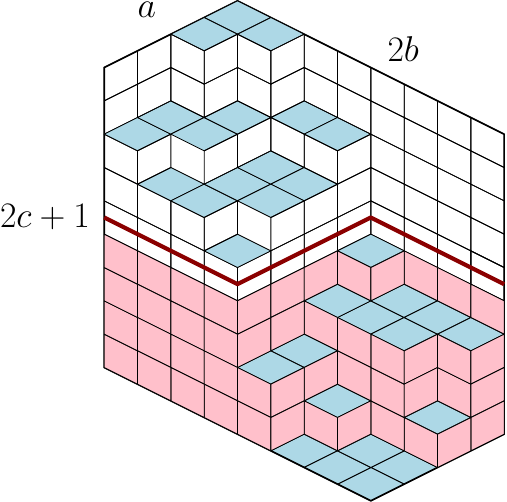}
\label{fig:app_path_pr}
}
\qquad \quad
\subfigure[]{
\includegraphics[scale=0.5]{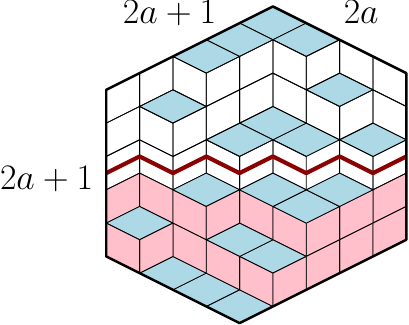}
\label{fig:app2_path_pr}}
\caption{(a) example of a tiling of the region in
  Example~\ref{ex:lozenge_det} with the path $\ssp = \ssp(2,3,3,3,3,2)$,  (b) example of a tiling of the region in
Corollary~\ref{cor:app_path_prob} for $a=b=c=4$, (c)~example of the
tiling of region in Corollary~\ref{prop:app_zigzag_path_prob} for $a=2$.}
\end{figure}

\begin{proof}
The formula for the partition function follows from a direct application of Theorem~\ref{thm:tiling_mu_d} with base $b^a$ and height $d=c$. For the probability, notice that the path $\ssp$ divides the boxed plane partition via a horizontal section along the path, and the partitions $\mu$ and $\mu^\ast$ are the corresponding bases outlined by the path. For the tiling with base $\mu^\ast$ we observe that it corresponds to a change of the coordinates with origin at the bottom corner of the hexagon, which corresponds to flipping the $x$ and $y$ coordinates in the opposite order.
\end{proof}

\subsection{First example} \label{ss:lozenge-explicit}

Denote by
$\Pr_{\lambda}(\ssp)$ the probability $\Pr(\ssp)$
in the special case of \hookweight
\ts $\bwt_{\lambda}(\cdot)$ \ts defined in~\eqref{eq:hwT}.

\begin{cor} \label{cor:app_path_prob}
Fix $a,b,c\in \nn$, partition \ts $\ts \lambda =(2b+2c+1)^{a+2c+1}$, and denote $k=a+2b+4c+3$.
Let 
$\bwt_{\lambda}(T)$ be the corresponding hook
weight of a lozenge tiling~$T$ of the hexagon \ts $\Hex(a,2a,2a+1)$.
Finally, if $a\leq b$ (and analogously if $a>b$), let \ts $\ssp$ \ts be the following path
in \ts $\Hex(a,2b,2c+1)$ \ts of length $(a+2b)~:$
\[
\ssp\,:=\, \ssp(   c,c+1,\cdots,(c+a)^{b-a},c+a-1,\cdots,c^b)\ts .  
\]
 Then:
\[
\Pr_{\lambda}(\ssp) \. = \,  \frac{Q(a,b,c,0,0)\ts \cdot \ts Q(a,b,c,c+1,2c+1)}{Q(a,2b,2c+1,0,0)}\,,
\]
where $Q(a,b,c,d,e)$ is the RHS of~\eqref{eq:cortilings}.
\end{cor}

The choice of weights here is made to correspond to counting of the SYT in
the previous section.  See Figure~\ref{fig:app_path_pr} for an illustration.

\begin{proof}
The path $\ssp$ partitions the rectangle \ts $[a\times 2b]$ \ts into \ts
$\mu = \mu^* = b^a$. By the proof of Theorem~\ref{thm:lozengepathpr}
\begin{equation} \label{eq:pfexplicit-tiling-ex}
\Pr_{\la}(\ssp) \, = \, \frac{N(b^a)\cdot N^*(b^a)}{Z(a,2b,2c+1)}\,,
\end{equation}
where
\[
N(b^a) \,:=\,  \sum_{T \in \OMabc} \. \bwt(T)\.,
\]
and $N^*(b^a)$ equals $N(b^a)$ after the substitution \ts $x_i\gets x_{a+2c+2-i}$,
\ts $y_j \gets y_{2b+2c+2-j}$.

Next, we evaluate $x_i =k-i$ and $y_j=j$ in
\eqref{eq:pfexplicit-tiling-ex} to obtain the \hookweight of
the tiling and thus get $\Pr_{\lambda}(\ssp)$. By Corollary~\ref{cor:tilingsprod}
for the hexagon $\Hex(a,2b,2c+1)$, we have:
\[
Z(a,2b,2c+1) \,\bigl\vert\,_{\substack{x_i = k-i \\ y_j=j}} \,=\, Q(a,2b,2c+1,0,0)\ts.
\]
By Corollary~\ref{cor:tilingsprod} for the hexagon \ts $\Hex(a,b,c)$,  we have:
\[
N(b^a) \,\bigl\vert\,_{\substack{x_i = k-i \\ y_j=j}} \, = \, Q(a,b,c,c+1,2c+1)\ts,
\]
and
\[
N^*(b^a) \,\bigl\vert\,_{\substack{x_i = k-i \\ y_j=j}} \, = \, Q(a,b,c,0,0)\ts.
\]
Together, these imply the desired expression for \ts $\Pr_{\lambda}(\ssp)$.
\end{proof}

In the notation of the proposition, 
let $a=b=c$.  Then  \ts $\rP_w(a):=\Pr_{\lambda}(\ssp)$ \ts
is exactly the probability that the random hook weighted lozenge
tiling of $\Hex(a,2a,2a+1)$ has two frozen $\<a\times a\>$ rombi as in
Figure~\ref{fig:app_path_pr}, where the weights are chosen to correspond
to SYT counting (cf.\ figures~\ref{f:lozenge-hexagon} and~\ref{f:stair}).
Of course, in the uniform case the corresponding probability $\rP_u(a)$
is a little easier to compute:
$$
\rP_u(a) \, = \, \frac{|\PP(a,a,a)|^2}{|\PP(a,2a,2a+1)|}\,,
$$
see the MacMahon box formula~\eqref{eq:macmahon}.  A direct calculation
shows that
$$\log \rP_w(a) \. = \. \al \ts a^2 + O(a\log a)\.,  \quad \log \rP_u(a) \. = \. \be \ts a^2 + O(a\log a) \,
\quad \text{and} \ \ \al < \be < 0\ts.
$$
Since there are \ts $\binom{3a}{a}=\exp O(a)$ \ts possible paths with the same
endpoints as~$\ssp$, this shows that $\ssp$ is exponentially unlikely in both cases,
and even less likely in the hook weighted lozenge tiling.

\subsection{Second example} \label{ss:lozenge-explicit2}

Our next example uses the number of excited diagrams of thick
ribbons \ts $\de_{2a+1}/\de_a$, studied extensively in~\cite{MPP3,MPP4}.

\begin{cor}\label{prop:app_zigzag_path_prob}
Fix $a\in \nn$, partition \ts $\lambda=(4a+1)^{4a+2}$, and let
$\bwt_{\lambda}(T)$ be the corresponding hook
weight of a lozenge tiling $T$ of the hexagon \ts $\Hex(2a+1,2a,2a+1)$.  Finally, let \ts $\ssq$ \ts be the
zigzag path in \ts $\Hex(2a+1,2a,2a+1)$ \ts of length $(4a+1)~:$
$$
\ssq \, := \, (a,a,a+1,a+1,a+2,a+2,\ldots,2a,2a,\ldots,a,a) \ts. 
$$
Then:
\begin{equation} \label{eq:bdsexpahtprob}
\frac{C(a)}{\bigl|\PP(2a+1,2a,2a+1)\bigr|} \,\leq \, \Pr_\la(\ssq) \, \leq \,
\frac{C(a)\ts \cdot \ts\bigl|\RPP_{\delta_{2a+1}}(a)\bigr|^2}{\bigl|\PP(2a+1,2a,2a+1)\bigr|}
\,,
\end{equation}
where \ts $\bigl|\RPP_{\delta_{2a+1}}(a)\bigr|$ \ts is given
by~\eqref{eq:proctor}, and
\[
C(a) \,=\, \frac{\Phi(8a+3)\. \Phi(4a+2)\. (4a)!^{2a+1}}{\Phi(6a+2)^2 \. (6a+2)!^{2a+1}\. (2a)!}\,.
\]
\end{cor}

\begin{proof}
The proof follows along the same lines as the proof of
Corollary~\ref{cor:app_path_prob} above.
The path~$\ssq$ partitions the rectangle \ts $\bigl[(2a+1)\times 2a\bigr]$ \ts
into shapes $\mu$ and~$\mu^*$, where $\mu=\mu^* = \delta_{2a+1}$.
By the proof of Theorem~\ref{thm:lozengepathpr}
\begin{equation}\label{eq:pfexplicit-tiling-ex2}
\Pr_{\la}(\ssq) \, = \, \frac{N(\delta_{2a+1}) \. N^*(\delta_{2a+1})}{Z(2a+1,2a,2a+1)}\,,
\end{equation}
where
\[
N(\delta_{2a+1})  \, :=\, \sum_{T \in \Omega_{\delta_{2a+1},a}} \prod_{(i,j) \in
  \HT(T)} \bwt(T)\.,
\]
and $N^*(\delta_{2a+1})$ is equal to $N(\delta_{2a+1})$ after the substitution \ts
$x_i \gets x_{4a-i+3}$ \ts and \ts $y_j \gets y_{4a-j+2}$. Next, we evaluate $x_i$ and
  $y_j$ as specified to obtain the \hookweight of the tiling and get
  $\Pr_{\lambda}(\ssq)$ (see figures~\ref{fig:hooks_app2_path_pr}
  and~\ref{fig:hooks_app2_path_pr_orig}).


\begin{figure}[hbt]
\subfigure[]{
\includegraphics[scale=0.60]{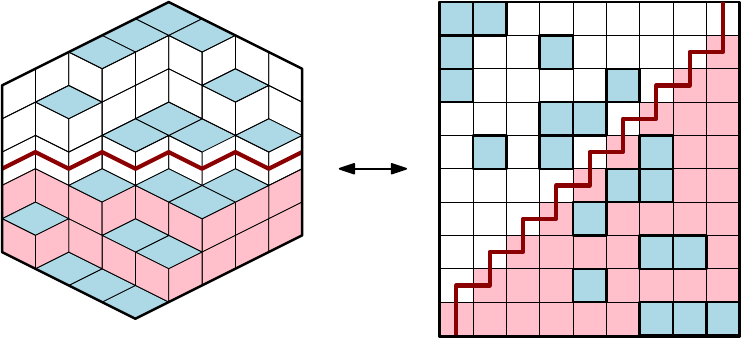}
\label{fig:hooks_app2_path_pr}
}
\qquad \qquad \ \
\subfigure[]{
\includegraphics[scale=0.60]{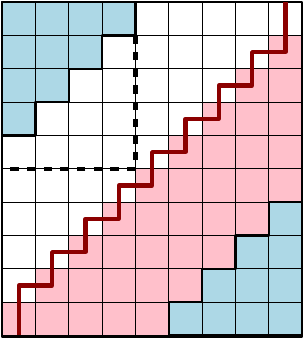}
\label{fig:hooks_app2_path_pr_orig}
}
\caption{
(a)~excited diagram interpretation of the example of a tiling of the region
in Corollary~\ref{prop:app_zigzag_path_prob}; \. (b)~skew shapes on
both sides of the path $\ssq$ and values of $x_i,y_j$ of this example.}
\end{figure}

By a similar argument to the first proof of
Corollary~\ref{cor:tilingsprod}, the partition function of the
denominator factors as follows:
\[
Z(2a+1,2a,2a+1)\bigr|_{\bwt(T)=\bwt_{\lambda}(T)}  \, = \,
 \. \frac{\Phi(6a+2)\.\Phi(2a+1)}{\Phi(4a+2)\.\Phi(4a+1)} \,\. \bigl|\PP(2a+1,2a,2a+1)\bigr|\..
\]
We have:
\[
N(\delta_{2a+1}) \bigr|_{\bwt(T)=\bwt_{\lambda}(T)} \.\,=\, \sum_{D\in
  \ED(\pi)} \prod_{(i,j) \in D} \. h_{\lambda}(i,j)\..
\]
By the proof of the bound in \cite[Thm.~1.1]{MPP3}, we have:
\[
\frac{\Phi(8a+3)}{\Phi(6a+3) \cdot (6a+2)!^{2a}} \,\leq\, N(\delta_{2a+1})
\bigr|_{\bwt(T)=\bwt_{\lambda}(T)}  \,\leq\,
\frac{\Phi(8a+3)}{\Phi(6a+3) \cdot (6a+2)!^{2a}} \, \bigl|\RPP_{\delta_{2a+1}}(a)\bigr| \..
\]
A similar calculation for  $N^*(\delta_{2a+1})$ gives
\[
\frac{\Phi(2a)
  (4a)!^{2a}}{\Phi(4a) }\,\leq\,  N^*(\delta_{2a+1}) \bigr|_{\bwt(T)=\bwt_{\lambda}(T)}
\,\leq\,  \frac{\Phi(2a)
  (4a)!^{2a}}{\Phi(4a) } \, \bigl|\RPP_{\delta_{2a+1}}(a)\bigr| \..
\]
Applying these bounds on the RHS of \eqref{eq:pfexplicit-tiling-ex2} gives the
desired result.
\end{proof}


In the notation of the proposition, the probability \ts $\rP_w(a):=\Pr_\la(\ssq)$ \ts
is exactly the probability that the random hook weighted lozenge tiling of
the hexagon \ts $\Hex(2a+1,2a,2a+1)$ \ts has a horizontal zigzag path as in
Figure~\ref{fig:app2_path_pr}, splitting the hexagon into two equal shapes.

For comparison, in the uniform case, the corresponding probability $\rP_u(a)$ is
given by:
\[
\rP_u(a) \,=\,  \frac{\bigl|\RPP_{\delta_{2a+1}}(a)\bigr|^2}{\bigl|\PP(2a+1,2a,2a+1)\bigr|}\,,
\]
so the second inequality in the corollary can be written as
$$
\rP_w(a) \,\le \, C(a)\. \rP_u(a)\ts.
$$

Now direct calculation using
the bounds above gives a remarkable contrast between the asymptotics:
$$
\log \rP_w(a) \. = \. \Theta(a^2)\.,  \quad \log \rP_u(a) \. = \. \gamma \ts a + O(\log a)
\,
\quad \text{for some} \ \ \ga < 0\ts.\footnote{While we certainly believe that \ts
$\log \rP_w(a) =  \tau \ts a^2 + o(a^2)$, for some $\tau <0$,
the bounds in Corollary~\ref{prop:app_zigzag_path_prob} are
too weak to establish that. It would also be interesting to
compute~$\tau$ explicitly. }
$$
This supports the intuition that for the uniform distribution path~$\ssq$ is
at least as likely as any other path among the $\binom{4a+1}{2a}$ possible, while
for the hook weighted distribution path~$\ssq$  is extremely unlikely.  In the
language of limit shapes in Figure~\ref{f:lozenge-hexagon}, this says that
in the uniform case, the limit shape (the Arctic circle) touches the vertical
sides in the middle, while in the hook weighted case it touches
someplace higher.

\begin{ex}
In the case $a=1$, we have: \ts $\lambda = 5^6$, \ts $C(1) = 54/35$, \ts
$|\PP(3,2,3)| =175$  \ts and \ts $|\RPP_{\delta_{3}}(1)|=5$. The bounds for \ts $\rP_w(1)=\Pr(\ssq)$
\ts given by~\eqref{eq:bdsexpahtprob}, are:
\[
\frac{54}{6125} \, \approx \, 0.0088 \,\leq \,  \rP_w(1)  \,\leq \, \frac{54}{245} \, \approx \, 0.2204\ts.
\]
The actual value of the probability is $\ts \rP_w(1) =246/4375 \approx 0.0562$\ts.  On the other hand,
in the uniform case we have \ts $\rP_u(1) =  1/7 \approx 0.1429$\ts.  Note that here we have $C(1)>1$,
while asymptotically $C(a) < \exp (-c\ts a^2)$, for some $c>0$.
\end{ex}

\bigskip

\section{Final remarks}\label{sec:fin-rem}

{\small

\subsection{Historical notes}\label{ss:fin-hist}
The \emph{hook-length formula} (HLF) plays an important role in both enumerative
and algebraic combinatorics, and has a large number of proofs, extensions and
generalizations.  We refer to~\cite{AR} for a  comprehensive recent survey,
and to \cite[$\S$9]{MPP1} for a review of the NHLF and other formulas
generalizing \ts $f^{\la/\mu}\ts =\ts \bigl|\SYT(\la/\mu)\bigr|$.

Likewise, the subject of domino and lozenge tilings is a large subject in its
own right, with many determinant and product formulas
(notably, for the \emph{Aztec diamond}), weighted extensions,
asymptotic and probabilistic results.  We refer to \cite{Bet,Lai} for extensive
recent discussions of both and overview of the literature. Note that even among
other tiling problems, domino and lozenge tilings are special
to have both determinantal formulas and  height functions (see~\cite{Pak}).

Finally, the subject of Schubert polynomials has several enumerative
formulas including Macdonald's identity \eqref{eq:macdonald}.
One of the most celebrated product formula \ts $|R(w_0)|=f^{\de_n}$, where
$w_0\in S_n$ is the longest permutation, is due to Stanley.  It is now generalized
in many directions including the Fomin--Kirillov identity
(Corollary~\ref{cor:FK-dominant}).  The formula \eqref{eq:SchubsKMY} for Schubert
polynomials of vexillary permutations appears in the literature in
terms of {\em flagged Schur functions} of shape $\mu(w)$ \cite[Thm.~2.6.9]{Man}. Also,
the Schubert polynomial of a $321$-avoiding permutation $w$ is a flagged skew
Schur function of shape $\skewsh(w)$ \cite[Thm.~2.2]{BJS}. We  refer to \cite{Las,Man} for detailed introductions to
the area.

\subsection{Bijective proofs for product formulas} \label{sec:bij}
The product formulas in Corollaries~\ref{cor:abcde-shape}--\ref{cor:abcde1-shape}
and their $q$-analogues beg for  bijective or hook-walk type proofs, see~\cite{NPS,GNW}.
We should warn the reader, however, of many related product formulas which have yet
to have a bijective proof.  Most famous of this is the product formula for the number
of \emph{alternating sign matrices} (ASM), which in Kuperberg's proof comes out as an
evaluation of a ``hidden symmetry'' of a multivariate determinant, of similar flavor
to our proof~\cite{Kup1} (see also~\cite{Bre,Kup2}).

Similarly, some years ago the second author proposed giving a combinatorial proof of the
{\em Selberg integral} by proving an explicit product formula with several parameters
counting linear extensions of certain posets (see~\cite[Ex.~3.11(b)]{EC}). The product
formulas are superficially similar in flavor to the ones in Corollary~\ref{cor:abcde-shape}
due to the structure of parameters; in fact they look even simpler.  While this project is yet to be realized, this connection was used in reverse direction in an elegant paper~\cite{KO}.

\begin{figure}[hbt]
\begin{center}
\includegraphics[scale=0.8]{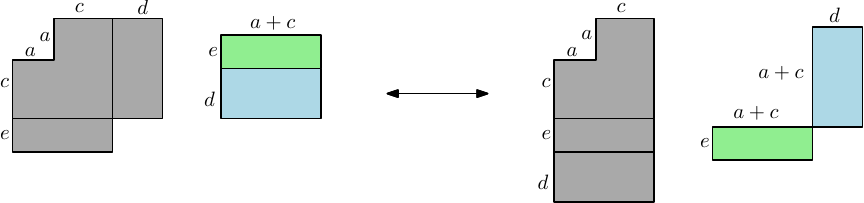}
\end{center}
\caption{Potential bijection for $\pi = \LA(a,a,c,d,e,0)$.}
\label{f:bij-sym}
\end{figure}

In a positive direction, we should mention that in the special case of \ts $\pi = \LA(a,a,c,d,e,0)$,
Corollary~\ref{cor:abcde-shape} implies that there is
a mysterious identity for \ts $f^\pi = \bigl|\SYT(\pi)\bigr|$~:
\begin{equation}
\aligned
& \bigl|\SYT(\pi)\bigr| \. \cdot \. \bigl|\SYT\bigl((a+c)^{d+e}\bigr)\bigr|
\,  = \,
\bigl|\SYT\bigl((a+c)^{a+c+d+e}/a^a\bigr)\bigr| \, \. \times \\
& \qquad  \times \,  \binom{(a+c)(d+e)}{(a+c)e}\. \cdot \.
\bigl|\SYT\bigl((a+c)^e\bigr)\bigr| \. \cdot \.  \bigl|\SYT\bigl((a+c)^d\bigr)\bigr|\..
\endaligned
\end{equation}
Since all other terms in the product do have a bijective proof of the corresponding
product formulas, a bijective proof
of this identity would imply a (rather involved combined) bijective proof of the product formula
for $\bigl|\SYT(\pi)\bigr|$.

Finally, we should note that our Theorem~\ref{thm:skewprod} should be viewed
as a stand-alone coincidence rather than beginning of the emerging pattern.  In some
sense, we are really saying that for certain families of skew shapes the determinantal
formula for $f^{\la/\mu}$ can be further simplified to a product formula.  Thus our
product formulas have a natural home in \emph{Determinantal Calculus}~\cite{Kra2,Kra3}
and lozenge tilings literature (see e.g.~\cite{Lai,Lai2}),
rather than the general study of linear extensions of posets.

\medskip

\subsection{Kim--Oh's theorem} \label{sec:kimoh}

We learned about~\cite{KO} only after this paper was finished. They
prove Corollary~\ref{cor:abcde-shape} via a product formula for {\em
  Young books}: pairs of SYT of shifted shape  $\LAM(a,c,d,0)$ and
$\LAM(b,c,e,0)$ with the same diagonal entries (see
\S~\ref{sec:shiftedshapes} for a definition of the shape $\LAM$). Their tools cannot be used to
derive our main product formula in Theorem~\ref{thm:skewprod}. This
would require a version of Young books of shapes $\LAM(a,c,d,m)$.
Note that the $q$-analogue in Corollary~\ref{cor:qabcde-shape} does not follow from~\cite{KO}, but
perhaps follows from a $q$-Selberg integral generalization of~\cite{KO}
given in~\cite{KOO}.

\medskip

\subsection{DeWitt's theorem} \label{sec:dewitt}
The case of the shape $\LA(a,b,c,0,0,1)$ in Corollary~\ref{cor:abc-shape} might be
the most tractable since its product formula is known to count another natural
object as we explain next. DeWitt showed in her thesis \cite{DeW} that in this case
$f^{\lambda/\mu}$ counts, up to a power of $2$, the number of
SYT of a shifted shape. Given nonnegative integers $a,b,c$ let $T(a,b)$ be the
trapezoid
\[
T(a,b)\. :=\. (a+b-1,a+b-3,\ldots,|b-a|+1)\ts,
\]
and let $D(a,b,c)$ be the
 shifted shape obtained by flipping by the diagonal $y=-x$ the shifted skew shape
$\delta_{a+b+2c}/T(a,b)$. See Figure~\ref{fig:dewittshape} for
an example of this shape.

\begin{figure} [hbt]
\includegraphics{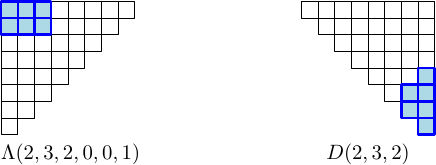}
\caption{Examples of the skew shape $\LA(a,b,c,0,0,1)$ and the
  shifted shape $D(a,b,c)$ that have equinumerous SYT.}
\label{fig:dewittshape}
\end{figure}

\begin{thm}[Thm.~V.3~\cite{DeW}] \label{thm:dewitt}
For the skew shape $\ts \lambda/\mu  = \delta_{a+b+2c}/b^a$, we have:
\begin{equation} \label{eq:iddewitt}
s_{\lambda/\mu} \,=\, P_{D(a,b,c)}\..
\end{equation}
\end{thm}

By taking the coefficient of $x_1x_2\ldots x_n$ in the equation above we obtain the
following identity between the number of SYT of skew shape
$\LA(a,b,c,0,0,1)$ and  shifted shape $D(a,b,c)$.

\begin{cor}[Cor.~V.7~\cite{DeW}] \label{cor:dewitt}
For the skew shape $\ts \lambda/\mu  = \delta_{a+b+2c}/b^a$, we have:
\[
f^{\lambda/\mu} \, = \, 2^{|\lambda/\mu|-a-b-2c+1} g^{D(a,b,c)}\..
\]
\end{cor}

Combining this identity with the hook-length formula for $g^{\nu}$
(see e.g.~\cite[Ex.~3.21]{SSG}), we obtain a product formula for \ts
$f^{\LA(a,b,c,0,0,1)}$ \ts coinciding with that of Corollary~\ref{cor:abc-shape}.
Similarly, by doing a stable principal specialization in \eqref{eq:iddewitt} and
using the \emph{Kawanaka product formula}~\cite{Ka} for this specialization of
Schur $P$-functions of straight shifted shapes, we obtain the formula in
Corollary~\ref{cor:qabc-shape}. This identity was obtained earlier and
more generally by Krattenthaler and Schlosser, see Eq.~(1.2) with
$n=a+b+2c$, $m=b$ and $r=a$ in~\cite{KS}.

Note that DeWitt and Ardila--Serrano~\cite{AS} showed independently that the skew
Schur function $s_{\delta_m/\mu}$ has a positive expansion in the Schur $P$-functions.
From this expansion one can obtain Theorem~\ref{thm:dewitt}.

It is natural to ask for a bijective proof of Corollary~\ref{cor:dewitt}. Such a bijection
combined with the hook-walk algorithm for shifted shapes~\cite{SHW}
or the bijective proof of the hook-length formula for $g^{\nu}$ \cite{Fis},
gives an algorithm to generate SYT  of skew shape $\LA(a,b,c,0,0,1)$ uniformly at random.
We obtain the desired bijection in the followup work~\cite{MPSW}.

\subsection{Shifted shapes} \label{sec:shiftedshapes}
One of the main results of this paper is to give families of skew
shapes whose number of SYT is given by a product formula. A natural
direction is to study the same question for shifted skew shapes.

Naruse in~\cite{Strobl} also announced two formulas (of type $B$ and type~$D$) 
for the number of standard tableaux of shifted skew shape 
(see \cite[$\S$8]{MPP2}), in terms of analogues of excited 
diagrams.\footnote{This result was recently proved and further
generalized in~\cite{NO}.}

The \emph{type~$B$ excited diagrams} are obtained from the diagram 
of~$\mu$ by applying the following \emph{type~$B$ excited moves}:
\begin{center}
\includegraphics{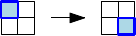} \, \qquad \ and \,\qquad \
\includegraphics{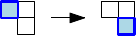}
\end{center}
We denote the set of type $B$ excited diagrams of shifted skew shape
$\lambda/\mu$ by $\ED^B(\lambda/\mu)$. Following the arguments in
Section~\ref{ss:excited-flagged} and \cite[\S 3]{MPP1}, the type~$B$
excited diagrams of $\lambda/\mu$ are equivalent to certain flagged
tableaux of shifted shape $\mu$ and to certain
non-intersecting paths (cf.~\cite{Ste}).

\begin{question}
Is there a determinantal or Pfaffian formula for $|\ED^B(\lambda/\mu)|$
counting the corresponding flagged tableaux of shifted shape~$\mu$?
\end{question}

Given a shifted shape $\lambda$, the type $B$ hook of a cell $(i,i)$ in the diagonal is the set of cells in
row $i$ of $\lambda$. The hook of a cell $(i,j)$ for $i\leq j$ is the set of cells in row $i$ right of $(i,j)$, the
cells in column $j$ below $(i,j)$, and if $(j,j)$ is one of these cells below
then the hook also includes the cells in the $j$th row of $\lambda$,
thus counting $(j,j)$ twice overall. See Figure~\ref{fig:acdm-shshape}.
The NHLF then extends verbatim for the number $g^{\lambda/\mu}$ of
 standard tableaux of shifted skew shape $\lambda/\mu$.

\begin{thm}[Naruse \cite{Strobl}]
Let $\lambda, \mu$ be partitions with distinct parts, such that $\mu
\subset \lambda$. We have
\begin{align}
g^{\lambda/\mu} &= |\lambda/\mu|! \sum_{S \in
  \ED^B(\lambda/\mu)} \prod_{(i,j) \in [\lambda]\setminus S}
                  \frac{1}{h^B(i,j)}\,. \label{eq:nhlfB}
\end{align}
\end{thm}

Next we describe the shifted analogue of the thick reverse hook (Example~\ref{ex:excited_macmahon}).

\begin{ex}[shifted reverse hook]
For the shape $R_{a,c}:=\delta_{a+c+1}/\delta_{a+1}$, the type $B$ excited
diagrams correspond to symmetric plane partitions with at most $a+1$ rows
and largest part at most $c$. By the \emph{Andrews--Gordon formula} for
symmetric plane partitions (see~\cite{St1}), we have:
\[
|\ED^B(R_{a,c})| \, = \, \prod_{1\leq i \leq  j \leq a} \. \prod_{k=1}^c \,
\frac{i+j+k-1}{i+j+k-2} \, =  \, \. \frac{\Phi(2a+c) \. \Phi(a)}{\Phi(2a)\. \Phi(a+c)} \, \cdot \, \frac{\Lam(2a) \. \Lam(c)}{\Lam(2a+c)}\,.
\]
\end{ex}

It is natural to study shifted analogues of our product formulas for skew
shapes. For nonnegative integers $a\le c,d$ and~$m$, let $\lambda/\mu=\LAM(a,c,d,m)$
be the following shifted skew partition
\[
\lambda \. = \. (c+a, c+a-1,\ldots,1) \. + \. \nu\ts,
\]
where $\ts \nu = (d+(a+c-1)m,d+(a+c-2)m,\ldots,d)$ \ts and $\ts \mu =
\delta_{a+1}$. See Figure~\ref{fig:acdm-shshape}.

\begin{figure}
\includegraphics{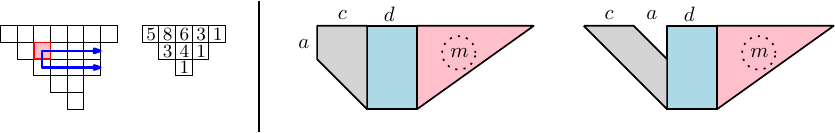}
\caption{Left: example of the type $B$ hook of a cell $(i,j)$ of $\lambda$
  of length~$9$ (cell $(3,3)$ is counted twice), and the type~$B$
  hook-lengths of the cells of the shifted shape $(5,3,1)$. Right: The
  shifted skew shape $\LAM(a,c,d,m)$ and the shape
  whose hooks
  appear in the product formula in Conjecture~\ref{conj:shiftedshapes}.}
\label{fig:acdm-shshape}
\end{figure}

Computations using the Pfaffian formula for $g^{\lambda/\mu}$
(see~\cite[Thm.~7.5]{I}), suggest
the following conjectured\footnote{This conjecture was recently established in~\cite{KY}.} product formula for these shifted skew shapes.

\begin{conj} \label{conj:shiftedshapes}
In the notation above, for \ts $\pi = \LAM(a,c,d,m)$, we have:
$$g^{\pi}  \, = \, \frac{n!}{2^{a}} \, \cdot \,
\frac{\Phi(2a+c) \. \Phi(a)}{\Phi(2a)\. \Phi(a+c)} \, \cdot \,
\frac{\Lam(2a) \. \Lam(c)}{\Lam(2a+c)} \, \cdot \,
\prod_{u \in \lambda \setminus (\delta_{a+c+1}/c^a\delta_{c+1})} \, \frac{1}{h_B(u)} \,.
$$
\end{conj}

See Figure~\ref{fig:acdm-shshape} for an illustration of the
cells of the shifted shape $\lambda$ whose hook-lengths appear in the conjectured
formula above. The special case \ts $\pi = \LAM(a,c,d,0)$ \ts is the (conjugated)
\emph{truncated rectangle shape}, and was established by the third author
using a different technique~\cite{Pan} and later in \cite[Cor. 4.6]{KO} by yet
again different methods.  In particular, for
\ts $\pi = \LAM(a,c,0,0)$, we obtain the product formula for shifted reverse
hook in the example above.  This both lends support to the conjecture
and explains its appearance, which seemed out of place until now, see~\cite{AR}.

\subsection{Racah and $q$-Racah formulas} \label{subsec:Racah}
In the Appendix of~\cite{BGR}, the authors generalize the MacMahon box formula~\eqref{eq:macmahon1}
to five variables, which they formulate in terms of lozenge tilings of $\Hex(a,b,c)$ with weights
$F(p,q,u_1,u_2,u_3)$
given by products of certain elliptic functions (see Theorem~10.5 in~\cite{BGR}).  Upon seeing our
main technical tool, Theorem~\ref{thm:thickstrip}, Eric Rains noticed\footnote{Personal
communication.} that there is a common
special case of both formulas giving the \emph{$q$-Racah formula}.  In the notation of~\cite{BGR}, let
$u_1=u$, $u_2=u_3=0$,  $p\to 0$ \ts to get the following result:

\begin{cor} [Appendix to~\cite{BGR}]  We have:
\begin{equation}\label{eq:racah}
\sum_{\Pi\. \subset \.  [a\times b\times c]} \,
\prod_{(i,j,k) \ts \in \ts \Pi} \,
\frac{q-q^{j+k-2i}u}
     {1-q^{j+k-2i+1}u} \ \, = \,
\prod_{(i,j,k)\ts \in \. [a\times b\times c]}
  \frac{(1-q^{i+j+k-1})(1-q^{j+k-i-1}u)}
       {(1-q^{i+j+k-2})(1-q^{j+k-i}u)}\,,
\end{equation}
where the summation is over all plane partitions $\Pi$ which fit inside the box \ts $[a\times b\times c]$.
\end{cor}

In the notation of Theorem~\ref{thm:thickstrip}, this identity follows by considering the
hexagon \ts $\Hex(b,c,a)$ ($a$~is the height this time), so $\mu = b^c$, then setting
\ts $x_i = q^{-i+a}$, $y_j = uq^{j-a-1}$, and noting that the RHS factors.  We omit the details.

For $u=0$, equation~\eqref{eq:racah} gives the MacMahon $q$-formula~\eqref{eq:qmacmahon},
where the sum is over plane partitions, while the \emph{Racah formula} follows by letting
$q,\ts u=q^h\ts\to 1$
\begin{equation}
\sum_{\Pi\. \subset \. [a\times b\times c]} \,
\prod_{(i,j,k) \ts \in \ts \Pi} \,
\frac{j+k-2i-1+h}
     {j+k-2i+1+h} \, \ =  \,
\prod_{(i,j,k)\ts \in \. [a\times b\times c]} \,
  \frac{(i+j+k-1)(j+k-i-1+h)}
       {(i+j+k-2)(j+k-i+h)}\,.
\end{equation}
When $a=1$, this gives~\eqref{eq:Naruse-1-path}, since plane partitions of height 1 correspond to a single lattice path.
Finally, when $h \to \infty$, this identity gives the MacMahon box
formula~\eqref{eq:macmahon1}.

\subsection{Excited diagrams and Grothendieck polynomials}
In addition to Theorem~\ref{thm:SchubsKMY}, Knutson--Miller--Yong
\cite[Thm.~5.8]{KMY} also gave a formula for the {\em
  Grothendieck polynomials} of vexillary permutations in terms of a
larger class of diagrams called {\em generalized excited
  diagrams}. For the shape $\lambda/\mu$ these diagrams are defined as
follows: for each active cell $(i,j)$ we do two types of {\em generalized excited
moves}: \ts (i)~the usual move replacing  $(i,j)$ by $(i+1,j+1)$, or~(ii) the
move which keeps $(i,j)$ and adds $(i+1,j+1)$~:
\begin{center}
\includegraphics{excited_move} \ \, \qquad or \qquad \ \,
\includegraphics{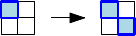}
\end{center}
These diagrams were also studied by
Kreiman~\cite{VK} and they are in correspondence with {\em set valued}
flagged tableaux. In~\cite{MPP5}, we use these diagrams to give a
generalization of Naruse's formula~\eqref{eq:Naruse} and the analysis
in Section~\ref{sec:KMY} for Grothendieck polynomials.

\begin{figure}[hbt]
\includegraphics[scale=0.31]{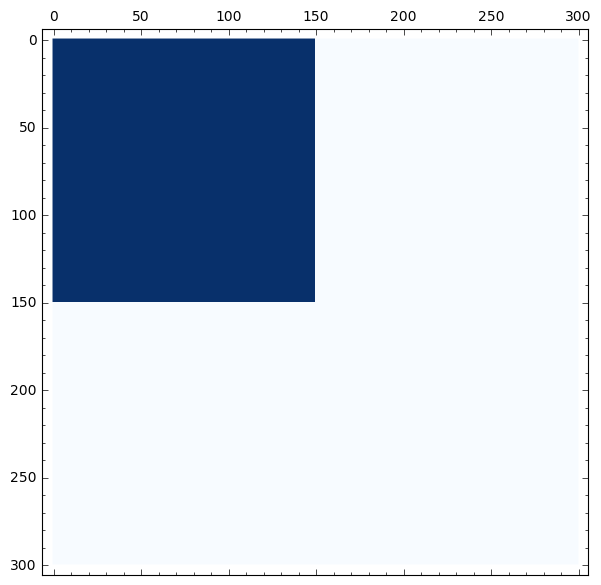}
\includegraphics[scale=0.31]{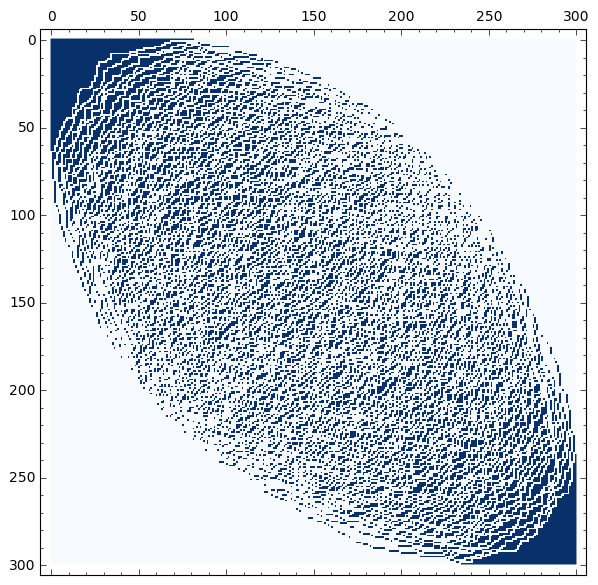} \includegraphics[scale=0.31]{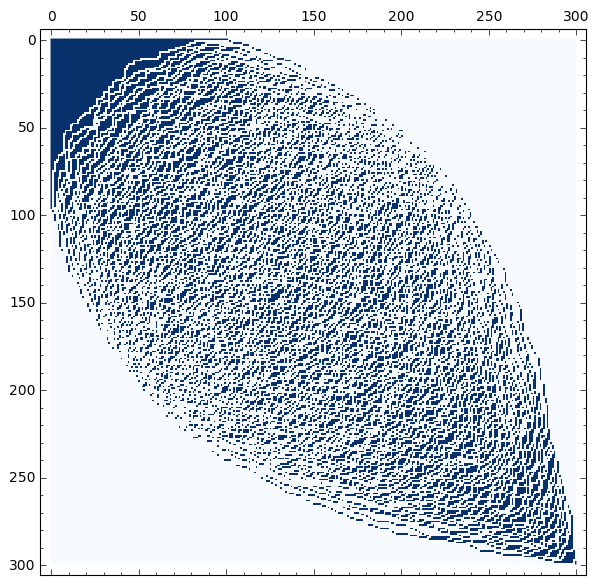}
\caption{A random excited diagram of shape
  $300^{300}/150^{150}$ (left), in the uniform (middle), and the
  hook-weighted distribution (right).}
\label{f:metro}
\end{figure}

\subsection{Limit shapes}\label{ss:asy-lozenge}
Since excited diagrams are in bijection with lozenge tilings
(see Theorem~\ref{prop:EDtiling}), one can translate known
limit shape results for tilings into the language of
excited diagrams.  For example, the middle picture in Figure~\ref{f:metro}
is a random excited diagram of shape  \ts $300^{300}/150^{150}$.
These are obtained by running a Metropolis algorithm for \ts $10^{10}$ \ts steps.
The visible limit shape is in fact a stretched circle.

Similarly, since $\SYT(\la/\mu)$ are enumerated by the \emph{weighted
excited diagrams} (by a product of hooks of the squares in the diagram),
one can ask about limit shapes of hook weighted lozenge tilings.
An example of a clearly visible limit shape is shown in the right picture
in Figure~\ref{f:metro}.  Both  examples are a larger version of the
lozenge tilings in Figure~\ref{f:lozenge-hexagon}.

\begin{figure}[hbt]
\begin{center}
\includegraphics[width=5.5cm]{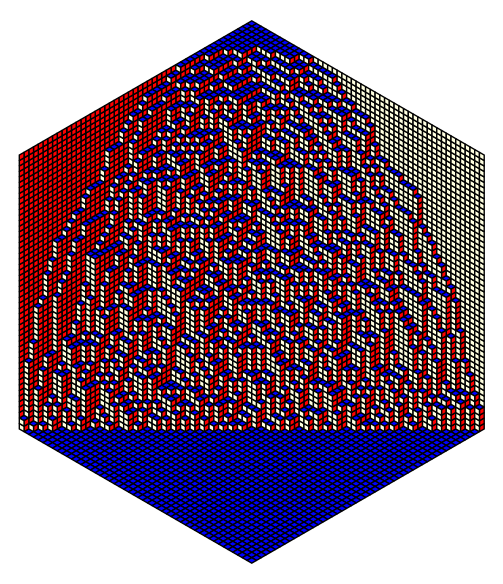}
\qquad \qquad
\includegraphics[width=5.5cm]{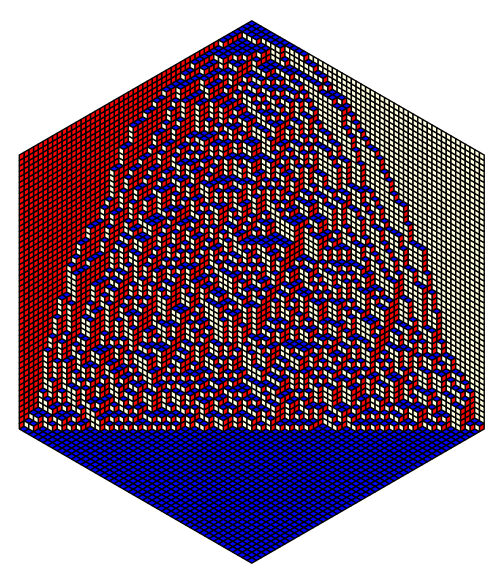}
\end{center}
\caption{Uniform and hook-weighted lozenge tilings of the hexagon $\Hex(50,50,50)$ slanted by a diagonal. }
\label{f:stair}
\end{figure}

Note that sometimes it is easier to analyze the limit shape for the lozenge tilings than for
the other skew shapes like staircases.  For examples, in Figure~\ref{f:stair} we show two lozenge tilings of
the hexagon $\Hex(50,50,50)$ slanted by a diagonal, one corresponding to the uniform
excited diagram of the staircase $\pi=\de_{150}/\de_{50}$, and another with the hook weights.
These random tilings are obtained by running a Metropolis algorithm for \ts $3\cdot 10^9$ \ts steps.
While the limit shapes have roughly similar outlines, in the
uniform case the limit shape curves are visibly tangent to the vertical sides of the hexagon,
and in the hook weighted case form an acute angle.

The observed behavior in the uniform case here is in line with the
cases of lozenge tilings of polygonal regions; for such regions the frozen boundary is an inscribed algebraic curve
as shown in e.g.~\cite{KO-Burgers, KOS}. However, the hexagon with slanted diagonal is not a
region which  has been treated with any of the classical methods even in the uniform case.
It would be interesting to
obtain the exact form of the limit shape in the hook weighted case in connection with our
detailed study of~$f^\pi$ in~\cite{MPP3,MPP4}; see~\cite{Rom} for some related results.

\vskip.6cm

\subsection*{Acknowledgements}
We are grateful to Sami Assaf, Dan Betea, Sara Billey, Valentin F\'eray, 
Vadim Gorin, Zach Hamaker, Tri Lai, Leonid Petrov, Dan Romik,
Luis Serrano, Richard Stanley, Hugh Thomas, Nathan Williams,
Damir Yeliussizov and Alex Yong for useful comments and help with the references,
and to Jane Austen~\cite{Aus} for the inspiration behind the first sentence.
We are very thankful to Eric Rains for showing us the connections to Racah
polynomials (see~$\S$\ref{subsec:Racah}), to Christian Krattenthaler for
pointing out to us the paper~\cite{KS} and graciously allowing us to publish his conjectural
formula~\eqref{eq:kratt}, and to Jang Soo Kim for pointing us out that
Corollary~\ref{cor:abcde-shape} appeared in~\cite{KO}. We thank the
anonymous referees for their careful reading, comments, and suggestions. 
The lozenge tilings in figures~\ref{f:lozenge-hexagon}
and~\ref{f:stair} were made using Sage and its algebraic combinatorics features
developed by the Sage-Combinat community~\cite{sage-combinat}.  Martin Tassy
generously helped us with the Metropolis simulations.  The first author was
partially supported by an AMS-Simons travel grant.  The second and third authors
were partially supported by the~NSF.
}

\end{document}